\documentclass[10pt,reqno]{amsart}

\usepackage{hyperref}
\usepackage{latexsym,amsmath, bm}
\usepackage{enumerate}
\usepackage{amsfonts}
\usepackage{amssymb}
\usepackage{geometry}
\usepackage{latexsym}
\usepackage{stmaryrd}

\usepackage[T1]{fontenc}
\usepackage{fourier}
\usepackage{bbm}
\usepackage{color}
\usepackage{fullpage}
\usepackage{tikz}
\usetikzlibrary{decorations.pathreplacing, matrix}

\renewcommand{\SS}{\mathbb S}

\newcommand{\fp}{\mathfrak p}

\newcommand{\vi}{\vec i}

\newcommand{\vs}{\vec s}
\newcommand{\vt}{\vec t}
\newcommand{\vz}{\vec z}
\newcommand{\vX}{\vec X}

\renewcommand{\epsilon}{\eps}

\newcommand\vI{\vec I}

\newcommand{\NN}{\mathbb N}

\newcommand\MU{\vec\mu}
\newcommand\vY{\vec Y}

\newcommand\nix{\,\cdot\,}

\newcommand\dd{{\mathrm d}}

\numberwithin{equation}{section}

\renewcommand{\vec}[1]{\boldsymbol{#1}}

\newcommand\KL[2]{D_{\mathrm{KL}}\bc{{{#1}\|{#2}}}}

\newcommand\SIGMA{\vec\sigma}

\newcommand\TAU{\vec\tau}

\newcommand\aco[1]{\textcolor{black}{#1}}
\newcommand\mhk[1]{\textcolor{black}{#1}}

\newtheorem{definition}{Definition}[section]

\newtheorem{example}[definition]{Example}
\newtheorem{remark}[definition]{Remark}
\newtheorem{theorem}[definition]{Theorem}
\newtheorem{lemma}[definition]{Lemma}
\newtheorem{proposition}[definition]{Proposition}
\newtheorem{corollary}[definition]{Corollary}

\newcommand\fd{\mathfrak{d}}

\newcommand\cA{\mathcal{A}}
\newcommand\cB{\mathcal{B}}

\newcommand\cD{\mathcal{D}}

\newcommand\cG{\mathcal{G}}
\newcommand\cE{\mathcal{E}}
\newcommand\cU{\mathcal{U}}
\newcommand\cN{\mathcal{N}}

\newcommand\cH{\mathcal{H}}
\newcommand\cS{\mathcal{S}}
\newcommand\cT{\mathcal{T}}
\newcommand\cI{\mathcal{I}}
\newcommand\cK{\mathcal{K}}

\newcommand\cL{\mathcal{L}}
\newcommand\cM{\mathcal{M}}

\newcommand\cP{\mathcal{P}}
\newcommand\cX{\mathcal{X}}

\def\cE{{\mathcal E}}

\newcommand\fX{\mathfrak{X}}
\newcommand\vx{\vec x}
\newcommand\vy{\vec y}
\newcommand\THETA{\vec\theta}

\newcommand{\beq}{\begin{equation}} \newcommand{\eeq}{\end{equation}}

\newcommand\eps{\varepsilon}

\newcommand\Erw{\mathbb{E}}
\newcommand{\vecone}{\vec{1}}

\newcommand\TV[1]{\left\|{#1}\right\|_{\mathrm{TV}}}

\newcommand\dTV{d_{\mathrm{TV}}}

\newcommand\bc[1]{\left({#1}\right)}
\newcommand\cbc[1]{\left\{{#1}\right\}}
\newcommand\bcfr[2]{\bc{\frac{#1}{#2}}}
\newcommand{\bck}[1]{\left\langle{#1}\right\rangle}
\newcommand\brk[1]{\left\lbrack{#1}\right\rbrack}
\newcommand\scal[2]{\bck{{#1},{#2}}}
\newcommand\norm[1]{\left\|{#1}\right\|}
\newcommand\abs[1]{\left|{#1}\right|}

\newcommand\RR{\mathbb{R}}
\newcommand\RRpos{\RR_{\geq0}}

\newcommand{\tensor}{\otimes}
\newcommand{\ind}{\mathbf{1}}

\newcommand{\Lovasz}{Lov\'asz}

\newcommand{\Szemeredi}{Szemer\'edi}

\newcommand\pr{\mathbb{P}} 
\renewcommand\Pr{\pr} 

\newcommand\Lem{Lemma}
\newcommand\Prop{Proposition}
\newcommand\Thm{Theorem}

\newcommand\Cor{Corollary}
\newcommand\Sec{Section}
\newcommand\Chap{Chapter}

\begin{document}

\title{The cut metric for probability distributions}

\author{Amin Coja-Oghlan, Max Hahn-Klimroth$^{*}$}
\thanks{Supported by Stiftung Polytechnische Gesellschaft Frankfurt am Main.}

\address{Amin Coja-Oghlan, {\tt acoghlan@math.uni-frankfurt.de}, Goethe University, Mathematics Institute, 10 Robert Mayer St, Frankfurt 60325, Germany.}

\address{Max Hahn-Klimroth, {\tt hahnklim@math.uni-frankfurt.de}, Goethe University, Mathematics Institute, 10 Robert Mayer St, Frankfurt 60325, Germany.}

\begin{abstract}
Guided by the theory of graph limits, we investigate a variant of the cut metric for limit objects of sequences of discrete probability distributions.
Apart from establishing basic results, we introduce a natural operation called {\em pinning} on the space of limit objects and show how this operation yields a canonical cut metric approximation to a given probability distribution akin to the weak regularity lemma for graphons.
We also establish the cut metric continuity of basic operations such as taking product measures.
\hfill{\em MSc: 60C05, 60B10}
\end{abstract}

\maketitle

\newcommand{\pin}{\downarrow}
\newcommand{\fK}{\mathfrak K}
\newcommand{\fU}{\mathfrak U}
\newcommand{\conf}{\cS}
\newcommand{\KernelR}{\cK_{\RR}}
\newcommand{\Kernel}{\cK}
\newcommand{\kernel}{\fK}
\newcommand\cutms{D_{\boxslash}} 
\newcommand\cutmgr{D_{\blacksquare}}
\newcommand\cutmFK{D_{\oblong}} 
\newcommand\cutm{D_{\boxtimes}} 
\newcommand{\Cutm}{\Delta_{\boxtimes}} 
\newcommand{\cutmw}{\Delta_{\boxslash}} 
\newcommand{\cutmW}{\cD_{\boxtimes}} 
\newcommand{\Law}{\cL}
\newcommand{\law}{\mathfrak L}
\newcommand{\exch}{\mathfrak X}
\newcommand{\Exch}{\exch}
\newcommand\cutnorm[1]{\norm{#1}_{\Box}}

\section{Introduction and results}\label{Sec_intro}

\subsection{Background and motivation}
The theory of graph limits clearly qualifies as one of the great recent success of modern combinatorics~\cite{BCLSV1,BCLSV2,Lovasz,BS1}.
Exhibiting a complete metric space of limit objects of sequences of finite graphs, the theory strikes a link between combinatorics and analysis.
In fact, the notion of graphon convergence unifies several combinatorially meaningful concepts, such as convergence of subgraph counts or with respect to the cut metric.
In effect, combinatorial ideas admit neat analytic interpretations.
For instance, the \Szemeredi{} regularity lemma yields the compactness of the graphon space~\cite{BS2}.

While sequences of graphs occur frequently in combinatorics (e.g., in the theory of random graphs), sequences of probability distributions on increasingly large discrete domains play no less prominent a role in the mathematical sciences.
For instance, such sequences are the bread and butter of mathematical physics.
A classical example is the Ising model on a $d$-dimensional integer lattice of side length $n$, a model of ferromagnetism.
The Ising model renders a probability measure, the so-called Boltzmann distribution, on the space 
$\{-1,+1\}^{[n]^d}$
that captures the distribution of the magnetic spins of the $n^d$ vertices.
The objective is to extract properties of this probability distribution in the limit of large $n$ such as the nature of correlations.
While mathematical physics has a purpose-built theory of limits of probability measures on lattices~\cite{Georgii},  this theory fails to cover other classes of important statistical mechanics models, such as mean-field models that `live' on random graphs~\cite{MPRTRLZ}.
Additionally, in statistics and data science sequences of discrete probability distributions arise naturally, e.g., as the empirical distributions of samples as more data are acquired.

\aco{The purpose of this paper is to show how the theory of graph limits can be adapted and extended to obtain a coherent theory of limits of probability distributions on discrete cubes.}
First cursory steps were already taken in an earlier contribution~\cite{AcoPerkinsSkubch}.
For instance, a probabilistic version of the cut metric was defined in that paper.
\aco{Moreover, Austin~\cite{Austin_ExchangeableRandomMeasures}, Diaconis and Janson~\cite{JansonDiaconis_GraphLimits} and (later) Panchenko~\cite{Panchenko_SherringtonKirkpatrick} pointed out the connection between the theory of graph limits and the Aldous-Hoover representation~\cite{Aldous,Hoover,Kallenberg}.}
But thus far a complete and concise disquisition has been lacking.
We therefore develop the basics of a cut-norm based limiting theory for probability measures, including the completeness and compactness of the space of limiting objects, a kernel representation, a sampling theorem and a discussion of the connection with exchangeable arrays.
Some of the proofs rely on arguments similar to the ones used in the theory of graph limits, and none of them will come as a gross surprise to experts.
In fact, a few statements (such as the compactness of the space of limiting objects) already appeared in~\cite{AcoPerkinsSkubch}, albeit without detailed proofs, and a few others (such as the characterisation of exchangeable arrays) are generalisations of results from~\cite{JansonDiaconis_GraphLimits}.
But here we present unified proofs of these basic results in full generality to provide a coherent and mostly self-contained treatment that, we hope, will facilitate applications.

\aco{Additionally, and this constitutes the main technical novelty of the paper, we present a new construction of regular partitions for limit objects of discrete probability distributions that constitutes a continuous generalisation of the {\em pinning operation} for discrete probability distributions from~\cite{AcoPinningPaper,Montanari,Raghavendra}.
The result provides an approximation akin to the graphon version of the Frieze-Kannan regularity lemma~\cite{FriezeKannan}.
The pinning operation merely involves a purely mechanical reweighting of the probability distribution.}
The `obliviousness' of the operation was critical to work on spin glass models on random graphs and on inference problems~\cite{AcoPinningPaper,AcoPerkins,AcoPerkinsSymmetry,AcoPerkinsSkubch}.
We show that a similarly oblivious procedure carries over naturally to the space of limit objects.
The proof, which hinges on a delicate analysis of cut norm approximations, constitutes the main technical achievement of the paper.

\subsection{Results}\label{Sec_results}
We proceed to set out the main concepts and to state the main results of the paper.
A detailed account of related work follows in \Sec~\ref{Sec_related}.
The cut metric is a mainstay of the theory of graph limits.
An adaptation for probability measures was suggested in~\cite{AcoPerkinsSymmetry,AcoPerkinsSkubch}.
Let us thus begin by recalling this construction.

\subsubsection{The cut metric}\label{Sec_cutm}
Let $\Omega\neq\emptyset$ be a finite set and let $n\geq1$ be an integer.
Further, for probability distributions $\mu,\nu$ on the discrete cube $\Omega^n$ let $\Gamma(\mu,\nu)$ be the set of all couplings of $\mu,\nu$, i.e., all probability distributions $\gamma$ on the product space $\Omega^n\times\Omega^n$ with marginal distributions $\mu,\nu$.
Additionally, let $\SS_n$ be the set of all permutations $[n]\to[n]$.
Following~\cite{AcoPerkinsSymmetry}, we define the {\em (weak) cut distance} of $\mu,\nu$ as
\begin{align}\label{eqdisc}
\Cutm(\mu,\nu)&=\inf_{\substack{\gamma\in\Gamma(\mu,\nu)\\\varphi\in\SS_n}}\sup_{\substack{S\subset\Omega^n\times\Omega^n\\X\subset[n]\\\omega\in\Omega}}\frac1n
\abs{\sum_{\substack{(\sigma,\tau)\in S\\x\in X}}\gamma(\sigma,\tau)\bc{\vecone\{\sigma_x=\omega\}-\vecone\{\tau_{\varphi(x)}=\omega\}}}.
\end{align}
The idea is that we first get to align $\mu,\nu$ as best as possible by choosing a suitable coupling $\gamma$ along with a permutation $\varphi$ of the $n$ coordinates.
Then an adversary comes along and points out the largest remaining  discrepancy.
Specifically, the adversary picks an event $S\subset \Omega^n\times\Omega^n$ under the coupling, a set $X\subset[n]$ of coordinates and an element $\omega\in\Omega$ and reads off the
discrepancy of the frequency of $\omega$ on $S,X$.
It is easily verified that \eqref{eqdisc} defines a pre-metric on the space $\Law_n=\Law_n(\Omega)$ of probability distribution on $\Omega^n$.
Thus, $\Cutm(\nix,\nix)$ is symmetric and satisfies the triangle inequality.
But distinct $\mu,\nu$ need not satisfy $\Cutm(\mu,\nu)>0$.
Hence, to obtain a metric space $\law_n=\law_n(\Omega)$ we identify any $\mu,\nu\in\Law_n$ with $\Cutm(\mu,\nu)=0$.

Following~\cite{AcoPerkinsSkubch}, we embed the spaces $\Law_n$ into a  joint space $\Law$.
Specifically, let $\cP(\Omega)$ be the space of all probability distributions on $\Omega$.
We identify $\cP(\Omega)$ with the standard simplex in $\RR^n$ and thus endow $\cP(\Omega)$ with the Euclidean topology and the corresponding Borel algebra.
Further, let $\conf$ be the space of all measurable maps $\sigma:[0,1]\to\cP(\Omega)$, $\sigma\mapsto\sigma_x$, up to equality (Lebesgue-)almost everywhere.
We equip $\conf$ with the $L_1$-metric
\begin{align*}
D_1(\sigma,\tau)&=\sum_{\omega\in\Omega}\int_0^1\abs{\sigma_x(\omega)-\tau_x(\omega)}\dd x
&&(\sigma,\tau\in\conf)
\end{align*}
and the corresponding Borel algebra, thus obtaining a complete, separable \aco{metric space.
The space $\cL$ is defined as the space of all probability measures on $\cS$.}

Much as in the discrete case, for probability distributions $\mu,\nu$ on $\conf$ we let $\Gamma(\mu,\nu)$ be the space of all couplings of $\mu,\nu$, i.e., probability distributions $\gamma$ on $\conf\times\conf$ with marginals $\mu,\nu$.
Moreover, let $\SS$ be the space of all measurable bijections $\varphi:[0,1]\to[0,1]$ such that both
$\varphi$ and its inverse $\varphi^{-1}$ 
map the Lebesgue measure to itself.\footnote{{We recall that on a standard Borel space the inverse map $\varphi^{-1}$ is measurable as well, see \Lem~\ref{Lemma_inverse}}.}
Then the {\em cut distance} of $\mu,\nu$ is defined by the expression
\begin{align}\label{eqcutm}
\cutm\bc{\mu,\nu}&=\inf_{\substack{\gamma\in\Gamma(\mu,\nu)\\\varphi\in\SS}}\sup_{\substack{S\subset\conf\times\conf\\X\subset[0,1]\\\omega\in\Omega}}
\abs{\int_S\int_X \bc{\sigma_x(\omega)-\tau_{\varphi(x)}(\omega)}\dd x\dd\gamma(\sigma,\tau) },
\end{align}
where, of course, $S,X$ range over measurable sets.
Thus, as in the discrete case we first align $\mu,\nu$ as best as possible by choosing a coupling and a suitable `permutation' $\varphi$.
Then the adversary puts their finger on the largest remaining discrepancy.
One easily verifies that \eqref{eqcutm} defines a pre-metric on $\Law$.
Thus, identifying any $\mu,\nu$ with $\cutm\bc{\mu,\nu}=0$, we obtain a metric space $\law$.
The points of this space we call $\Omega$-{\em laws}.

\begin{theorem}\label{Prop_Polish}
\aco{The metric space $\law$ is compact.}
\end{theorem}

\noindent
\Thm~\ref{Prop_Polish} was already stated in~\cite{AcoPerkinsSkubch}, but no detailed proof was included.
We will give a full proof based on a novel analytic argument in \Sec~\ref{Sec_fundamentals}.

What is the connection between the spaces $\law_n$ and the `limiting space' $\law$?
As pointed out in~\cite{AcoPerkinsSkubch}, a probability distribution $\mu$ on $\Omega^n$ naturally induces an $\Omega$-law.
Indeed, we represent each $\sigma\in\Omega^n$ by a step function $\dot\sigma:[0,1]\to\cP(\Omega)$ whose value on the interval $[(i-1)/n,i/n)$ is just the atom $\delta_{\sigma_i}\in\cP(\Omega)$ for each $i\in[n]$.
(This construction is somewhat similar to the one proposed for `decorated graphs' in~\cite{LovaszSzegedy_LimitsDecorated}.)
Then we let $\dot\mu\in\Law$ be the distribution of $\dot\sigma\in\conf$ for $\sigma$ chosen from $\mu$; in symbols,
\begin{align*}
\dot\mu&=\sum_{\sigma\in\Omega^n}\mu(\sigma)\delta_{\dot\sigma}\in\Law.
\end{align*}
Thus, we obtain a map $\Law_n\to\Law$, $\mu\mapsto\dot\mu$.
The definition of the cut metric guarantees that $\cutm(\dot\mu,\dot\nu)=0$ if $\Cutm(\mu,\nu)=0$.
Consequently, the map $\mu\mapsto\dot\mu$ induces a map $\law_n\to\law$.
The following statement shows that this map is in fact an embedding, and that therefore the space $\law$ unifies all the spaces $\law_n$, $n\geq1$.

\begin{theorem}\label{Prop_embedding}
There exists a function $\fd:[0,1]\to[0,1]$ with $\fd^{-1}(0)=\{0\}$ such that for all $n\geq1$ and all 
$\mu,\nu\in\law_n$ we have
$\fd(\Cutm(\mu,\nu))\leq\cutm(\dot\mu,\dot\nu)\leq\Cutm(\mu,\nu).$
\end{theorem}

\noindent
We will see a few examples of convergence in the cut metric momentarily.
But let us first explore a convenient representation of the space $\law$.

\aco{\begin{remark}
	The definition of the space $\law$ is based on measurable functions $\sigma:[0,1]\to\cP(\Omega)$.
	Of course, one could replace the unit interval by another atomless probability space, and this may be natural/convenient in some situations.
	(The definition of the product and direct sum in \Sec~\ref{Sec_contovp} below could be quoted as a case a point.)
	But the use of the unit interval is without loss of generality (see \Lem~\ref{Lemma_Lambda} below).
\end{remark}}

\subsubsection{The kernel representation}
As in the case of graph limits, $\Omega$-laws can naturally be represented by functions on the unit square that we call kernels.
To be precise, let $\Kernel{}$ be the set of all measurable maps $\kappa:[0,1]^2\to\cP(\Omega)$,
$(s,x)\mapsto\kappa_{s,x}$, up to equality almost everywhere.
For $\kappa,\kappa'\in\Kernel{}$ we define, with $S,X$ ranging over measurable sets,
\begin{align}\label{eqcutmkernel}
\cutm(\kappa,\kappa')&=\inf_{\varphi,\varphi'\in\SS}\sup_{\substack{S,X\subset[0,1]\\\omega\in\Omega}}\abs{\int_S\int_X\bc{\kappa_{s,x}(\omega)-\kappa'_{\varphi(s),\varphi'(x)}(\omega)}\dd x\dd s}.
\end{align}
As before \eqref{eqcutmkernel} defines a pre-metric on $\Kernel{}$.
We obtain a metric space $\kernel{}$ by identifying $\kappa,\kappa'\in\Kernel{}$ with $\cutm(\kappa,\kappa')=0$.

There is a natural map $\Kernel{}\to\Law{}$.
Namely, for a kernel $\kappa$ and $s\in[0,1]$ let $\kappa_s:[0,1]\to\cP(\Omega)$ be the measurable map $x\mapsto\kappa_{s,x}$.
This map belongs to the space $\conf$.
Thus, $\kappa$ induces a probability distribution $\mu^\kappa$ on $\conf$, namely the distribution of $\kappa_{\vs}$ for a uniformly random $\vs\in[0,1]$.
The definition of the cut distance guarantees that $\cutm(\mu^\kappa,\mu^{\kappa'})=0$ if $\cutm(\kappa,\kappa')=0$.
Therefore, 
as pointed out in \cite{AcoPerkinsSkubch},
the map $\kappa\mapsto\mu^\kappa$ induces a map $\kernel{}\to\law{}$.

\begin{theorem}\label{Prop_kernel}
The map $\kernel\to\law$ induced by $\kappa\mapsto \mu^\kappa$ is an isometric bijection.
\end{theorem}

\noindent
Thus, any $\Omega$-law $\mu$ can be represented by an $\Omega$-kernel, which we denote by $\kappa^\mu$.

\begin{example}\label{Ex_parity}
With $\Omega=\cbc{0,1}$ let $\mu^{(n)}\in\Law_n$ be uniformly distributed over all $\sigma \in \cbc{0,1}^n$ with even parity.
In symbols,
\begin{align*}
	\mu^{(n)}(\sigma)&=2^{1-n}\vecone\cbc{\sum_{i=1}^n\sigma_i\equiv0\mod2}.
\end{align*}
Similarly, let $\nu^{(n)}$ be uniformly distributed on the set of $\sigma \in \cbc{0,1}^n$ with odd parity.
Then $\mu^{(n)},\nu^{(n)}$ have total variation distance one for all $n$ because they are supported on disjoint subsets of $\{0,1\}^n$.
Nevertheless, in the cut distance both sequences $(\mu^{(n)})_n,(\nu^{(n)})_n$ converge to the common limit $\mu=\delta_u\in\Law$ supported on $u:[0,1]\to\cP(\{0,1\})$, $x\mapsto(1/2,1/2)$.
Specifically, we claim that
\begin{align}\label{eqEx_parity1}
	\Cutm(\mu^{(n)},\nu^{(n)})&=O(n^{-1}),&\cutm(\dot\mu^{(n)},\mu)&=O(n^{-1/2}).
\end{align}
To verify the first bound, consider the following coupling $\gamma^{(n)}$:
choose the first $n-1$ bits $\SIGMA_1,\ldots,\SIGMA_{n-1}\in\cbc{0,1}$ uniformly and independently
and choose $\SIGMA_n\in\cbc{0,1}$ so that $\sum_{i=1}^n\SIGMA_i\equiv0\mod2$.
Then $\gamma^{(n)}\in\cP(\Omega^n\times\Omega^n)$ is the distribution of
$((\SIGMA_1,\ldots,\SIGMA_n),(\SIGMA_1,\ldots,1-\SIGMA_n))$.
In effect, under $\gamma^{(n)}$ the two $n$-bit vectors differ in exactly one position, whence the first part of \eqref{eqEx_parity1} follows from \eqref{eqdisc}.
The second bound in \eqref{eqEx_parity1} follows from the central limit theorem.
\end{example}

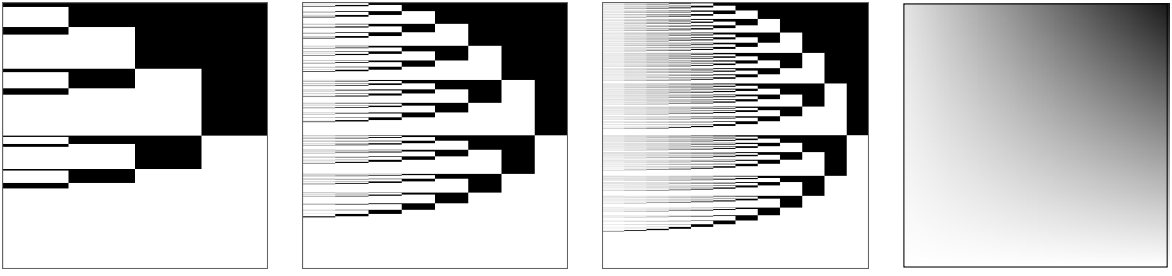
\begin{figure}
\centering
\begin{minipage}{.2\textwidth}
\includegraphics[width=100px]{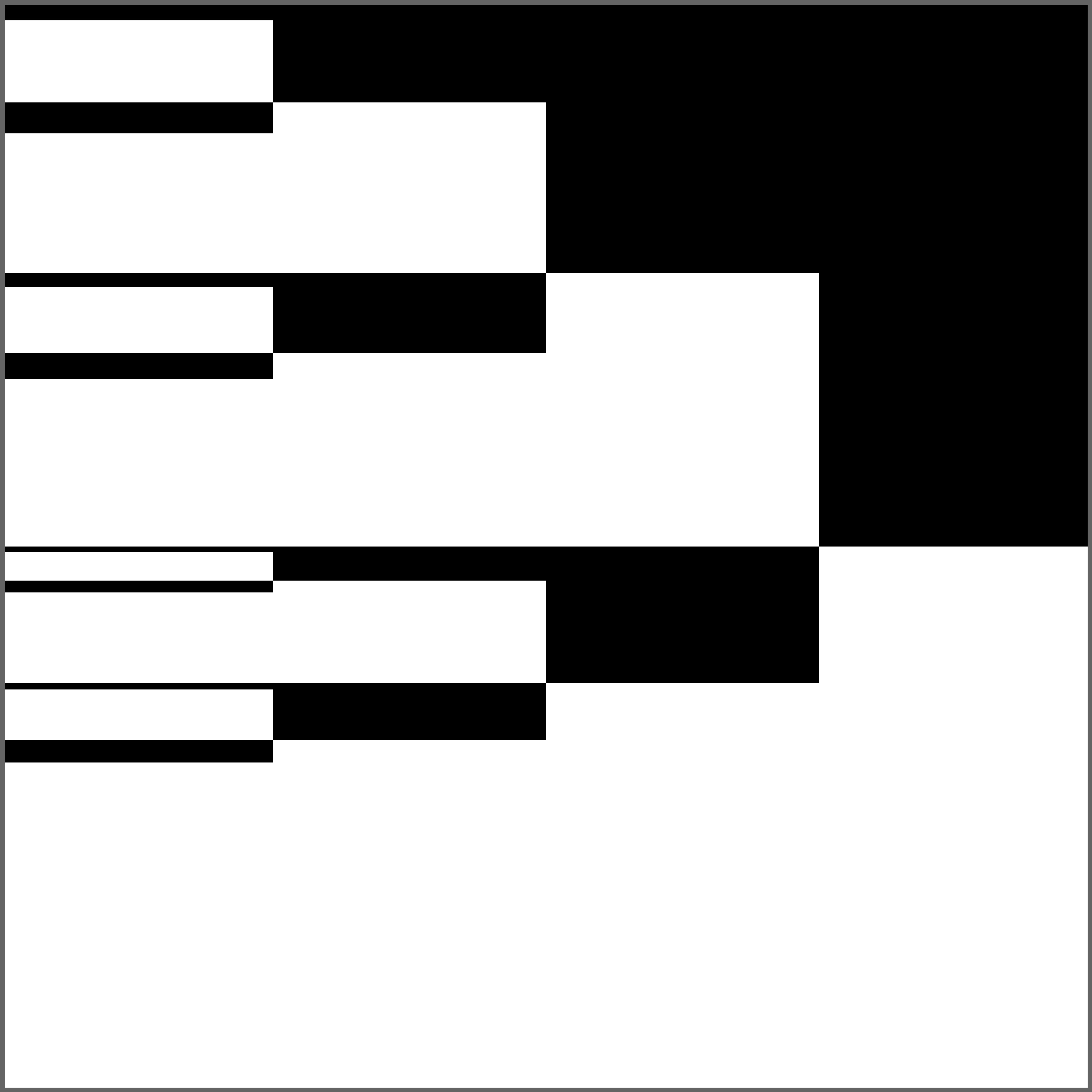}
\end{minipage}
\qquad
\begin{minipage}{.2\textwidth}
\includegraphics[width=100px]{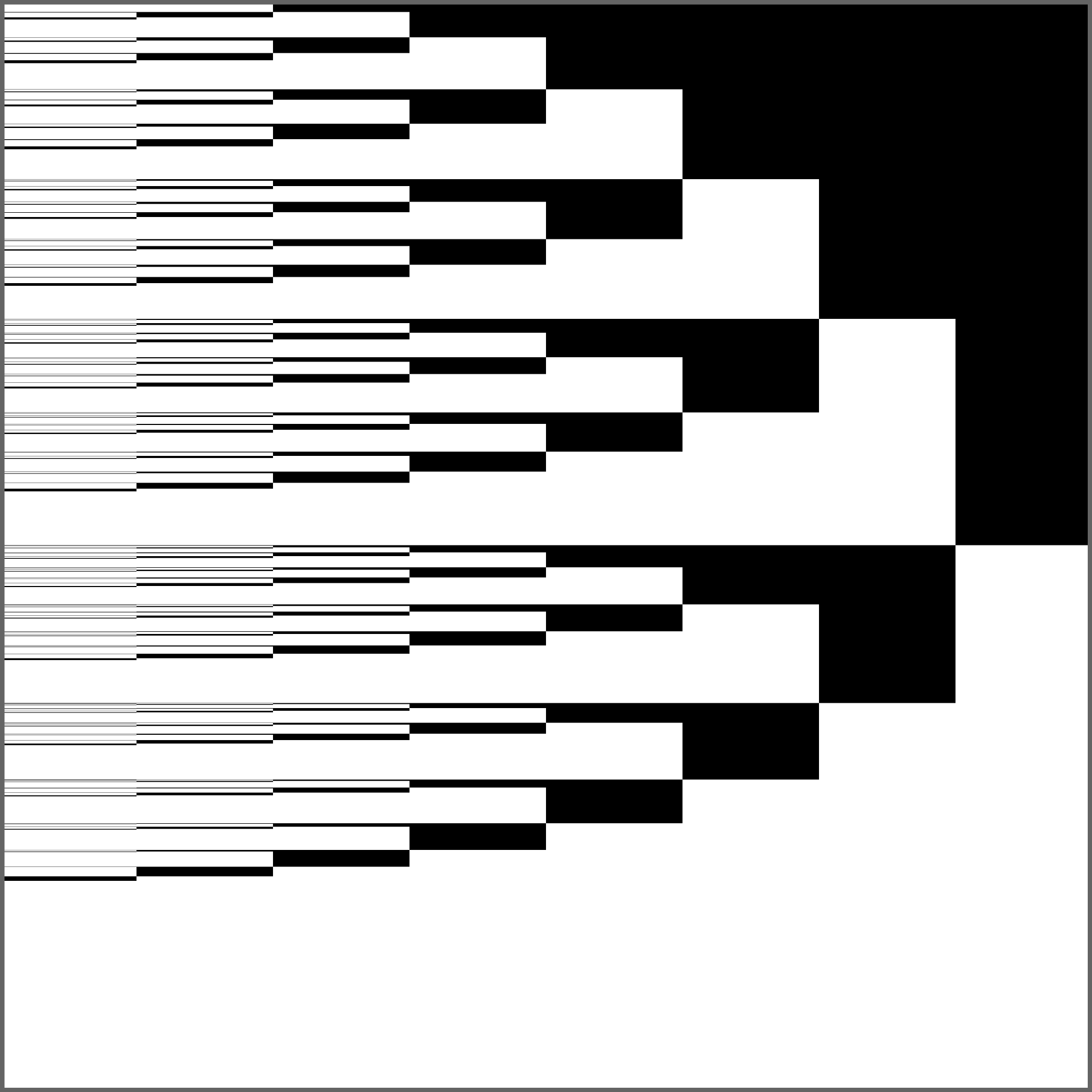}
\end{minipage}
\qquad
\begin{minipage}{.2\textwidth}
\includegraphics[width=100px]{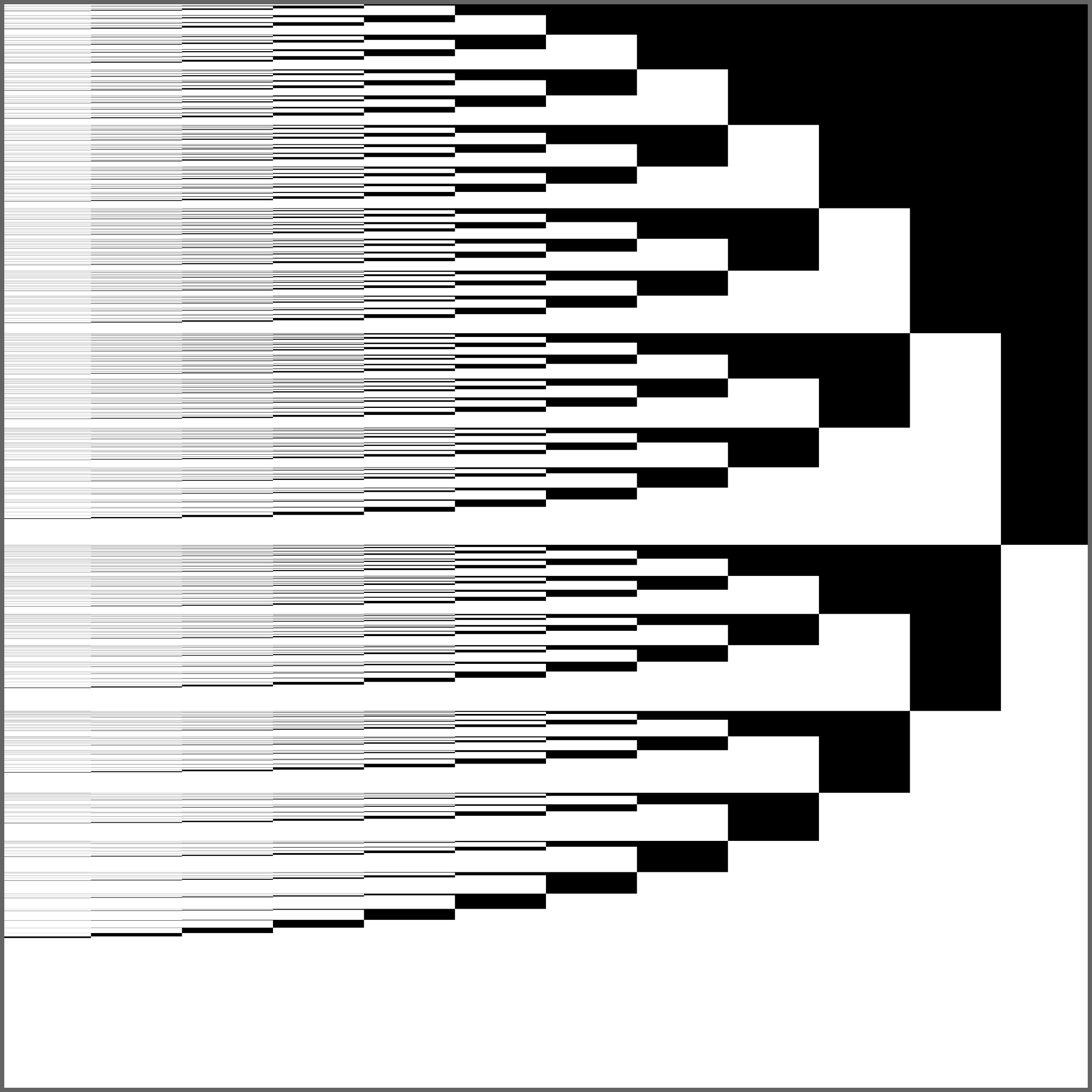}
\end{minipage}
\qquad
\begin{minipage}{.2\textwidth}
	\begin{tikzpicture}
	\pgfmathsetmacro\n{90}
	\foreach \x in {1,2,...,\n} 
	{
		\pgfmathsetmacro\colval{255 - \x * (255 / \n) }
		\pgfmathsetmacro\wid{3.5 / \n }
		\pgfmathsetmacro\stren{90 - \x * 80 / \n  }
		\definecolor{col}{RGB}{0, 0, 0}
		
		\shade[bottom color=white, top color=white!\stren!col] (\x*\wid -\wid,0) rectangle (\x*\wid + 0.03,3.5);
	}
	\draw[fill=none] (0,0) rectangle (3.5, 3.5);
	
	\end{tikzpicture}
\end{minipage}	
\caption{\mhk{The maps $(s,x)\in[0,1]^2\mapsto\kappa_{s,x}^{(n)}(1)$ for $n = 4, 8, 12$ and the limiting kernel $(s,x)\mapsto\kappa_{s,x}(1)$ from Example~\ref{Ex_cont}.}}
\label{Fig_ExampleConvergence}
\end{figure}

\begin{example}\label{Ex_cont}
Let $\mu^{(n)}$ be the probability distribution on $\{0,1\}^n$ induced by the following experiment.
First, pick $\vec{s} \in [0,1]$ uniformly at random.
Then, given $\vs$, obtain $\SIGMA\in \cbc{0,1}^n$ by letting $\SIGMA_i=1$ with probability $ i\vs/n$ independently for each $i\in[n]$.
In formulas, 
$$\mu^{(n)}(\sigma)=\int_0^1\prod_{i=1}^n\bcfr{is}{n}^{\sigma_i}\bc{1-\frac{is}{n}}^{1-\sigma_i}\dd s.$$
Kernel representations $\kappa^{(n)}$ of $\mu^{(n)}$ are displayed in Figure~\ref{Fig_ExampleConvergence} for some values of $n$.
The sequence $\kappa^{(n)}$ converges to the kernel $\kappa:[0,1]^2\to\cP(\cbc{0,1})$ defined by $\kappa_{s,x}(1)=sx$, $\kappa_{s,x}(0)=1-sx$.
\end{example}

\begin{example}\label{Ex_CW}
The Curie-Weiss model is an (extremely) simple model of ferromagnetism.
The vertices of a complete graph of order $n$ correspond to iron atoms that can take one of two possible magnetic spins $\pm1$.
Energetically it is beneficial for atoms to be aligned and the impact of the energetic term is governed by a temperature parameter $T>0$.
To be precise, the {\em Boltzmann distribution} $\mu^{(n)}$ on $\{\pm1\}^n$ defined by
\begin{align*}
	\mu_T^{(n)}(\sigma)&
	\propto\exp\bc{\frac Tn\sum_{1\leq i<j\leq n}\sigma_i\sigma_j}
\end{align*}
captures the distribution of spin configurations at a given temperature.
The Curie-Weiss model is completely understood mathematically and it is well known that a phase transition occurs at $T=1$.
In the framework of the cut distance, this phase transition manifests itself in the different limits that the sequence $(\mu_T^{(n)})_n$ converges to.
Specifically, the kernel $\kappa_T$ representing  the limit reads
\begin{align*}
	\kappa_T&:(s,x)\in[0,1]^2\mapsto(1/2,1/2)&\mbox{ for }T\leq1,\\
	\kappa_T&:(s,x)\in[0,1]^2\mapsto\begin{cases}
		((1+m_T)/2,(1-m_T)/2)&\mbox{ if }s\leq 1/2,\\
		((1-m_T)/2,(1+m_T)/2)&\mbox{ if }s>1/2
	\end{cases}&\mbox{ for }T>1,
\end{align*}
where $0<m_T<1$ is the unique zero of $m_T/T-\ln(1+m_T)/2+\ln(1-m_T)/2$ for $T>1$.
\end{example}

\subsubsection{Counting and sampling}\label{Sec_cands}
In the theory of graph limits convergence with respect to the cut metric is equivalent to convergence of subgraph counts.
We are going to derive a similar equivalence for $\Omega$-laws.
In fact, we are going to derive an extension of this result that links the cut metric to the theory of exchangeable arrays.
We recall that a probability distribution $\Xi$ on the space $\Omega^{\NN\times\NN}$ of infinite $\Omega$-valued arrays is {\em exchangeable} if the following is true.
If $\vX^{\Xi}=(\vX^{\Xi}(i,j))_{i,j\geq1}\in\Omega^{\NN\times\NN}$ is drawn randomly from $\Xi$, then
for any integer $n$ and for any permutations $\varphi,\psi:[n]\to[n]$ the random $n\times n$-arrays
\begin{align*}
(\vX^\Xi(i,j))_{i,j\in[n]}\qquad\mbox{and}\qquad(\vX^\Xi(\varphi(i),\psi(j)))_{i,j\in[n]}
\end{align*}
are identically distributed.
Let $\exch=\exch(\Omega)$ denote the set of all exchangeable distributions.
Since the product space $\Omega^{\NN\times\NN}$ is compact by Tychonoff's theorem, endowed with the weak topology $\exch$  is a compact, separable space.

A kernel $\kappa\in\kernel$ naturally induces an exchangeable distribution.
Specifically, let $\vs_1,\vx_1,\vs_2,\vx_2,\ldots\in[0,1]$ be mutually independent uniformly distributed random variables.
We obtain a random array $\vX^\kappa\in\Omega^{\NN\times\NN}$ by drawing independently for any $i,j\in\NN$ an element $\vX^\kappa(i,j)\in\Omega$ from the distribution $\kappa_{\vs_i,\vx_j}\in\cP(\Omega)$.
Clearly, the distribution $\Xi^\kappa$ of $\vX^\kappa$ is exchangeable.
By extension, a probability distribution $\pi$ on $\kernel$ induces an exchangeable distribution as well.
Indeed, with $\kappa^\pi\in\kernel$ drawn from $\pi$, we let 
$\Xi^\pi\in\exch$ be the distribution of the random array $\vX^\pi\in\Omega^{\NN\times\NN}$ obtained by first drawing $\kappa^\pi$ independently of the $(\vs_k,\vx_\ell)_{k,l\geq1}$ and then drawing each entry $\vX^\pi(i,j)$ from $\kappa^\pi_{\vs_i,\vx_j}$.
We equip the space $\cP(\kernel)$ of probability measures on $\kernel$ with the weak topology.

\begin{theorem}\label{Thm_exch}
The map $\cP(\kernel)\to\exch$, $\pi\mapsto\Xi^\pi$ is a homeomorphism.
\end{theorem}

\noindent {For the special case $\Omega=\cbc{0,1}$ \Thm~\ref{Thm_exch} is \mhk{the directed graph version of} ~\cite[Theorem 5.3]{JansonDiaconis_GraphLimits}.}

For $\mu\in\law$ let us write $\vX^\mu$ for the exchangeable array $\vX^{\kappa^\mu}$ induced by a kernel representation of $\mu$.
Suppose that $(\mu_N)_{N\geq1}$ is a sequence of $\Omega$-laws that converges to $\mu\in\law$. 
Then \Thm~\ref{Thm_exch} shows that for any $n\geq1$ and for any $\tau=(\tau_{i,j})_{i,j\in[n]}\in\Omega^{n\times n}$,
\begin{align}\label{eqThm_exch}
\lim_{N\to\infty}\pr\brk{\forall i,j\in[n]:\vX^{\mu_N}(i,j)=\tau_{i,j}}
&=\pr\brk{\forall i,j\in[n]:\vX^{\mu}(i,j)=\tau_{i,j}}.
\end{align}
Conversely, if $\mu_N,\mu\in\law$ are such that \eqref{eqThm_exch} holds for all $n,\tau$, then 
\Thm~\ref{Thm_exch} implies that $\lim_{N\to\infty}\cutm(\mu_N,\mu)=0$.
Thus, with $\Omega^{n\times n}$-matrices replacing subgraphs, \Thm~\ref{Thm_exch}, provides the probabilistic counterpart of the equivalence of subgraph counting and graphon convergence~{\cite[Theorem 11.5]{Lovasz}}.

Additionally, the theory of graph limits shows that a large enough random graph obtained from a graphon by sampling is close to the original graphon in the cut metric.
There is a corresponding statement in the realm of probability distributions as well.
Specifically, for an integer $n\geq1$ let $\mu_n\in\cP(\Omega^n)$ be the discrete probability distribution defined by
\begin{align*}
\mu_n(\sigma)&=\frac{1}{n}\sum_{i=1}^n\vecone\cbc{\forall j\in[n]:\vX^\mu(i,j)=\sigma_j}&&(\sigma\in\Omega^n).
\end{align*}
In words, $\mu_n$ is the empirical distribution of the rows of $(\vX^\mu(i,j))_{i,j\in[n]}$.
Strictly speaking, being dependent on the random coordinates $(\vs_i,\vx_j)_{i,j\geq1}$, $\mu_n$ is a {\em random} probability distribution on $\Omega^n$.
The following theorem supplies a probabilistic version of the sampling theorem for graphons~{\cite[Lemma 4.4]{BCLSV1}}.

\begin{theorem}\label{Thm_sampling}
There exists $c=c(\Omega)>0$ such that for all $n>1$ and all $\mu\in\law$ we have
$\Erw\brk{\cutm(\mu,\mu_{n})}\leq c/\sqrt{\log n}.$
\end{theorem}

The following theorem implies that the dependence on $n$ in \Thm~\ref{Thm_sampling} is best possible, apart from the value of the constant $c$.
\begin{theorem}\label{Thm_rl_lower}
There is a constant $c>0$ such that
for any $\eps>0$ there exists $\mu\in\law$ such that $\cutm(\mu,\nu)\geq\eps$ for all $\nu\in\Law$ whose support contains at most $\exp(c/\eps^2)$ configurations.
\end{theorem}

\subsubsection{Extremality}
Among all the probability measures on the discrete domain $\Omega^n$, the product measures are clearly the simplest.
We will therefore be particularly interested in distributions that are close to product measures in the cut metric.
To this end, for a probability measure $\mu$ on $\Omega^n$ we let
\begin{align*}
\bar\mu_i(\sigma)&=\sum_{\tau\in\Omega^n}\vecone\cbc{\tau_i=\sigma}\mu(\tau)&
\quad\mbox{ for }\sigma\in\Omega,\mbox{ and}\qquad&&
\bar\mu&=\bigotimes_{i=1}^n\bar\mu_i.
\end{align*}
Thus, $\bar\mu_i\in\cP(\Omega)$ is the marginal distribution of the $i$th coordinate under the measure $\mu$, and $\bar\mu$ is the product measure with the same marginals as $\mu$.
Then $\Cutm(\mu,\bar\mu)$ gauges how `similar' $\mu$ is to a product measure.
To be precise, since the cut metric is quite weak, a `small' value of $\Cutm(\mu,\bar\mu)$ need not imply that $\mu$ behaves like a product measure in every respect.
For instance, the entropy of $\mu$ might be much smaller than that of $\bar\mu$.
But if $\Cutm(\mu,\bar\mu)$ is small, then \eqref{eqThm_exch} implies that the joint distribution of a bounded number of randomly chosen coordinates of $\mu$ is typically close to a product measure in total variation distance.

A similar measure of proximity to a product distribution is meaningful on the space of $\Omega$-laws as well.
Formally, for $\mu\in\Law$ define $\bar\mu\in\Law$ as the atom concentrated on the single function
\begin{align}\label{eqbarmu}
[0,1]&\to\cP(\Omega),&x&\mapsto\int_{\conf}\sigma_x\dd\mu(\sigma).
\end{align}
Since  $\cutm(\bar\mu,\bar\nu)=0$ whenever $\cutm(\mu,\nu)=0$,  \eqref{eqbarmu} induces a map $\mu\in\law\mapsto\bar\mu\in\law$.
The laws $\bar\mu$ with $\mu\in\law$ represent the generalisation of discrete product measures.
Since each $\mu\in\law$ is represented by a distribution on $\conf$ that places all the probability mass on a single point, we call the laws $\bar\mu$ {\em extremal}.
{Moreover, $\mu\in\law$ is called {\em $\eps$-extremal} if $\cutm(\mu,\bar\mu)<\eps$.}
The following result summarises basic properties of extremal laws and of the map $\mu\mapsto\bar\mu$.

\begin{theorem}\label{Prop_metric_ext}
For all $\mu,\nu\in\Law$ we have
\begin{align}\label{eqProp_metric_ext1}
	\cutm(\bar\mu,\bar\nu)&\leq\cutm(\mu,\nu)&&\mbox{and}\\
	\cutm(\bar\mu,\bar\nu)&\leq\max_{\omega\in\Omega}\int_0^1\abs{\int_\conf\sigma_x\dd\mu(\sigma)-\int_\conf\sigma_x\dd\nu(\sigma)}\dd x\leq2\cutm(\bar\mu,\bar\nu).
	\label{eqProp_metric_ext2}
\end{align}
Furthermore, the set of extremal laws is a closed subset of $\law$.
\end{theorem}

\subsubsection{Pinning}
The regularity lemma constitutes one of the most powerful tools of modern combinatorics.
In a nutshell, the lemma shows that any graph can be approximated by a mixture of a bounded number of `simple' graphs, namely quasi-random bipartite graphs.
We will present a corresponding result for probability measures, respectively laws.
Specifically, we will show that any law can be approximated by a mixture of a small number of extremal laws.
Indeed, we will show that actually this approximation can be obtained by a simple, mechanical procedure called `pinning'.
This is in contrast to the proof of the graphon regularity lemma, where the regular partition results from a delicate construction that involves tracking a potential function.

To describe the pinning procedure, consider $\mu\in\Law$, $\theta\geq1$, $x_1,\ldots,x_\theta\in[0,1]$ and $\tau\in\Omega^\theta$.
Then we define
\begin{align*}
z_{\mu}(\tau,x_1,\ldots,x_\theta)&=\int_\conf\prod_{i=1}^\theta\sigma_{x_i}(\tau_i)\dd\mu(\sigma).
\end{align*}
Further, assuming that $z_{\mu}(\tau,x_1,\ldots,x_\theta)>0$, we define a reweighted probability distribution $\mu_{\tau\pin x_1,\ldots,x_\theta}$ by
\begin{align}\label{eqpin2}
\dd \mu_{\tau\pin x_1,\ldots,x_\theta}(\sigma)
&=\frac{1}{z_\mu(\tau,x_1,\ldots,x_\theta)}\prod_{i=1}^\theta\sigma_{x_i}(\tau_i)\dd\mu(\sigma);
\end{align}
Thus, $\mu_{\tau\pin x_1,\ldots,x_\theta}$ is obtained by reweighting $\mu$ according to the `reference configuration' $\tau$, evaluated at the coordinates $x_1,\ldots,x_\theta$.
For completeness we also let $\mu_{\tau\pin x_1,\ldots,x_\theta}=\mu$ if $z_{\mu}(\tau,x_1,\ldots,x_\theta)=0$.

The effect of this reweighting procedure becomes particularly interesting if the reference configuration and the coordinates are chosen randomly.
Specifically, let $\hat\vx_1,\hat\vx_2,\ldots\in[0,1]$ be uniform and mutually independent.
Further, for an integer $\theta\geq1$ draw $\hat\TAU=\hat\TAU^\mu\in\Omega^\theta$ from the distribution
\begin{align}\label{eqpin1}
\pr\brk{\hat\TAU^\mu=\tau\mid\hat\vx_1,\ldots,\hat\vx_\theta}&=\frac{\vz_\mu(\tau)}{\vz_\mu},\qquad\mbox{where}\quad
\vz_\mu(\tau)=\int_0^1\prod_{i=1}^\theta\sigma_{\hat\vx_i}(\tau_i)\dd\mu(\sigma),\qquad
\vz_\mu=\sum_{\tau\in\Omega^\theta}\vz_\mu(\tau).
\end{align}
Equivalently, and perhaps more intuitively, we can describe the choice of $\hat\TAU$ as follows.
First, draw $\TAU\in\conf$ from the distribution $\mu$; then pick $\hat\TAU$ from the product measure $\TAU_{\vx_1}\tensor\cdots\tensor\TAU_{\vx_\theta}\in\cP(\Omega^\theta)$.
Now, having drawn the `reference vector' $\hat\TAU$, we obtain the reweighted distribution 
$\mu_{\hat\TAU\pin\theta}=\mu_{\hat\TAU\pin\hat\vx_1,\ldots,\hat\vx_\theta}$ as defined in \eqref{eqpin2}.
Clearly, \eqref{eqpin1} guarantees that $\vz_\mu(\hat\TAU)>0$ almost surely.
Finally, we define
\begin{align*}
\mu_{\pin\theta}&=\Erw\brk{\overline{\mu_{\hat\TAU\pin\theta}}\mid\vx_1,\ldots,\vx_\theta}\in\Law.
\end{align*}
Hence, $\mu_{\pin\theta}$ weights each possible outcome according to the probability of its reference configuration $\hat\TAU$.
The discrete version of the operation $\mu\mapsto\mu_{\pin\theta}$ for $\mu\in\Law_n$ was introduced in \cite{AcoPerkinsSymmetry}.
Following the terminology from that paper, we refer to the map $\mu\mapsto\mu_{\hat\TAU\pin\theta}$ as the {\em pinning operation}.
The term is explained by the fact that in the discrete case, each of the products on the r.h.s.\ of \eqref{eqpin2} is either one or zero.

The next theorem shows that pinning furnishes a probabilistic equivalent of weak regular graphon partitions.
To state this result, we observe that the pinning construction is well-defined on the space $\law$ as well.
To be precise, if $\mu,\nu\in\law$ have cut distance zero, then $\mu_{\hat\SIGMA\pin\theta},\nu_{\hat\SIGMA\pin\theta}$ are identically distributed, and so are $\mu_{\pin\theta}$ and $\nu_{\pin\theta}$.
Consequently, we can apply the pinning operation directly to elements of the space $\law$.

\begin{theorem}\label{Thm_pinning}
Let $0<\eps<1$, let $\mu\in\law$ and
draw $0\leq\THETA=\THETA(\eps)\leq 64\eps^{-8}\log|\Omega|$ uniformly and independently of everything else.
Then $\pr\brk{\mu_{\hat\TAU\pin\THETA}\mbox{ is $\eps$-extremal}}\geq1-\eps$ and
$\Erw[\cutm(\mu,\mu_{\pin\THETA})]<\eps$.
\end{theorem}

\noindent
Hence, the law $\mu_{\pin\THETA}$, a mixture of no more than $|\Omega|^{\THETA}$ extremal laws, likely provides an $\eps$-approximation to $\mu$.

\subsubsection{Continuity and overlaps}\label{Sec_contovp}
There are certain natural operations on probability measures and, by extension, laws that turn out to be continuous with respect to the cut metric.
First, we consider the construction of the product measure.
For discrete measures $\mu,\nu\in\Law_n(\Omega)$ we can view their product $\mu\tensor\nu$ as a probability distribution on $(\Omega\times\Omega)^n$ such that for any $\sigma_1,\tau_1,\ldots,\sigma_n,\tau_n\in\Omega$,
\begin{align*}
\mu\tensor\nu\bc{\binom{\sigma_1}{\tau_1},\ldots,\binom{\sigma_n}{\tau_n}}
=\mu(\sigma_1,\ldots,\sigma_n)\nu(\tau_1,\ldots,\tau_n).
\end{align*}
We extend this construction to laws by way of the kernel representation.
To this end, let $\Lambda:[0,1]\to[0,1]\times[0,1]$, $x\mapsto(\Lambda_1(x),\Lambda_2(x))$ be a measurable bijection 
that maps the Lebesgue measure on $[0,1]$ to the Lebesgue measure on $[0,1]^2$ such that, conversely, $\Lambda^{-1}$ maps the Lebesgue measure on $[0,1]^2$ to the Lebesgue measure on $[0,1]$.\footnote{The existence of such a $\Lambda$ follows from \Lem~\ref{Lemma_Lambda} below.}
Following~\cite{AcoPerkinsSkubch}, for measurable maps $\kappa,\kappa':[0,1]^2\to\cP(\Omega)$ we introduce
\begin{align*}
\kappa\tensor{}\kappa'&:[0,1]^2\to\cP(\Omega^2),&
(s,x)\in[0,1]\times[0,1]\mapsto\kappa_{\Lambda_1(s),x}\tensor\kappa'_{\Lambda_2(s),x}\in\cP(\Omega^2).
\end{align*}
For any kernels $\kappa,\kappa',\kappa'',\kappa'''$ such that $\cutm(\kappa,\kappa'')=\cutm(\kappa',\kappa''')=0$ we clearly have $\cutm(\kappa\tensor{}\kappa',\kappa''\tensor{}\kappa''')=0$.
Thus, the $\tensor{}$-operation is well defined on the kernel space $\kernel{}$.
Hence, due to \Thm~\ref{Prop_kernel} the construction extends to laws, i.e., given $\Omega$-laws $\mu,\nu$ we obtain an $\Omega^2$-law $\mu\tensor{}\nu$.
Furthermore, it is easy to see that for any $\mu,\nu\in\Law_n(\Omega)$ the $\Omega^2$-law representing the product measure $\mu\tensor{}\nu$ is precisely the $\tensor{}$-product of the laws $\dot\mu,\dot\nu$ representing $\mu,\nu$.

\begin{theorem}\label{Thm_tensor}
The map $(\mu,\nu)\in\law(\Omega)\mapsto\mu\tensor\nu\in\law(\Omega^2)$ is continuous.
\end{theorem}

There is a second fundamental operation on distributions/laws that resembles the operation of obtaining a $n\times n$-rank one matrix from two vectors of length $n$.
Specifically, for vectors $\sigma,\tau\in\Omega^{[n]}$ let $\sigma\oplus\tau\in(\Omega^2)^{[n]\times[n]}$ be the vector with entries $(\sigma\oplus\tau)_{ij}=(\sigma_i,\tau_j)$ for all $i,j\in[n]$.
Additionally, for distributions $\mu,\nu\in\Law_n(\Omega)$ let $\mu\oplus\nu$ be the distribution of the pair $\SIGMA^\mu\oplus\TAU^\nu$ with $\SIGMA^\mu,\TAU^\nu\in\Omega^n$ chosen from $\mu,\nu$, respectively.

We extend the $\oplus$-operation to kernels as follows.
For $\kappa,\kappa':[0,1]^2\to\cP(\Omega)$ let
\begin{align*}
\kappa\oplus\kappa'&:[0,1]^2\to\cP(\Omega^2),&
(s,x)\mapsto\kappa_{s,\Lambda_1(x)}\tensor\kappa_{s,\Lambda_2(x)}.
\end{align*}
It is easy to see that for $\kappa,\kappa',\kappa'',\kappa'''$ with
$\cutm(\kappa,\kappa'')=\cutm(\kappa',\kappa''')=0$ we have
$\cutm(\kappa\oplus\kappa',\kappa''\oplus\kappa''')=0$.
Hence, the $\oplus$-operation is well-defined on the space $\kernel{}$ and thus, due to \Thm~\ref{Prop_kernel}, on the space $\law{}$ as well.
Moreover, for discrete measures $\mu,\nu\in\cP(\Omega^n)$ one verifies immediately that the law representing $\mu\oplus\nu$ coincides with $\dot\mu\oplus\dot\nu$.

\begin{theorem}\label{Thm_oplus}
The map $\law(\Omega)\to\law(\Omega^2)$, $(\mu,\nu)\mapsto\mu\oplus\nu$ is continuous.
\end{theorem}

\Thm s~\ref{Thm_tensor} and \ref{Thm_oplus} immediately imply the continuity of further functionals that play a fundamental role in mathematical physics.
Specifically, let $\sigma_1,\ldots,\sigma_n\in\conf$.
For $\sigma_1,\ldots,\sigma_n\in\conf$ and $\omega_1,\ldots,\omega_n\in\Omega$ we define
\begin{align*}
R_{\omega_1,\ldots,\omega_n}\bc{\sigma_1,\ldots,\sigma_n}&=\int_0^1\prod_{i=1}^n\sigma_{i,x}(\omega_i)\dd x.
\end{align*}
Furthermore, for $\mu\in\Law$ and $\ell\geq1$ we define
\begin{align*}
R_{\ell,\omega_1,\ldots,\omega_n}(\mu)&
=\int_{\conf}\cdots\int_\conf R_{\omega_1,\ldots,\omega_n}\bc{\sigma_1,\ldots,\sigma_n}^\ell
\dd\mu(\sigma_1)\cdots\dd\mu(\sigma_n).
\end{align*}
Additionally, let $R_{\ell,n}(\mu)=(R_{\ell,\omega_1,\ldots,\omega_n}(\mu))_{\omega_1,\ldots,\omega_n\in\Omega}$.
In physics jargon, the arrays $R_{\ell,n}(\mu)$ are known as {\em multi-overlaps} of $\mu$.
Since $R_{\ell,n}(\mu)=R_{\ell,n}(\nu)$ if $\cutm(\mu,\nu)=0$, the multi-overlaps are well-defined on the space $\law$ of laws.

\begin{corollary}\label{Thm_overlap}
The functions $\mu\in\law\mapsto R_{\ell,n}(\mu)$ with $\ell,n\geq1$ are continuous.
\end{corollary}

\subsection{Discussion and related work}\label{Sec_related}
Borgs, Chayes, \Lovasz, S\'os, Szegedy and Vesztergombi launched the theory of (dense) graph limits in a series of important and influential articles~\cite{BCLSV1,BCLSV2,LovaszSzegedy_LimitsDecorated,BS1,BS2}.
\Lovasz~\cite{Lovasz} provides a unified account of the state of the art up to about 2012.
Moreover, Janson~\cite{Janson_CutMetric} gives an excellent account of the measure-theoretic foundations of the theory of graph limits and some of its generalisation.

Given the many areas of application where sequences of probability measures on increasingly large discrete cubes appear, the most prominent example being perhaps the study of Boltzmann distributions in mathematical physics, it is unsurprising that attempts have been made to construct limiting objects for such sequences.
The theory of Gibbs measures embodies the classical, physics-inspired approach to this task~\cite{Georgii}.
Here the aim is to construct and classify all possible `infinite-volume' limits of Boltzmann distributions defined on spatial structures such as trees or lattices.
The limiting objects are called {\em Gibbs measures}.
A fundamental question, whose ramifications extend from the study of phase transitions in physics to the computational complexity of counting and sampling, is whether there is a unique Gibbs measure that satisfies all the finite-volume conditional equations (e.g, ~\cite{Andreas,SlyUniqueness,SS}).
However, since the theory of Gibbs measures is confined to systems with an underlying lattice-like geometry, numerous applications are beyond its reach.
For instance, Marinari et al.~\cite{MPRTRLZ} argued that the classical theory of Gibbs measures does not provide an {appropriate} framework for the study of (diluted) mean-field models such as the Sherrington-Kirkpatrick model, the Viana-Bray model or the hardcore model on a sparse random graph.
Further examples of `non-spatial' sequences of distributions abound in computer science, statistics and data science.

Panchenko~\cite{Panchenko_SherringtonKirkpatrick,Panchenko} employed the more abstract Aldous-Hoover representation of exhangeable arrays in his work on mean-field models~\cite{Aldous,Hoover}.
Kallenberg's monograph~\cite{Kallenberg} provides the definite treatment of this abstract theory.
Furthermore, Austin~\cite{Austin_ExchangeableRandomMeasures} extends and generalises the concept of exchangeable arrays and discusses applications to the Viana-Bray spin glass model.
The close relationship between the theory of graph limits and exchangeable arrays was first noticed by Diaconis and Janson~\cite{JansonDiaconis_GraphLimits}.
Their~\cite[\Thm~9.1]{JansonDiaconis_GraphLimits} is essentially \mhk{a directed graph version of}  \Thm~\ref{Thm_exch} \mhk{in the special case $\Omega = \cbc{0,1}$}.
Moreover, the appendix of Panchenko's monograph~\cite{Panchenko_SherringtonKirkpatrick} also contains a proof of the Aldous-Hoover representation theorem via graph limits.

Although the connection between genuinely probabilistic constructions such as the Aldous-Hoover representation and graph limits was noticed in prior \aco{work~\cite{Austin_ExchangeableRandomMeasures,JansonDiaconis_GraphLimits,Panchenko_SherringtonKirkpatrick}, those contributions stopped short of working out a fully-fledged adaptation of the theory of graph limits to a limit theory for probability measures on discrete cubes.}
A prior article by Coja-Oghlan, Perkins and Skubch~\cite{AcoPerkinsSkubch} made a first cursory attempt at filling this gap and already contained the definition~\eqref{eqcutm} of the cut metric and of the space $\law$ of laws.
Additionally, the compactness of the space $\law$ (\Thm~\ref{Prop_Polish})  and a weaker version of the kernel representation (\Thm~\ref{Prop_kernel}) were stated in~\cite{AcoPerkinsSkubch}, although no detailed proofs were given.
Furthermore, a definition similar to the discrete cut metric~\eqref{eqdisc} was devised in~\cite{AcoPerkinsSymmetry} and a statement similar to \Thm~\ref{Thm_oplus} was previously proved by Coja-Oghlan and Perkins~\cite[\Prop~A.2]{AcoPerkins}.
Finally, versions of the pinning operation for discrete probability measures appeared in~\cite{AcoPinningPaper,Montanari,Raghavendra} and recently Eldad~\cite{Eldad} devised an extension to subspaces of $\RR^n$, i.e., to the case of spins that need not take discrete values.

The contribution of the present paper is that we expressly and explicitly adapt and extend the concepts of the theory of graph limits to the context of probability distributions on increasing \aco{sequences of discrete cubes.}
We present in a unified way the proofs of the most important basic facts such as the relationship between the discrete and the continuous cut metric (\Thm~\ref{Prop_embedding}), the kernel representation (\Thm~\ref{Prop_kernel}), the sampling theorem (\Thm~\ref{Thm_sampling}) and the continuity of product measures (\Thm s~\ref{Thm_tensor} and~\ref{Thm_oplus}).
The proofs of these results are based on extensions and adaptations of techniques from the theory of graph limits.
Moreover, we present a self-contained derivation of the representation theorem for exchangeable arrays (\Thm~\ref{Thm_exch}).
The added value by comparison to prior work~\cite{AcoPerkinsSkubch,JansonDiaconis_GraphLimits} is that here we present detailed, unified proofs that operate directly in the probabilistic setting, rather than by extensive allusion to the graphon space.
Additionally, we present a self-contained proof of the compactness result (\Thm~\ref{Prop_Polish}).
While the argument set out in, e.g., \cite[\Chap~9]{Lovasz} could be adapted to the probabilistic setting, we present a different argument based on analytic techniques that might be of independent interest.
But the main technical novelty is certainly the pinning theorem (\Thm~\ref{Thm_pinning}) that generalises the discrete version from~\cite{AcoPinningPaper}.
The proof is delicate and uses many of the other, more basic results.

The pinning operation from \Thm~\ref{Thm_pinning} is somewhat reminiscent of Tao's construction of regular partitions~\cite{TaoRegular} \aco{and of the construction of \Lovasz\ and Szegedy~\cite{BS2}}.
For example, Tao's construciton of a regular partition is based on sampling a number $\theta$ of vertices of a graph $G$ and then partitioning the remaining vertices into $2^\theta$ classes according to their adjacencies with the reference vertices.
The discrete version pinning operation from~\cite{AcoPinningPaper,Raghavendra} proceeds similarly; see \Thm~\ref{Thm_pin} below, except that the number of pinned coordinates $\THETA$ is chosen randomly, rather than deligently given $G$.
The same is true of the number of pinned coordinates in \Thm~\ref{Thm_pinning}, which additionally yields a continuous version applicable to general $\Omega$-laws.

\aco{
Finally, there have been several further related contributions that extend the classical (dense) theory of graph limits as set out in~\cite{Lovasz}.
Just as the classical theory, these extensions partly have a probabilistic component as they incorporate random graphs.
For example, the important $L^p$-theory of sparse graph convergence covers limit objects of exchangeable sparse graphs~\cite{borgs_lp}. 
A further contribution pertinent to sparse random graphs is the work of Crane and Dempsey~\cite{crane_edge_exchange} and Cai, Campbell and Broderick~\cite{cai_edge_exchange} on edge-exchangeability.
Moreover, the articles~\cite{borgs_graphexes,veitch_graphexes} deal with graphexes, which are limit objects of random geometric graphs.
Further important extensions of the theory of graph limits include the work of Ne\v{s}et\v{r}il and Ossona de Mendez~\cite{nesetril_modelings} on convergence of sparse graphs that satisfy first order formulas,
the paper of Hoppen, Kohayakawa, Moreira,  Rath and Sampaio~\cite{hoppen_permutons} on sequences of permutations (permutons), \mhk{the article by Coregliano and Razborov \cite{Razborov_DenseObjects} on limits on dense combinatorial objects},
and Janson's work \cite{Janson_posetons} on limits of posets. 
Some of these contributions, as well as Austin's work~\cite{Austin_exchange, Austin_ExchangeableRandomMeasures} involve stronger versions of exchangeability than the classical de Finetti or Aldous-Hoover notions of exchangeability.
}

\subsection{Outline}
After presenting the necessary background and notation in \Sec~\ref{Sec_pre}, in \Sec~\ref{Sec_fundamentals} we will prove the basic facts about laws and the cut metric stated above.
Specifically, \Sec~\ref{Sec_fundamentals} contains the proofs of \Thm s~\ref{Prop_Polish},
\ref{Prop_kernel}, \ref{Thm_exch}, \ref{Thm_sampling}, \ref{Prop_metric_ext}, \ref{Thm_tensor} and \ref{Thm_oplus}.
Subsequently, in \Sec~\ref{Sec_Extremality} we prove \Thm~\ref{Thm_pinning}, which constitutes the main technical contribution of the paper.
Finally, in \Sec~\ref{Sec_Prop_embedding} we establish \Thm~\ref{Prop_embedding}.

\section{Preliminaries}\label{Sec_pre}

\subsection{Measure theory}\label{Sec_MeasureTheory}
Throughout the paper we continue to denote by $\lambda$ the Lebesgue measure on the Euclidean space $\RR^k$; the reference to $k$ will always be clear from the context.
For the convenience of the reader we collect a few basic facts from measure theory that we will need.
The first lemma follows from the Isomorphism Theorem, see e.g. \cite[Sec. 15.B]{Kechris}. 

\begin{lemma}\label{Lemma_standard}
Suppose that $\cE=(X,\cA,\mu)$ is a standard Borel space equipped with a probability measure $\mu$.
Then there exists a measurable map $f:[0,1]\to X$ that maps the Lebesgue measure to $\mu$.
\end{lemma}

\begin{lemma}[Theorem 3.2 of \cite{Mackey}]\label{Lemma_inverse}
Suppose that $\cE=(X,\cA)$, $\cE'=(X',\cA')$ are standard Borel spaces and that $f:X\to X'$ is a measurable bijection.
Then its inverse $f^{-1}$ is measurable.
\end{lemma}

\aco{\begin{lemma}[Theorem A.7 of \cite{Janson_CutMetric}]\label{Lemma_Lambda}
If $(\cX, \mu)$ is an atomless Borel probability space and $\lambda$ is the Lebesgue measure, then there is a measure preserving bijection of $(\cX, \mu)$ to $([0,1], \lambda)$.
\end{lemma}}

\noindent
The following is the Riesz-Markov-Kakutani representation theorem \cite{Hartig}.

\begin{lemma}\label{Lemma_Riesz}
\aco{Suppose that $\cE_0$ is a compact metric space and that $\varphi:C(\cE_0)\to\RR$ is a positive linear functional on the space of continuous functions $C(\cE_0)$ on $\cE_0$.
Moreover, assume that $\varphi(\vecone)=1$.
Then there exists a unique probability measure $\mu$ on $\cE_0$ such that $\varphi(f)=\int_{\cE_0} f\dd\mu$ for all $f\in C(\cE_0)$.}
\end{lemma}

\noindent
\aco{We will need \Lem~\ref{Lemma_Riesz} in \Sec~\ref{Sec_Prop_Polish} to prove the completeness of the space of laws with respect to the cut metric.}

\aco{Additionally, in several places throughout the paper we will need the following metric on probability measures.}
Suppose that $(\cE,D)$ is a complete separable metric space \aco{and that $D$ is bounded}.
Then the space $\cP(\cE)$ of probability measures on $\cE$ equipped with the {\em Wasserstein metric}
\begin{align}\label{eqWasserstein}
\cD(\mu,\nu)&=\inf\cbc{\int_{\cE\times\cE}D(x,y)\dd\gamma(x,y):\gamma\in\Gamma(\mu,\nu)},
\end{align}
where we recall that $\Gamma(\mu,\nu)$ is the set of all couplings of $\mu,\nu$, also is complete and separable.
The Wasserstein metric induces the weak topology on $\cP(\cE)$~{\cite[\Thm~ 6.9]{Villani}}.
The definition \eqref{eqWasserstein} extends to $\cE$-valued random variables $\vX,\vY$, for which we define 
\begin{align*}
\cD(\vX,\vY)&=\inf\cbc{\int_{\cE\times\cE}D(x,y)\dd\gamma(x,y):\gamma\in\Gamma(\vX,\vY)},
\end{align*}
with $\Gamma(\vX,\vY)$ denoting the set of all couplings of $\vX,\vY$.
We will frequently be working with the Wasserstein metric $\cutmW(\nix,\nix)$ induced by the cut metric on $\law$ or $\kernel{}$.

\subsection{Variations on the cut metric}
When we defined the cut metric $\cutm(\mu,\nu)$ in \eqref{eqcutm} we allowed for a coupling of $\mu,\nu$ as well as a `coordinate permutation' $\varphi\in\SS$.
Sometimes the latter is not desirable.
Therefore, for $\mu,\nu\in\Law$ we define the {\em strong cut distance} as
\begin{align}\label{eqNoPerm}
\cutms(\mu,\nu)&=\inf_{\substack{\gamma\in\Gamma(\mu,\nu)}}
\sup_{\substack{S\subset\conf\times\conf\\X\subset[0,1]\\\omega\in\Omega}}
\abs{\int_S\int_X \bc{\sigma_x(\omega)-\tau_{x}(\omega)}\dd x\dd\gamma(\sigma,\tau)}
\end{align}
with $S,X$ ranging over measurable sets.
It is easily verified that $\cutms(\nix,\nix)$ is a pre-metric on $\Law$.
Analogously, for  $\mu, \nu \in \cP \left( \Omega^n \right)$ let
\begin{align}\label{eqdisc_strong}
\cutmw(\mu,\nu)&=\inf_{\gamma\in\Gamma(\mu,\nu)}\sup_{\substack{S\subset\Omega^n\times\Omega^n\\X\subset[n]\\\omega\in\Omega}}\frac1n
\abs{\sum_{\substack{(\sigma,\tau)\in S\\x\in X}}\gamma(\sigma,\tau)\bc{\vecone\{\sigma_x=\omega\}-\vecone\{\tau_{x}=\omega\}}}.
\end{align}
Similarly, we will be led to consider several variants of the kernel cut metric from \eqref{eqcutmkernel}.
Specifically, let $\KernelR=\KernelR(\Omega)$ be the set of all maps $\kappa,\kappa':[0,1]^2\to\RR^\Omega$ such that the functions $(s,x)\in[0,1]^2\mapsto\kappa_{s,x}(\omega)$ belong to $L^1([0,1]^2,\RR)$ for all $\omega\in\Omega$, up to equality almost everywhere.
Then for $\kappa,\kappa'\in\KernelR$ we define
\begin{align*}
\cutm(\kappa,\kappa')&=\inf_{\substack{\varphi,\psi\in\SS}}
\sup_{\substack{S,X\subset[0,1]\\\omega\in\Omega}}
\abs{\int_S\int_X \bc{\kappa_{s,x}(\omega)- \kappa'_{\varphi(s),\psi(x)}(\omega)}\dd x\dd s},\\
\cutms(\kappa,\kappa')&=\inf_{\substack{\varphi\in\SS}}
\sup_{\substack{S,X\subset[0,1]\\\omega\in\Omega}}
\abs{\int_S\int_X\bc{ \kappa_{s,x}(\omega)- \kappa'_{\varphi(s),x}(\omega)}\dd x\dd s},\\
\cutmFK(\kappa,\kappa')&=
\sup_{\substack{S,X\subset[0,1]\\\omega\in\Omega}}
\abs{\int_S\int_X\bc{ \kappa_{s,x}(\omega)- \kappa'_{s,x}(\omega)}\dd x\dd s},\\
\cutmgr(\kappa,\kappa')&=\inf_{\varphi\in\SS}
\sup_{\substack{S,X\subset[0,1]\\\omega\in\Omega}}
\abs{\int_S\int_X\bc{\kappa_{s,x}(\omega)-\kappa'_{\varphi(s),\varphi(x)}(\omega)}\dd x\dd s}.
\end{align*}
Thus, $\cutm(\nix,\nix)$ is the natural extension of \eqref{eqcutmkernel} to $\KernelR$, $\cutms(\nix,\nix)$ is the kernel version of \eqref{eqNoPerm}, $\cutmFK(\nix,\nix)$ represents the strongest variant of the cut metric that does not allow for any measure-preserving transformations, and $\cutmgr(\nix,\nix)$ is the graphon cut metric as studied in~\cite{LovaszSzegedy_LimitsDecorated}.
We also recall the graphon cut (pre-)metric for $L^1$-functions $\kappa,\kappa':[0,1]^2\to\RR$ from~\cite{Lovasz}, which is defined as
\begin{align*}
\cutmgr(\kappa,\kappa')&=\inf_{\varphi\in\SS}\sup_{S,X\subset[0,1]}\abs{\int_S\int_X\bc{\kappa_{s,x}-\kappa'_{\varphi(s),\varphi(x)}}\dd x\dd s}.
\end{align*}	

The different variants of the cut metric are related as follows.
For a measurable map $\varphi:[0,1]\to[0,1]$ and $\kappa\in\KernelR$ define
$\kappa_{\varphi},\kappa^{\varphi}\in\KernelR$ by letting
$\kappa_{\varphi\,s,x}=\kappa_{s,\varphi(x)}$ and $\kappa^{\varphi}_{s,x}=\kappa_{\varphi(s),x}$, respectively.
Then
\begin{align}\label{eqinfs}
\cutm(\kappa,\kappa')&=\inf_{\psi\in\SS}\cutms(\kappa,\kappa_\psi'),&
\cutms(\kappa,\kappa')&=\inf_{\varphi\in\SS}\cutmFK(\kappa,{\kappa'}^\varphi).
\end{align}
As a consequence, for all $\kappa, \kappa'\in\KernelR$ we have
\begin{align}\label{eqpinch}
\cutm(\kappa, \kappa') &\leq \cutms(\kappa, \kappa') \leq\cutmFK(\kappa, \kappa')&&\mbox{and}& \aco{\cutmgr(\kappa, \kappa') \geq \cutm(\kappa, \kappa').}
\end{align}

For a function $W:(s,x)\mapsto W_{s,x}$ defined on $[0,1]^2$ we define the {\em transpose} $W^\dagger:(s,x)\mapsto W_{x,s}$.
We call $W$ {\em symmetric} if $W=W^\dagger$.
For $\kappa\in\KernelR$ we define a family $(\kappa^{(\omega)})_{\omega\in\Omega}$ of symmetric functions defined by
\begin{align}\label{eqkappaomega}
\kappa^{(\omega)}_{s/2,(1+x)/2}&=\kappa_{s,x}(\omega),
&\kappa^{(\omega)}_{(1+s)/2,x/2}&=\kappa_{x,s}(\omega),&
\kappa^{(\omega)}_{s/2,x/2}&=\kappa^{(\omega)}_{(1+s)/2,(1+x)/2}=0.
\end{align}	
We can interpret $\kappa_{s,x}^{(\omega)}$ {as} the edge weight in a bipartite graph with vertex set $[0,1]$. \mhk{We stress, that in \eqref{eqkappaomega} the ordering of $s$ and $x$ is quite delicate.}

\begin{lemma}\label{Lemma_bipartite}
For all $\kappa\in\KernelR$ we have
$\cutmFK(\kappa,\kappa')= 2 \max_{\omega\in\Omega}\cutmFK(\kappa^{(\omega)},{\kappa'}^{(\omega)}).$
\end{lemma}
\begin{proof}
Given $\omega \in \Omega$ and $S, X \subset [0,1]$ let
$T=\cbc{(1+s)/2:s\in S}\cup\cbc{x/2:x\in X},\ Y=\cbc{(1+x)/2:x\in X}\cup\cbc{s/2:s\in S}.$
Then by construction
\begin{align}\label{eqLemma_bipartite1}
	2\left| \int_{T} \int_{Y} \bc{\kappa_{s,x}^{(\omega)} - {\kappa'}^{(\omega)}_{s,x} }\dd s \dd x \right| = \left| \int_{S} \int_{X} \bc{\kappa_{s,x}(\omega) - {\kappa'}_{s,x}(\omega) }\dd s \dd x \right|.
\end{align}
Hence, $\cutmFK(\kappa,\kappa')\leq 2\max_{\omega\in\Omega}\cutmFK(\kappa^{(\omega)},{\kappa'}^{(\omega)})$.
Regarding the converse bound, we may assume by symmetry that $T,Y\subset[0,1]$ satisfy $T=1-Y$. \mhk{Indeed, the choice $T = 1-Y$ incorporates that the upper right part of the kernel is the transposed lower left part.}
Therefore, letting $S=\cbc{2t-1:t\in T\cap[1/2,1]},\ X=\cbc{2t:t\in T\cap[0,1/2]}$
we again obtain \eqref{eqLemma_bipartite1}, and thus 
$\cutmFK(\kappa,\kappa')\geq 2\max_{\omega\in\Omega}\cutmFK(\kappa^{(\omega)},{\kappa'}^{(\omega)})$.
\end{proof}

\aco{\begin{remark}
Clearly, the cut metric from the theory of graph limits $\cutmgr(\cdot, \cdot)$ can be bounded from below by  the present definition $ \cutm(\cdot, \cdot)$, as one must apply the same measure-preserving transformation on both axes. 
On the other hand, for $\Omega = \cbc{0,1}$, once we turn kernels $\kappa,\kappa'\in\kernel$ into `bipartite graphons' via \eqref{eqkappaomega}, we find directly 
$$\cutmgr(\kappa^{(1)}, {\kappa'}^{(1)})\leq \frac{1}{2} \cutm(\kappa,\kappa').$$
The converse bound does not hold for any constant as can be seen as follows. Let $\kappa_{s,x}=\vecone\cbc{s<1/2}$ and $\kappa'_{s,x}=\kappa_{x,s}$. By choosing the measure preserving map $\varphi(x) = 1-x$, we get $$\cutmgr(\kappa,\kappa') \leq \sup_{S,X\subset[0,1]}\abs{\int_S\int_X\bc{\kappa_{s,x}-\kappa'_{\varphi(s),\varphi(x)}}\dd x\dd s} = 0.$$ But as $\kappa'$ represents the law $\nu$ supported only on $\delta_\sigma$ with $\sigma_x = \ind\cbc{x \leq 1/2} $, whilst $\kappa'$ is the uniform distribution over the two configurations $\sigma_1 = 1$ and $\sigma_0 = 0$, we can bound $\cutm(\kappa, \kappa') \geq1/4$.
\end{remark}}

For $\kappa,\kappa'\in L^1([0,1]^2,\RR)$ we define
\begin{align}\label{Def_CutnormKernel}
\cutnorm{\kappa}&=\sup_{S,X\subset[0,1]}\abs{\int_S\int_X\kappa_{s,x}\dd s\dd x},&
\cutmFK(\kappa,\kappa')&=\cutnorm{\kappa-\kappa'}=
\sup_{\substack{S,X\subset[0,1]}}
\abs{\int_S\int_X \bc{\kappa_{s,x}- \kappa'_{\varphi(s),x}}\dd x\dd s}.
\end{align}
Then $\cutnorm\nix$ is a norm on $L^1([0,1]^2,\RR)$.
Analogously, for a matrix $A \in \RR^{n \times n}$ we define
\begin{align} \label{Def_CutnormMatrix}
\cutnorm{A}&= \frac{1}{n^2} \max_{S,X\subset[ n ]}\abs{ \sum_{s\in S}\sum_{x\in X} A_{s\,x} }.
\end{align}
We need the following `sampling lemma' for the cut norm.

\begin{lemma}[{\cite[\Lem~10.6]{Lovasz}}]\label{Lemma_LLnorm}
\mhk{Suppose that $\kappa:[0,1]^2\to[-1,1]$ is symmetric.
Let $\vx_1, \ldots, \vx_k$ be independently and uniformly chosen from $[0,1]$. Denote by $\kappa[k]\in[-1,1]^{k\times k}$ the matrix with entries $\kappa_{i,j}[k]=\kappa_{\vx_i,\vx_j}$.}
Then $$\pr\brk{\cutnorm{\kappa[k]}\leq\cutnorm{\kappa}+8/k^{1/4}}\geq1-4\exp(-\sqrt{k}/10).$$
\end{lemma}

\subsection{The $L_1$-metric}

We define a subspance of $\KernelR$ by letting
\begin{align*}
\Kernel_1=\cbc{\kappa\in\KernelR:0\leq\kappa_{s,x}(\omega)\leq1}.
\end{align*}
Similarly, we let $\cS_1$ be the space of all measurable functions $\sigma:[0,1]\to[0,1]^\Omega$.
Further, we denote the $L_1$-metric on $\Kernel_1$ and $\cS_1$ by $D_1(\nix,\nix)$.
Thus,
\begin{align*}
D_1(\kappa,\kappa')&=\sum_{\omega\in\Omega}\int_0^1\int_0^1 \abs{\kappa_{s,x}(\omega)-\kappa'_{s,x}(\omega)}\dd x\dd s&&(\kappa,\kappa'\in \Kernel_1),
\end{align*}
and similarly for $\cS_1$.

\subsection{Regularity}

For a kernel $\kappa\in\KernelR{}$ and partitions $S=(S_1,\ldots,S_k)$, $X=(X_1,\ldots,X_\ell)$ of the unit interval into pairwise disjoint measurable subsets define $\kappa^{S,X}\in\Kernel{}$ by
\begin{align*}
\kappa_{s,x}^{S,X}(\omega)&=\sum_{i\in[k]:\lambda(S_i)>0}\sum_{j\in[\ell]:\lambda(X_j)>0}
\frac{\vecone\cbc{\aco{(s,x)\in S_i\times X_j}}}{\lambda(S_i)\lambda(X_j)}
\int_{S_i}\int_{X_i}\kappa_{t,y}(\omega)\dd y\dd t.
\end{align*}
In words, $\kappa^{S,X}_{s,x}$ is the conditional expectation of $\kappa_{\vs,\vx}$ given the $\sigma$-algebra generated by the rectangles $S_i\times X_j$.
If the two partitions $S,X$ are identical, we write $\kappa^{S}$ instead of $\kappa^{S,X}$.
We use similar notation for maps $\kappa:[0,1]^2\to\RR$.
The following fact is a kernel variant of the well-known Frieze-Kannan regularity lemma.

\begin{lemma}[{\cite[Corollary 9.13]{Lovasz}}]\label{Lem_WeakRegularityKernels}
For every symmetric $\kappa: [0,1]^2 \to [0,1]$ and every $k \geq 1$ there exists a partition $S=(S_1,\ldots,S_k)$ of $[0,1]$ into pairwise disjoint measurable sets such that $\cutmFK(\kappa, \kappa^{S}) \leq 2/\sqrt{\log k}.$
\end{lemma} 

\noindent
This notion of regularity is robust with respect to refining the partition.

\begin{lemma}[{\cite[Lemma 9.12]{Lovasz}}]\label{Lem_RefiningPartitionCutDistance}
Let $\kappa: [0,1]^2 \to [0,1]$ be symmetric and $\kappa':[0,1]^2 \to [0,1]$ be a symmetric step function and denote by $S$ \aco{a partition of $[0,1]$ into a finite number of meausrable sets on which $\kappa'$ is constant}. Then $\cutmFK(\kappa, \kappa^S) \leq 2 \cutmFK(\kappa, \kappa').$
\end{lemma} 

\noindent
Applying \Lem~\ref{Lem_RefiningPartitionCutDistance} to the step function $\kappa^R$ for a partition $R$ that refines a partition $S$ of $[0,1]$, we obtain the following corollary.

\begin{corollary}\label{Cor_RefiningPartitionCutDistance}
Let $R, S$ be partitions of $[0,1]$ such that $R$ refines $S$. Then $\cutmFK(\kappa, \kappa^R) \leq 2 \cutmFK(\kappa, \kappa^S).$
\end{corollary}
\begin{proof}
	\aco{	This follows from \Lem~\ref{Lem_RefiningPartitionCutDistance} because $\kappa^S$ is constant on the partition classes of $R$.}
\end{proof}

\section{Fundamentals}\label{Sec_fundamentals}

\noindent
This section contains the proofs of the basic facts, namely the compactness of the space of $\Omega$-laws (\Thm~\ref{Prop_Polish}), the isometric property of the kernel representation (\Thm~\ref{Prop_kernel}), the sampling theorem (\Thm~\ref{Thm_sampling}), the comparison of the discrete and the continuous cut metric (\Thm~\ref{Prop_embedding}),
the continuity statements from \Thm s~\ref{Thm_tensor} and~\ref{Thm_oplus}
and the connection to exchangeable arrays (\Thm~\ref{Thm_exch}).
We begin with the proof of \Thm~\ref{Prop_kernel}.

\subsection{Proof of \Thm~\ref{Prop_kernel}}
Any measurable map $f:[0,1]\to\conf$, $s\mapsto f_s$ induces a kernel $\kappa^f:[0,1]^2\to\cP(\Omega)$,
$(s,x)\mapsto f_{s,x}\in\cP(\Omega)$.
Moreover, $f$ maps the Lebesgue measure on $[0,1]$ to a probability distribution $\mu^f\in\Law$.

\begin{lemma} \label{Lem_Kernel1}
Suppose that $f,g:[0,1]\to\conf$ are measurable.
Then $\cutms(\mu^f, \mu^g) \leq \cutms(\kappa^f, \kappa^g)$.
\end{lemma}
\begin{proof}
Fix $\omega\in\Omega$ and $\varphi\in\SS$.
The construction of $\kappa^f,\kappa^g$ guarantees that with $\vs\in[0,1]$ chosen uniformly at random, the distribution $\gamma$ of the pair
$(\kappa^f_{\vec s},\kappa^g_{\varphi(\vec s)})\in\conf\times\conf$
is a coupling of $\mu^f,\mu^g$.
We now claim that
\begin{align}\label{eqLem_Kernel1_1}
	\sup_{T\subset\conf^2,X\subset[0,1]}\abs{\int_T\int_X\bc{\sigma_{x}(\omega)-\tau_x(\omega)}\dd x\dd\gamma(\sigma,\tau)}
	\leq\sup_{S,X\subset[0,1]}\abs{\int_S\int_X\bc{\kappa^f_{s,x}(\omega)-\kappa^g_{s,\varphi(x)}(\omega)}\dd x\dd s}.
\end{align}
Indeed, fix measurable $T\subset\conf^2$ and $X\subset[0,1]$ and let
$S=\cbc{s\in[0,1]:(\kappa^f_s,\kappa^g_{\varphi(s)})\in T}$.
Then by the construction of $\gamma$,
\begin{align*}
	\int_T\int_X\bc{\sigma_{x}(\omega)-\tau_x(\omega)}\dd x\dd\gamma(\sigma,\tau)&=
	\int_S\int_X\bc{\kappa^f_{s,x}(\omega)-\kappa^g_{s,\varphi(x)}(\omega)}\dd x\dd s,
\end{align*}
whence \eqref{eqLem_Kernel1_1} follows.
Finally, since \eqref{eqLem_Kernel1_1} holds for all $\varphi,\omega$, we conclude that 
$\cutms(\mu^f, \mu^g) \leq \cutms(\kappa^f, \kappa^g)$.
\end{proof}

\noindent
The following lemma establishes the converse of \Lem~\ref{Lem_Kernel1} for functions that take only finitely many values.

\begin{lemma} \label{Lem_Kernel2}
Suppose that $f,g:[0,1]\to\conf$ are measurable maps whose images $f([0,1]),g([0,1])\subset\conf$ are finite sets.
Then $\cutms(\kappa^f, \kappa^g)\leq\cutms(\mu^f, \mu^g)$.
\end{lemma}
\begin{proof}
Suppose that $f([0,1])=\{\sigma_1,\ldots,\sigma_k\}$ and $g([0,1])=\{\tau_1,\ldots,\tau_\ell\}$.
Moreover, let $V_i$ be the set of all $s\in[0,1]$ such that $f(s)=\sigma_i$ and let
$W_j$ be the set of all $s\in[0,1]$ such that $g(s)=\tau_j$.
In addition, let $v_i=\lambda(V_i)$, $w_j=\lambda(W_j)$.
Then
\begin{align*}
	\mu^f&=\sum_{i=1}^k v_i\delta_{\sigma_i},&\mu^g&=\sum_{j=1}^\ell w_j\delta_{\tau_j}.
\end{align*}
Consequently, any coupling $\gamma$ of $\mu^f,\mu^g$ induces a coupling $\Gamma\in\cP([k]\times[\ell])$ of
the probability distributions $(v_1,\ldots,v_k)$ and $(w_1,\ldots,w_\ell)$.
To turn $\Gamma$ into a measure-preserving map $[0,1]\to[0,1]$ we partition any sets $V_i,W_j$ into pairwise disjoint measurable subsets $(V_{i,h})_{h\in[\ell]}$ and $(W_{h,j})_{h\in[k]}$, respectively, such that for all $i,j,h$,
\begin{align*}
	\lambda(V_{i,h})&=g(i,h),&\lambda(W_{h,j})&=g(h,j).
\end{align*}
Then by \Lem~\ref{Lemma_Lambda} for any $i,j$ there exists a bijection $\varphi_{i,j}:V_{i,j}\to W_{i,j}$ such that both $\varphi_{i,j}$ and $\varphi_{i,j}^{-1}$ are measurable and preserve the Lebesgue measure.
Piecing these maps together, we obtain the bijection
\begin{align*}
	\varphi&:[0,1]\to[0,1],&s\mapsto\sum_{(i,j)\in[k]\times[\ell]}\vecone\cbc{s\in V_{i,j}}\varphi_{i,j}(s).
\end{align*}
Both $\varphi$ and $\varphi^{-1}$ are measurable and preserve the Lebesgue measure, i.e., $\varphi\in\SS$.
Moreover, for any sets $S,X\subset[0,1]$ and any $\omega\in\Omega$ we have
\begin{align}\label{eqLem_Kernel2_1}
	\int_S\int_X\bc{\kappa_{s,x}^f(\omega)-\kappa_{\varphi(s),x}^g(\omega)}\dd x\dd s
	&=\sum_{i=1}^k\sum_{j=1}^\ell\lambda(S\cap V_{ij})\int_X\bc{\sigma_{i,x}(\omega)-\tau_{i,x}(\omega)}\dd x.
\end{align}
Hence, \eqref{eqLem_Kernel2_1} is extremised by sets $S$ such that for all $i,j$ either $V_{ij}\subset S$ or $S\cap V_{ij}=\emptyset$.
For such a set $S$ let $T=T(S)$ contain all pairs $(i,j)$ such that $V_{ij}\subset S$.
Then \eqref{eqLem_Kernel2_1} yields
\begin{align}\label{eqLem_Kernel2_2}
	\abs{\int_S\int_X\bc{\kappa_{s,x}^f(\omega)-\kappa_{\varphi(s),x}^g(\omega)}\dd x\dd s}
	&=\abs{\sum_{(i,j)\in T}\Gamma(i,j)\int_X\bc{\sigma_{i,x}(\omega)-\tau_{i,x}(\omega)}\dd x}
	\leq\sup_{U\subset\conf^2}
	\abs{\int_U\int_X\bc{\sigma_{x}(\omega)-\tau_{x}(\omega)}\dd x\dd\gamma(\sigma,\tau)}.
\end{align}
Since \eqref{eqLem_Kernel2_2} holds for all $S,X,\omega,\gamma$, the assertion follows.
\end{proof}

\begin{corollary}\label{Cor_Kernel2}
Let $f,g:[0,1]\to\conf$ be measurable.
Then $\cutms(\kappa^f, \kappa^g)\leq\cutms(\mu^f, \mu^g)$.
\end{corollary}
\begin{proof}
Because $\conf$ is a convex subset of the separable Banach space $L^1([0,1],\RR^\Omega)$,
the measurable maps $f,g$ are pointwise limits
of sequences $(f_n)_{n\geq1}$, $(g_n)_{n\geq1}$ of measurable functions $f_n,g_n:[0,1]\to\conf$ 
whose images are finite sets.
Moreover, \Lem~\ref{Lem_Kernel2} implies that
\begin{align}\label{eqCor_Kernel2_1}
	\cutms(\mu^{f_n},\nu^{f_n})&\geq\cutms(\kappa^{f_n},\kappa^{g_n})&&\mbox{for all }n\geq1.
\end{align}
Further, for all $\omega\in\Omega$ and $S,X\subset[0,1]$ we have
\begin{align}\label{eqCor_Kernel2_2}
	\abs{\int_S\int_X \bc{\kappa^{f_n}_{s,x}(\omega)-\kappa^{f}_{s,x}(\omega)}\dd s\dd x}
	&\leq\int_0^1\int_0^1\abs{\kappa^{f_n}_{s,x}(\omega)-\kappa^{f}_{s,x}(\omega)}\dd s\dd x.
\end{align}
Because $f_n\to f$ pointwise,  the r.h.s.\ of \eqref{eqCor_Kernel2_2} vanishes as $n\to\infty$.
Consequently,
\begin{align}\label{eqCor_Kernel2_3}
	\lim_{n\to\infty}\cutms(\kappa^{f_n},\kappa^f)&=0,&\mbox{and similarly}&&
	\lim_{n\to\infty}\cutms(\kappa^{g_n},\kappa^g)&=0.
\end{align}
Combining \eqref{eqCor_Kernel2_3} with \Lem~\ref{Lem_Kernel1}, we conclude that
\begin{align}\label{eqCor_Kernel2_4}
	\lim_{n\to\infty}\cutms(\mu^{f_n},\mu^f)&=0,&
	\lim_{n\to\infty}\cutms(\nu^{f_n},\nu^f)&=0.
\end{align}
Finally, the assertion follows from \eqref{eqCor_Kernel2_1}, \eqref{eqCor_Kernel2_3}, \eqref{eqCor_Kernel2_4}
and the triangle inequality.
\end{proof}

\begin{corollary}\label{Cor_Kernel3}
For all $\kappa,\kappa'\in\Kernel{}$ we have
$\cutms(\mu^\kappa, \mu^{\kappa'}) =\cutms(\kappa, \kappa')$.
\end{corollary}
\begin{proof}
This is an immediate consequence of \Lem~\ref{Lem_Kernel1} and \Cor~\ref{Cor_Kernel2}.
\end{proof}

\begin{proof}[Proof of \Thm~\ref{Prop_kernel}]
\Cor~\ref{Cor_Kernel3} and \eqref{eqinfs} show that the map $\kernel\to\law$, $\kappa\mapsto\mu^\kappa$  is an isometry.
Moreover, \Lem~\ref{Lemma_standard} implies that this map is surjective.
Thus, because $\kernel,\law$ are metric spaces, $\kappa\mapsto\mu^\kappa$ is an isometric bijection.
\end{proof}

\subsection{Proof of Theorem \ref{Thm_sampling}}
We begin by extending \Lem~\ref{Lem_WeakRegularityKernels} to (not necessarily symmetric) kernels $\kappa\in\Kernel{}$.

\begin{lemma}\label{Thm_rl}
There is $c=c(\Omega)>0$ such that for any $\aco{\eps\in(0,1)}$, $\kappa\in\Kernel$ there exist partitions $S=(S_1,\ldots,S_k)$, $X=(X_1,\ldots,X_\ell)$ of the unit interval into measurable subsets such that
$k+\ell\leq\exp(c/\eps^2)$ and $\cutmFK(\kappa,\kappa^{S,X})<\eps$.
\end{lemma}
\begin{proof}
Let $\ell=\lceil\exp(c'/\eps^2)\rceil$ for a large enough $c'=c'(\Omega)$.
Applying \Lem~\ref{Lem_WeakRegularityKernels} to the kernels $\kappa^{\bc\omega}$ from \eqref{eqkappaomega},
we obtain partitions $T^{\bc\omega}=(T^{\bc\omega}_1,\ldots,T^{\bc\omega}_\ell)$ of $[0,1]$ such that
\begin{align}\label{eqThm_rl1}
	\cutmFK\bc{\kappa^{\bc\omega},\kappa^{\bc\omega\,T^{\bc\omega}}}<\eps/4.
\end{align}
Let $T=(T_1,\ldots,T_k)$ be the coarsest common refinement of all the partitions $T^{\bc\omega}$ and of the partition $\{[0,1/2),[1/2,1]\}$.
Then
\begin{align}\label{eqThm_rl2}
	|T|&\leq 2\ell^{|\Omega|}.
\end{align}
Moreover, \eqref{eqThm_rl1} and \Cor~\ref{Cor_RefiningPartitionCutDistance} imply that
\begin{align}\label{eqThm_rl3}
	\cutmFK(\kappa^{\bc\omega},\kappa^{\bc\omega\,T})&<\eps/2&\mbox{ for all }\omega\in\Omega.
\end{align}
Further, let $S'=(S_1',\ldots,S_K')$ comprise all partition classes $T_i\subset[0,1/2]$ and let
$X'=(X_1',\ldots,X_L')$ be the partition of $[1/2,1]$ consisting of all the classes $T_i\subset[1/2,1]$.
Finally, let $S_i=\{2s:s\in S_i'\}$ and $X_i=\{2x-1:x\in X_i'\}$.
Then the partitions $S=(S_1,\ldots,S_K)$ and $X=(X_1,\ldots,X_L)$ satisfy
$\cutmFK(\kappa,\kappa^{S,X})<\eps$ by \Lem~\ref{Lemma_bipartite}.
\aco{The desired bound on the total number $K+L$ of classes of $S,X$ follows from \eqref{eqThm_rl2}.}
\end{proof}

For a kernel $\kappa$ and an integer $n$ obtain $\kappa_{n}$ as follows.
Draw $\vx_1,\vs_1,\ldots,\vx_n,\vs_n\in[0,1]$ uniformly and independently  and let $\kappa_{n}$ be the kernel
representing the matrix $(\kappa_{\vs_i,\vx_j})_{i,j}$.
Additionally, obtain $\hat\kappa_n\in\Omega^{n\times n}$ by letting $\hat\kappa_{n,i,j}=\omega$ with probability
$\kappa_{\vs_i,\vx_j}(\omega)$ independently for all $i,j$.
We identify $\hat\kappa_n$ with its kernel representation.
\aco{Moreover, we notice that $\hat\kappa_n$ coincides with the $n\times n$ upper left sub-matrix of $\vX^\kappa$ from \Sec~\ref{Sec_cands}.}

\begin{lemma}\label{Lemma_fsl}
Let $\kappa,\kappa'\in\Kernel{}$.
With probability $1-\exp(-\Omega(\sqrt n))$ we have 
$\cutmFK\bc{\kappa_n,\kappa_n'} = O(\cutmFK(\kappa,\kappa')+n^{-1/4})$.
\end{lemma}
\color{red}
\begin{proof}
Let $\bar \kappa, \bar \kappa'$ be the symmetric kernel representations of $\kappa$ and $\kappa'$ respectively given via \eqref{eqkappaomega}.
Sample $\vec y_1, \ldots, \vec y_{2n}$ points in $[0,1]$ uniformly and independently at random. Denote by $\cB$ the event that $\abs{ i: \vy_i \leq \frac{1}{2}} = n$ and assume, given $\cB$, that without loss  $\vy_1, \ldots, \vy_n \leq 1/2$ and $\vy_{n+1}, \ldots, \vy_{2n} \geq 1/2$. Denote by $\vx_1, \ldots, \vx_n = 2 \vy_1, \ldots, 2 \vy_n$ and by $\vs_1, \ldots, \vs_{n} = 2\vy_{n+1} - 1, \ldots, 2 \vy_{2n} - 1$. Clearly, $(\vx_1, \vs_1), \ldots, (\vx_n, \vs_n) $ are independent uniform samples from $[0,1]^2$.

Now, let
$\tilde{\kappa}:[0,1]^2\to[-1,1]$ be the kernel representing the matrix $(\bar \kappa_{\vy_i,\vy_j}- \bar \kappa_{\vy_i,\vy_j}')_{i,j\in[2n]}$.
Applying \Lem~\ref{Lemma_LLnorm} to $\tilde\kappa^{\bc\omega}$, we obtain
\begin{align}\label{eqLemma_fsl1}
	\pr\brk{\cutnorm{\tilde\kappa^{\bc\omega}}\leq \cutnorm{\bar \kappa - \bar \kappa'}+8n^{-1/4}}&\geq1-4\exp(-\sqrt n/10)
	&(\omega\in\Omega).
\end{align}

Given $\cB$, we translate $\tilde \kappa, \bar \kappa, \bar \kappa'$ back into kernels via \eqref{eqkappaomega} and apply \Lem~\ref{Lemma_bipartite}, thus  $$\cutmFK\bc{\kappa_n,\kappa_n'}\leq4\max_{\omega\in\Omega}\cutnorm{\tilde\kappa^{\bc\omega}}.$$ Hence, the assertion follows from \eqref{eqLemma_fsl1} and the fact that $\Pr \bc{\cB} = \Omega \bc{ n^{-1/2} }$.
\end{proof}
\color{black}

\begin{lemma}\label{Lemma_sl}
We have $\Erw[\cutmFK(\kappa_n,\hat\kappa_n)]=O(n^{-1/2})$.
\end{lemma}
\begin{proof}
We adapt the simple argument from the proof of \cite[Lemma 10.11]{Lovasz} for our purposes.
Letting $X_{i,j, \omega} = \ind\cbc{\hat\kappa_{n,i,j}=\omega}$, we have 
$\Erw[X_{i,j, \omega}] = \kappa_{n,i,j}(\omega)$.
Furthermore, because both $\kappa_n,\hat\kappa_n$ are kernel representations of $n\times n$ matrices, the supremum
\begin{align*}
	\sup_{\omega\in\Omega,S,X\subset[0,1]}\abs{\int_S\int_X\bc{\kappa_{n,s,x}(\omega)-\hat\kappa_{n,s,x}(\omega) }\dd x \dd s}
\end{align*}
is attained at sets $S,X$ that are unions of intervals $[(i-1)/n,i/n)$ with $i\in[n]$.
Hence,
\begin{align}\label{eqLemma_sl_1}
	\cutmFK\bc{\kappa_n,\hat\kappa_n}&=
	\sup_{\omega\in\Omega,S,X\subset[0,1]}\abs{\int_S\int_X\bc{\kappa_{n,s,x}(\omega)-\hat\kappa_{n,s,x}(\omega) }\dd x \dd s }
	=\aco{n^{-2}}\max_{\omega\in\Omega,I,J\subset[n]}\abs{ \sum_{i \in I}\sum_{j \in J} X_{i, j, \omega} - \Erw[X_{i, j, \omega}] }.
\end{align}
Now, for any $\omega,I,J$ the random variable $\sum_{i \in I}\sum_{j \in J} X_{i, j, \omega}$ is a sum of $|I\times J|$ independent Bernoulli variables.
Therefore, Azuma's inequality yields
\begin{align}\label{eqLemma_sl_2}
	\pr\brk{\abs{ \sum_{i \in I}\sum_{j \in J} X_{i, j, \omega} - \Erw[X_{i, j, \omega} ] } >10n^{3/2}}
	&\leq\exp(-10n).
\end{align}
Since \eqref{eqLemma_sl_2} holds for any specific $I,J,\omega$, the assertion follows from the union bound and \eqref{eqLemma_sl_1}.
\end{proof}

\begin{proof}[Proof of Theorem \ref{Thm_sampling}]
\Lem~\ref{Thm_rl} yields partitions $X=(X_1,\ldots,X_\ell)$, $S=(S_1,\ldots,S_\ell)$ of $[0,1]$ with $\ell\leq n^{1/4}$ such that
\begin{align}\label{eqThm_sampling1}
	\cutmFK(\kappa,\kappa^{S,X})& = O(\log^{-1/2}n).
\end{align}
Applying \Lem~\ref{Lemma_fsl} to $\kappa$ and $\kappa^{S,X}$, we obtain
\begin{align}\label{eqThm_sampling2}
	\Erw\brk{\cutmFK(\kappa_n,\kappa^{S,X}_n)} = O\bc{\cutmFK(\kappa,\kappa^{S,X})+n^{-1/4}}.
\end{align}
In addition, we claim that
\begin{align}\label{eqThm_sampling3}
	\Erw\brk{\cutm(\kappa^{S,X},\kappa^{S,X}_n)} = O(n^{-1/4}\log n).
\end{align}
To see this, let
\begin{align*}
	N_h&=\cbc{i\in[n]:\vx_i\in X_h},&M_h&=\cbc{j\in[n]:\vs_j\in S_h}
	&(h\in[\ell]).
\end{align*}
Since $N_h,M_h$ are binomial variables, the Chernoff bound shows that with probability $1-o(1/n)$,
\begin{align}\label{eqThm_sampling4}
	\max_{h\in\ell}{\abs{|N_h|-n\lambda(X_h)}}&\leq\sqrt{n}\log n,&
	\max_{h\in\ell}{\abs{|M_h|-n\lambda(S_h)}}&\leq\sqrt{n}\log n.
\end{align}
Let
\begin{align*}
	\cN_h&=\bigcup_{i\in N_h}[(i-1)/n,i/n),&\cM_h=\bigcup_{i\in M_h}[(i-1)/n,i/n).
\end{align*}
Providing that the bounds \eqref{eqThm_sampling4} hold, we can construct $\varphi,\psi\in\SS$ such that for all $h\in[n]$,
\begin{align}\label{eqThm_sampling5}
	\lambda\bc{\varphi(\cN_h)\triangle X_h}&\leq n^{-1/2}\log n,&
	\lambda\bc{\psi(\cM_h)\triangle S_h}&\leq n^{-1/2}\log n.
\end{align}
Furthermore, by construction we have $\kappa^{S,X}_{\varphi(s),\psi(x)}=\kappa^{S,X}_{s,x,n}$ if 
there exist $h,h'\in[\ell]$ such that $x\in\cN_h$, $\varphi(x)\in X_h$ and $s\in\cM_h$, $\psi(x)\in S_h$.
Therefore, \eqref{eqThm_sampling5} implies that for all $T,Y\subset[0,1]$, $\omega\in\Omega$,
\begin{align*}
	&\abs{\int_T\int_Y \bc{\kappa^{S,X}_{n,y,t}(\omega)-\kappa^{S,X}_{\psi(y),\varphi(t)}(\omega) }\dd y\dd t}
	\leq\sum_{h,h'=1}^\ell
	\abs{\int_{T\cap\cM_h}\int_{Y\cap\cN_{h'}} \bc{\kappa^{S,X}_{n,y,t}(\omega)-\kappa^{S,X}_{\psi(y),\varphi(t)}(\omega) }\dd y\dd t}
	\\
	&\qquad\leq \aco{\sum_{h,h'=1}^\ell{\lambda(T\cap \cM_h\triangle\psi^{-1}(S_h))\lambda(Y\cap \cN_{h'}) +\lambda(T\cap \cM_h)\lambda(Y\cap \cN_{h'}\triangle\psi^{-1}(X_{h'}))}}\\
	&\qquad=O(\ell n^{-1/2}\log n)=O(n^{-1/4}\log n),
\end{align*}
whence \eqref{eqThm_sampling3} follows.
\aco{Combining \eqref{eqThm_sampling1}, \eqref{eqThm_sampling2} and \eqref{eqThm_sampling3}, we see that
\begin{align}\label{eqFix}
\Erw\brk{\cutm(\kappa,\kappa_n)}&=O(\log^{-1/2}n).
\end{align}
Finally, \eqref{eqFix}, \Lem~\ref{Lemma_sl} and \Thm~\ref{Prop_kernel} imply the assertion.}
\end{proof}

\subsection{Proof of \Thm s~\ref{Thm_tensor} and~\ref{Thm_oplus}}

For measurable $k,k':[0,1]^3\to[0,1]^\Omega$ we let
\begin{align*}
\cutmFK(k,k')&=\sup_{S\subset[0,1],X\subset[0,1]^2,\omega\in\Omega}
\abs{\int_S\int_X\bc{ k_{s,x,y}(\omega)-k'_{s,x,y}(\omega)}\dd x\dd y\dd s}.
\end{align*}
Then $\cutmFK(\nix,\nix)$ defines a pre-metric.
Further, for measurable $\kappa,\kappa':[0,1]^2\to[0,1]^\Omega$ we define
\begin{align*}
\kappa\oplus'\kappa'&:[0,1]^3\to[0,1]^\Omega,&
(s,x,y)\mapsto\kappa_{s,x}\tensor\kappa_{s,y}'.
\end{align*}
We will derive \Thm~\ref{Thm_oplus} from the following statement.

\begin{proposition}\label{Prop_oplus}
The map $(\kappa,\kappa')\mapsto\kappa\oplus'\kappa'$ is $\cutmFK$-continuous.
\end{proposition}
\begin{proof}
Given $\eps>0$ choose a small $\delta=\delta(\eps)>0$. 
Suppose that $\cutmFK(\kappa,\kappa')<\delta$.
Due to the triangle inequality, to establish continuity it suffices to show that for every $\kappa'':[0,1]^2\to[0,1]^\Omega$,
\begin{align}\label{eqThm_oplus2}
	\cutmFK(\kappa\oplus'\kappa'',\kappa'\oplus'\kappa'')
	=\sup_{\substack{S\subset[0,1]\\X\subset[0,1]^2\\\omega,\omega'\in\Omega}}
	\abs{\int_S\int_X 
		\bc{\kappa_{s,x}(\omega)\kappa''_{s,y}(\omega')-\kappa'_{s,x}(\omega)\kappa''_{s,y}(\omega')}\dd x\dd y\dd s}<\eps.
\end{align}
Thus, consider measurable $X,S$ and fix $\omega,\omega'\in\Omega$.
To estimate the last integral consider $y\in[0,1]$ and let $X_y=\cbc{x\in [0,1]:(x,y)\in X}\subset[0,1]$.
Moreover, let $T_1,\ldots,T_\ell$ be a decomposition of $S$ into pairwise disjoint measurable sets such that for all $j\in[\ell]$ we have
\begin{align*}
	t_{j,*}&\leq t_{j}^*+\eps/4,&&\mbox{where}& 
	t_{j,*}&=\inf_{s\in T_j}\kappa''_{s,y}(\omega'),&t_j^*=\sup_{s\in T_j}\kappa''_{s,y}(\omega').
\end{align*}
Since $\kappa''_{s,y}(\omega')\in[0,1]$, we may assume that $\ell\leq4/\eps$.
Furthermore,
\begin{align}\nonumber
	\abs{\int_{X_y}\int_S\bc{\kappa_{s,x}(\omega)-\kappa'_{s,x}(\omega)}\kappa''_{s,y}(\omega')\dd x\dd s}
	&\leq\sum_{j=1}^\ell 
	\abs{\int_{X_y}\int_{S\cap T_j} \bc{\kappa_{s,x}(\omega)-\kappa'_{s,x}(\omega)}\kappa''_{s,y}(\omega')\dd x\dd s}\\
	&\leq\frac\eps 4+
	\sum_{j=1}^\ell t_j^*
	\abs{\int_{X_y}\int_{S\cap T_j} \bc{\kappa_{s,x}(\omega)-\kappa'_{s,x}(\omega)}\dd x\dd s}\nonumber\\
	&\leq\frac\eps4+2\ell\cutmFK(\kappa,\kappa')\leq\frac\eps4+2\ell\delta<\eps/2.\nonumber 
\end{align}
Since this estimate holds for all $y\in[0,1]$, we obtain
\begin{align*}
	\abs{\int_S\int_X\bc{\kappa_{s,x}(\omega)\kappa''_{s,y}(\omega')-\kappa'_{s,x}(\omega)\kappa''_{s,y}(\omega')}\dd x\dd y\dd s}
	&\leq\int_0^1\abs{\int_{X_y}\int_S\bc{\kappa_{s,x}(\omega)\kappa''_{s,y}(\omega')-\kappa'_{s,x}(\omega)\kappa''_{s,y}(\omega')}
		\dd s\dd x}\dd y<\frac\eps 2
\end{align*}
for all $S,X,\omega,\omega'$.
Thus, we obtain \eqref{eqThm_oplus2}.
\end{proof}

\begin{proof}[Proof of \Thm~\ref{Thm_oplus}]
\Thm~\ref{Thm_oplus} follows from \Prop~\ref{Prop_oplus} and \eqref{eqinfs}.
\end{proof}

We use a similar argument to prove \Thm~\ref{Thm_tensor}.
Specifically, for $\kappa,\kappa':[0,1]^2\to[0,1]^\Omega$ define
\begin{align*}
\kappa\tensor'\kappa'&:[0,1]^3\to[0,1]^\Omega,&
(s,t,x)\mapsto\kappa_{s,x}\tensor\kappa_{t,x}'.
\end{align*}

\begin{proposition}\label{Prop_tensor}
The map $(\kappa,\kappa')\mapsto\kappa\tensor'\kappa'$ is $\cutmFK$-continuous.
\end{proposition}
\begin{proof}
The definition of $\cutmFK(\nix,\nix)$ ensures that the map
$\kappa\mapsto\kappa^\dagger$, where $\kappa^\dagger_{s,x}=\kappa_{x,s}$, is continuous.
Therefore, the assertion follows from \Prop~\ref{Prop_oplus}.
\end{proof}

\begin{proof}[Proof of \Thm~\ref{Thm_tensor}]
\Thm~\ref{Thm_tensor} follows immediately from \Prop~\ref{Prop_tensor} and  \eqref{eqinfs}.
\end{proof}

\subsection{Proof of \Thm \ref{Prop_Polish}}\label{Sec_Prop_Polish}
We begin by proving that the space $\Kernel{}$ is complete with respect to $\cutmFK(\nix,\nix)$, the strongest version of the cut metric.

\begin{lemma}\label{Lemma_complete}
The space $\Kernel$ equipped with the $\cutmFK(\nix,\nix)$ metric is complete.
\end{lemma}
\begin{proof}
Suppose that $(\kappa_n)_{n\geq1}$ is a Cauchy sequence.
Then for any measurable $S,X\subset[0,1]$ and any $\omega\in\Omega$ the sequence $\int_S\int_X\kappa_{n,s,x}(\omega)\dd x\dd s$ is Cauchy as well.
\aco{Therefore, because any continuous function $f:[0,1]^2\to\RR$, $(s,x)\mapsto f_{s,x}$ is uniformly continous, the limit
$$\lim_{n\to\infty}\int_0^1\int_0^1f_{s,x}\kappa_{n,s,x}(\omega)\dd x\dd s$$
exists for every $\omega\in\Omega$.
Indeed, the map 
$$f\mapsto\lim_{n\to\infty}\int_0^1\int_0^1f_{s,x}\kappa_{n,s,x}(\omega)\dd x\dd s$$
defines a positive linear functional on the space of all continuous functions $[0,1]^2\to\RR$.}
Hence, by the Riesz representation theorem (\Lem~\ref{Lemma_Riesz}) there exists a unique measure $\mu_\omega$ on $[0,1]^2$ such that
\begin{align}\label{eqLemma_complete1}
	\mu_\omega(S\times X)&=\lim_{n\to\infty}\int_S\int_X\kappa_{n,s,x}(\omega)\dd x\dd s.
\end{align}
Indeed, the condition~\eqref{eqLemma_complete1} ensures that $\mu_\omega$ is absolutely continuous with respect to the Lebesgue measure.
Therefore, the Radon-Nikodym theorem yields an $L^1$-function $(s,x)\in[0,1]^2\mapsto\kappa_{s,x}(\omega)\in\RRpos$ such that 
\begin{align}\label{eqLemma_complete2}
	\mu_\omega(Y)&=\int_Y \kappa_{s,x}(\omega)\dd s\dd x\enspace&&\mbox{for all measurable $Y\subset[0,1]^2$}.
\end{align}
We claim that $\kappa$ is a kernel, i.e., that $\sum_{\omega\in\Omega}\kappa_{s,x}(\omega)=1$ for almost all $s,x$.
Indeed, combining \eqref{eqLemma_complete1} and \eqref{eqLemma_complete2} yields
\begin{align}\label{eqLemma_complete3}
	\int_S\int_X1\dd x\dd s=\sum_{\omega\in\Omega}\mu_\omega(S\times X)
	=\sum_{\omega\in\Omega}\int_{S}\int_X\kappa_{s,x}(\omega)\dd x\dd s
	=\int_{S}\int_X\sum_{\omega\in\Omega}\kappa_{s,x}(\omega)\dd x\dd s.
\end{align}
Since the rectangles $S\times X$ generate the Borel algebra on $[0,1]^2$, \eqref{eqLemma_complete3}
implies that $\sum_{\omega\in\Omega}\kappa_{s,x}(\omega)=1$ almost everywhere.

\aco{Finally, \eqref{eqLemma_complete1} and \eqref{eqLemma_complete2} show that $\lim_{n\to\infty}\cutmFK(\kappa_n,\kappa)=0$.
Indeed, given $\eps>0$ consider a large enough $n$ and let $S,X,\omega$ be such that
\begin{align}
	\cutmFK(\kappa_n,\kappa)<\eps+\abs{\int_S\int_X\bc{\kappa_{s,x}(\omega)-\kappa_{n,s,x}(\omega)}\dd x\dd s}.
	\label{eqLemma_complete8}
\end{align}
Equations \eqref{eqLemma_complete1} and \eqref{eqLemma_complete2} show that for large enough $N>n$,
\begin{align}
	\abs{\int_S\int_X\bc{\kappa_{s,x}(\omega)-\kappa_{N,s,x}(\omega)}\dd x\dd s}&<\eps
	\label{eqLemma_complete9}
\end{align}
Combining \eqref{eqLemma_complete8} and \eqref{eqLemma_complete9} and recalling that $(\kappa_n)_n$ is $\cutmFK$-Cauchy, for large enough $n$ we obtain
\begin{align*}
	\cutmFK(\kappa_n,\kappa)&<\eps+	\abs{\int_S\int_X\bc{\kappa_{s,x}(\omega)-\kappa_{N,s,x}(\omega)}\dd x\dd s}+\cutmFK(\kappa_n,\kappa_N)<3\eps.
\end{align*}
Hence, $\kappa_n$ converges to $\kappa$.}
\end{proof}

\begin{corollary}\label{Cor_complete1}
The space $\Kernel$ equipped with the $\cutms(\nix,\nix)$ metric is complete.
\end{corollary}
\begin{proof}
{We adapt a well know proof that a quotient of a Banach space with respect to a linear subspace is complete \cite[Theorem 1.12.14]{Buehler}}.
Thus, suppose that $(\kappa_n)_n$ is a $\cutms(\nix,\nix)$-Cauchy sequence.
There exists a subsequence $(\kappa_{n_\ell})_\ell$ such that  $\cutms(\kappa_{n_\ell},\kappa_{n_{\ell+1}})<2^{-\ell}$ for all $\ell$.
Hence, passing to this subsequence, we may assume that $(\kappa_n)_n$ satisfies
\begin{align}\label{eqCor_complete1_1}
	\cutms(\kappa_n,\kappa_{n+1})&<2^{-n}&&\mbox{for all }n.
\end{align}
We are now going to construct a sequence $(k_n)_n$ of maps $[0,1]^2\to\cP(\Omega)$ such that
$\cutms(\kappa_n,k_n)=0$ for all $n$ and
\begin{align}\label{eqCor_complete1_2}
	\cutmFK(k_n,k_{n+1})&<2^{-n}&&\mbox{for all }n.
\end{align}
We let $k_1$ be any kernel such that $\cutms(k_1,\kappa_1)=0$ and proceed by induction.
Having constructed $k_1,\ldots,k_n$ already, we observe that the definition of $\cutms(\nix,\nix)$ ensures that
\begin{align*}
	\cutms(\kappa_n,\kappa_{n+1})=\cutms(k_n,\kappa_{n+1})&=\inf\cbc{\cutmFK(k_n,k):k:[0,1]^2\to\cP(\Omega),\cutms(\kappa_{n},k)=0}.
\end{align*}
Therefore, \eqref{eqCor_complete1_1} implies that there is $k_{n+1}:[0,1]^2\to\cP(\Omega)$ with 
$\cutms(\kappa_{n+1},k)=0$ such that $\cutmFK(k_n,k_{n+1})<2^{-n}$.
Thus, we obtain a sequence $(k_n)_n$ satisfying \eqref{eqCor_complete1_2}.
Finally, any sequence $(k_n)_n$ that satisfies \eqref{eqCor_complete1_2} is $\cutmFK(\nix,\nix)$-Cauchy.
Therefore, \Lem~\ref{Lemma_complete} shows that $(k_n)_n$ has a limit $k$.
Since $\cutms(k_n,\kappa_n)=0$, we conclude that $$\lim_{n\to\infty}\cutms(\kappa_n,k)=0,$$ i.e., $(\kappa_n)_n$ converges to $k$.
\end{proof}

\begin{corollary}\label{Cor_complete2}
The spaces $\kernel$ and $\law$ equipped with the $\cutm(\nix,\nix)$ metric are complete and separable.
\end{corollary}
\begin{proof}
To establish the completeness of $\kernel{}$ we repeat the same argument as in the proof of \Cor~\ref{Cor_complete1}.
The completeness of $\law$ then follows from \Thm~\ref{Prop_kernel}.
Moreover, \Thm~\ref{Thm_rl} shows that the set of laws with finite support is dense in $\law$.
\aco{Hence, to prove the separability of $\law$ it suffices to observe that the space $\conf$ is separable, which it is because the set of all finite linear combinations of indicator functions $x\mapsto\vecone\{a<x<b\}$ with $a,b\in\mathbb Q$ is dense in $L^1([0,1],\RR)$.}
Finally, \Thm~\ref{Prop_kernel} implies that $\kernel{}$ is separable as well.
\end{proof}

We denote by $\cP(\law)$ the space of probability distributions on the Polish space $\law$, endowed with the topology of weak convergence.
As we saw in \Sec~\ref{Sec_MeasureTheory}, this topology is metrised by the Wasserstein metric
\begin{align*}
\cutmW(\rho,\rho')&=\inf\cbc{\int_{\law\times \law}\cutm(\mu,\nu)\dd g(\mu,\nu):
	g\in\Gamma(\rho,\rho')}&&(\rho,\rho'\in\cP(\law)).
\end{align*}
We begin by proving that $\cP(\law)$ is compact.
To this end we will construct a continuous map from another compact space onto $\cP(\law)$.
Specifically, recall that $\Omega^{\NN\times\NN}$ is a compact Polish space with respect to the product topology.
The space $\cP(\Omega^{\NN\times\NN})$ equipped with the weak topology is therefore compact as well.
Further, the space $\exch\subset\cP(\Omega^{\NN\times\NN})$ of exchangeable distributions is closed with respect to the weak topology, and therefore compact.

To construct a continuous map $\exch\to \cP(\law)$, $\xi\mapsto\rho^\xi$ we are going to take a pointwise limit of maps $\exch\to \cP(\law)$, $\xi\mapsto\rho^{\xi,n}$.
Given $\xi\in\exch$ and $n\geq1$ we define $\rho^{\xi,n}$ as follows.
Draw $\vX^\xi=(\vX_{i,j}^\xi)_{i,j\geq1}\in\Omega^{\NN\times\NN}$ from $\xi$.
Then define a probability distribution on $\Omega^n$ by letting
\begin{align}\label{eqMU*xin}
\MU_*^{\xi,n}(\sigma)&=
\frac1n\sum_{i=1}^n\prod_{j=1}^n\vecone\cbc{\sigma_j=\vX_{i,j}^\xi}
&&(\sigma\in\Omega^n).
\end{align}
Thus, $\MU_*^{\xi,n}$ is the empirical distribution of the rows of the top-left $n\times n$ minor of $\vX^{\xi}$.
Finally, let $\MU^{\xi,n}=\dot\MU_*^{\xi,n}\in\law$ be the law induced by this discrete distribution
and let $\rho^{\xi,n}\in\cP(\law)$ be the distribution of $\MU^{\xi,n}$ (with respect to the choice of $\vX^\xi$).

\begin{lemma}\label{Lemma_empLimit}
For every $\xi\in\exch$ the limit $\rho^\xi=\lim_{n\to\infty}\rho^{\xi,n}$ exists and the map
$\xi\mapsto\rho^\xi$ is continuous.
\end{lemma}
\begin{proof}
Let $\xi\in\exch$.
Since $\cP(\law)$ is complete, to establish the {existence} of the limit we just need to prove that the sequence $(\rho^{\xi,n})_n$ is Cauchy.
\aco{To this end it suffices to verify the following condition:
\begin{align}\label{eqLemma_empLimit_0}
	\exists f:\NN\to\NN\ \forall \eps>0\ \exists n_0=n_0(\eps)>0\ \forall n>n_0,N>f(n):\cutmW(\rho^{\xi,n},\rho^{\xi,N})<\eps.
\end{align}
Indeed, if \eqref{eqLemma_empLimit_0} is satisfied, then there exists a subsequence $(\rho^{\xi,n_M})_M$ such that $\cutmW(\rho^{\xi,n_M},\rho^{\xi,n_{M+1}})<2^{-M}$ for all $M$.
In particular, the subsequence is Cauchy.
Because $\cP(\law)$ is complete, this subsequence thus has a limit $\rho^\star$, and \eqref{eqLemma_empLimit_0} ensures that the entire sequence $(\rho^{\xi,n})_n$ converges to $\rho^\star$ as well.}

To verify \eqref{eqLemma_empLimit_0} let $\eps>0$, pick a large $n=n(\eps)>0$ and choose $N=N(n)$ large enough.
We aim to prove that 
\begin{align}\label{eqLemma_empLimit_1}
	\cutmW(\rho^{\xi,n},\rho^{\xi,N})&<\eps.
\end{align}
To this end, we couple $\rho^{\xi,n},\rho^{\xi,N}$ by drawing $\vX^\xi\in\Omega^{\NN\times\NN}$ from $\xi$ and letting $g$ be the distribution of the pair $(\MU^{\xi,n},\MU^{\xi,N})$.
By definition of the Wasserstein metric, to establish \eqref{eqLemma_empLimit_1} it suffices to show that 
\begin{align}\label{eqLemma_empLimit_2}
	\Erw\brk{\cutm(\MU^{\xi,n},\MU^{\xi,N})}&<\eps
\end{align}
But \eqref{eqLemma_empLimit_2} follows from \Thm~\ref{Thm_sampling}.
Indeed, the construction~\eqref{eqMU*xin} ensures that $\MU^{\xi,n}$ is the empirical distribution of rows of the upper left $n\times n$-block of $\vX^\xi$, while $\MU^{\xi,N}$ is the empirical distribution of the rows of the $N\times N$-upper left block.
Due to the exchangeability of $\xi$, the distribution of the upper left $n\times n$-block is identical to the distribution of a random $n\times n$-minor of the matrix $\vX^\xi$.
\aco{Therefore, assuming that $N\gg n^2$ so that upon sub-sampling $n$ out of $N$ indices no index is chosen twice, in the notation of \Thm~\ref{Thm_sampling} we have $\dTV(\MU^{\xi,n},\MU^{\xi,N}_{n, n})<\eps/2$, whence we obtain \eqref{eqLemma_empLimit_2} and thus \eqref{eqLemma_empLimit_1}.}
Hence, the limit $\rho^\xi=\lim_{n\to\infty}\rho^{\xi,n}$ exists for all $\xi$.

To show continuity fix $\eps>0$ and let $\xi,\eta\in\exch$.
Due to~\eqref{eqLemma_empLimit_2} there exists $n=n(\eps)>0$ independent of $\xi,\eta$ such that
\begin{align}\label{eqLemma_empLimit_3}
	\cutmW(\rho^\xi,\rho^{\xi,n})&<\eps/4,&\cutmW(\rho^\eta,\rho^{\eta,n})&<\eps/4.
\end{align}
Since $\exch$ is equipped with the weak topology, any $\xi>0$ admits a neighbourhood $U$ such that {for all $\eta \in U$},
\begin{align*}
	\sum_{X\in\Omega^{n\times n}}\abs{\pr\brk{\forall i,j\in[n]:\vX^\xi_{i,j}=X_{i,j}}-\pr\brk{\forall i,j\in[n]:\vX^\eta_{i,j}=X_{i,j}}}&<\eps/8.
\end{align*}
Hence, the upper left $n\times n$-corners of $\vX^\xi,\vX^\eta$ have total variation distance at most $\eps/4$.
In effect, there is a coupling of $(\vX_{i,j}^\xi)_{i,j\in[n]},(\vX_{i,j}^\eta)_{i,j\in[n]}$ under which both these
random $n\times n$-matrices coincide with probability at least $1-\eps/4$.
Clearly, this coupling extends to a coupling of the measures $\MU^{\xi,n}$, $\MU^{\eta,n}$ such that
$\Erw[\cutm(\MU^{\xi,n},\MU^{\eta,n})]\leq\eps/4$.
Consequently, $\cutmW(\rho^{\xi,n},\rho^{\eta,n})\leq\eps/4$.
Combining this bound with \eqref{eqLemma_empLimit_3}, we conclude that
$\cutmW(\rho^{\xi},\rho^{\eta})<\eps$ for all $\eta\in U$, whence $\xi\mapsto\rho^\xi$ is continuous.
\end{proof}

\noindent
As a next step we are going to embed the space $\law$ into $\Exch{}$.
For a given law $\mu\in\law$ let $\xi^\mu$ be the distribution of $\vX^\mu$.

\aco{\begin{lemma}\label{Lemma_empsurj}
	The map $\mu\mapsto\xi^\mu$ is continuous and $\rho^{\xi^\mu}=\delta_\mu$.
\end{lemma}
\begin{proof}
Due to \Thm~\ref{Prop_kernel} it suffices to show that the map $\kappa\in\kernel\mapsto\xi^\kappa$, where $\xi^\kappa$ is the distribution of $\vX^\kappa$, is continuous.
	Combining \Thm s~\ref{Thm_oplus} and~\ref{Thm_tensor}, we conclude that the map
	$\kernel_\Omega\to\kernel_{\Omega^{[n]\times[n]}}$, $\kappa\mapsto(\kappa^{\oplus n})^{\tensor n}$ is continuous, where we iterate the $\oplus$ and the $\tensor$ operations $n$ times.
	For the sake of clarity, let us spell out the precise meaning of iterating these operations.
		The definition of the $\oplus$-operation extends to kernels $\kappa,\kappa'$ that take values in $\cP(\Omega)$, $\cP(\Omega')$ for different sets $\Omega,\Omega'$ by simply viewing $\kappa,\kappa'$ as $\Omega\cup\Omega'$-kernels.
		With this extension it makes sense to iterate the $\oplus$-operation; notice that $\kappa^{\oplus h}$ takes values in $\Omega^h$.
		We define $\kappa^{\tensor h}$ analogously.
	Finally, combining these two operations we obtain $\kappa^{\oplus n\,\tensor n}=(\kappa^{\oplus n})^{\tensor n}$, which is an $\Omega^{n\times n}$-kernel.
	Furthermore, for any $\sigma\in\Omega^{[n]\times[n]}$ the map
	\begin{align*}
		\kernel_{\Omega^{[n]\times[n]}}&\to[0,1],&
		k\mapsto\int_0^1\int_0^1k(\sigma)\dd s\dd x
	\end{align*}
	is continuous by the definition of the cut metric.
	Therefore, being a concatenations of continuous maps, the functions
	\begin{align*}
		\cT_\sigma:\kernel{}&\to[0,1],&\kappa\mapsto\int_0^1\int_0^1\kappa^{\oplus n\tensor n}_{s,x}(\sigma)\dd s\dd x
	\end{align*}
	are continuous as well.
	Since $\exch$ carries the weak toplogy, the continuity of the maps $\cT_\sigma$ implies the continuity of the map $\kappa\mapsto\xi^\kappa$.
\end{proof}
}
%

\begin{corollary}\label{Cor_empsurj}
The map $\exch\to\cP(\law)$, $\xi\mapsto\rho^\xi$ is surjective.
\end{corollary}
\begin{proof}
Suppose that $\fp\in\cP(\law)$.
With $\nu\mapsto\xi^\nu$ the measurable map from \Lem~\ref{Lemma_empsurj}, we define
$\xi^{\fp}=\int_{\law}\delta_{\xi^\mu}\dd\fp(\mu)$.
Then $\rho^{\xi^{\fp}}=\fp$. 
\end{proof}

\begin{corollary}\label{Cor_comp}
The space $\law$ is compact.
\end{corollary}
\begin{proof}
The space $\fX$ is compact as it is the space of probability measures on the compact Polish space $\Omega^{\NN\times\NN}$.
Since \Lem~\ref{Lemma_empLimit} and \Cor~\ref{Cor_empsurj} render a continuous surjective map
$\fX\to\cP(\law)$ and a continuous image of a compact space is compact, the space $\cP(\law)$ is compact.
To finally conclude that $\law$ is compact as well, consider a sequence $(\mu_n)_{n\geq1}$ in $\law$.
Because $\cP(\law)$ is compact, the sequence $(\delta_{\mu_n})_{n\geq1}$ possesses a convergent subsequence $(n_\ell)_{\ell\geq1}$.
Let $\pi$ be the limit of that subsequence.
Consider a point $\nu$ in the support of $\pi$ and let $(U_k)_{k\geq0}$ be a sequence of open neighbourhoods of $\nu$ such that \aco{$U_{k+1}\subset U_k$} for all $k$ and $\bigcap_{k\geq1}U_k=\{\nu\}$.
By Urysohn's lemma there are continuous {functions} $f_k:\law\to[0,1]$ such that $f_k$ takes the value one on $U_k$ and the value $0$ outside $U_{k-1}$.
Now, for all $k\geq1$ we have
\begin{align*}
	0<\int f_k\dd\pi=\lim_{\ell\to\infty}\int f_k\dd\mu_{n_\ell}\leq\lim_{\ell\to\infty}\vecone\{\mu_{n_\ell}\in U_{k-1}\}.
\end{align*}
Hence, $\mu_{n_\ell}\in U_{k-1}$ for almost all $\ell$.
Consequently, $\nu=\lim_{\ell\to\infty}\mu_{n_\ell}$.
Thus, the metric space $\law$ is sequentially compact and therefore compact.
\end{proof}

\noindent
\aco{Of course the second part of the proof above merely establishes the well known fact that the mapping $\law\to\cP(\law)$, $\mu\mapsto\delta_\mu$ is a homeomorphic embedding onto a closed subspace.
We included the brief argument for the sake of completeness.}

\begin{proof}[Proof of \Thm \ref{Prop_Polish}]
The theorem follows from Corollaries~\ref{Cor_complete2} and~\ref{Cor_comp}.
\end{proof}

\subsection{Proof of \Thm~\ref{Prop_metric_ext}}
Let $\mu,\nu\in\Law$.
Toward the proof of \eqref{eqProp_metric_ext1} let
\begin{align}\nonumber
X^+(\omega)&=\cbc{x\in[0,1]:\int_\conf \sigma_{x}(\omega)\dd\mu(\sigma) -\int_\conf  \sigma_x(\omega)\dd\nu(\sigma) \geq 0},&
X^-(\omega)&=[0,1]\setminus X^+.
\end{align}
Since $\mu,\nu$ are atoms concentrated on the pure state \eqref{eqbarmu}, respectively, we obtain
\begin{align}\label{eqProp_metric_ext_1}
\cutms(\bar\mu,\bar\nu)
&=\max_{\omega\in\Omega}\abs{\int_{X^+(\omega)}\int_\conf \sigma_{x}(\omega)\dd\mu(\sigma) -\int_\conf  \sigma_x(\omega)\dd\nu(\sigma)}
\vee
\abs{\int_{X^-(\omega)}\int_\conf \sigma_{x}(\omega)\dd\mu(\sigma) -\int_\conf  \sigma_x(\omega)\dd\nu(\sigma)}
\\
&=\max_{\omega\in\Omega}\abs{\int_{X^+(\omega)}\int_{\conf\times\conf}
\bc{	\sigma_{x}(\omega)-\tau_x(\omega)}\dd(\mu\tensor\nu)(\sigma,\tau)}
\vee\abs{\int_{X^-(\omega)}\int_{\conf\times\conf}
\bc{	\sigma_{x}(\omega)-\tau_x(\omega)}\dd(\mu\tensor\nu)(\sigma,\tau)}\nonumber
\leq\cutms(\mu,\nu),
\end{align}
whence \eqref{eqProp_metric_ext1} is immediate. 
Moreover, the first part of \eqref{eqProp_metric_ext2} follows from \eqref{eqProp_metric_ext_1},
while the second part is immediate from the triangle inequality.

\subsection{Proof of \Thm~\ref{Thm_exch}}
The product topology on $\Omega^{\NN\times\NN}$ is the weakest topology under which all the functions
\begin{align*}
T_\sigma&:\Omega^{\NN\times\NN}\to\cbc{0,1},&
(X_{i,j})_{i,j\geq1}&\mapsto\prod_{i,j=1}^n\vecone\cbc{X_{i,j}=\sigma_{i,j}}
&\bc{n\geq1,\sigma\in\Omega^{[n]\times[n]}}.
\end{align*}
are continuous.
Equivalently, the product topology is induced by the metric
\begin{align}\label{eqDmax}
D_{\max}:\Omega^{\NN\times\NN}\times \Omega^{\NN\times\NN}&\to[0,1],&
(X,Y)\mapsto2^{-\max\cbc{n\geq0:\forall i,j\leq n:X_{i,j}=Y_{i,j}}}.
\end{align}
Hence, the weak topology on $\Exch{}\subset\cP(\Omega^{\NN\times\NN})$ is induced by the corresponding Wasserstein metric $\cD_{\max}(\nix,\nix)$.

\aco{As a first step we are going to show that the map $\pi\mapsto\Xi^\pi$ is $(\cutmW,\cD_{\max})$-continuous.
Indeed, assume that $\pi,\pi'\in\cP(\kernel)$ satisfy $\cutmW(\pi,\pi')<\delta$ for a small enough $\delta=\delta(\eps)>0$.
Then the corresponding coupling shows together with \Lem~\ref{Lemma_empsurj} and \Thm~\ref{Prop_kernel} that $\cD_{\max}(\nix,\nix)<\eps$.
Furthermore, $\pi\mapsto\Xi^\pi$ is one-to-one because it can be inverted via \Cor~\ref{Cor_empsurj}.
Moreover, \Lem~\ref{Lemma_empLimit} implies that the map $\pi\mapsto\Xi^\pi$ is surjective, as the inverse image of $\xi\in\exch$ is just $\rho^\xi\in\cP(\law)\cong\cP(\kernel)$.
Thus, we know that $\cP(\kernel)\to\Exch$, $\pi\mapsto\Xi^\pi$ is a continuous bijection.
Finally, since $\cP(\kernel)$ is compact and the continuous image of a compact set is compact, the map 
$\pi\mapsto\Xi^\pi$ is open and thus a homeomorphism.
\qed
}

\subsection{Proof of \Thm~\ref{Thm_rl_lower}}
For a bipartite graph $G = (U, V, E)$ with $|U| = |V| = n$, and a partition $P = (S_1...S_l, V_1...V_k)$, denote by $G^P$ the weighted bipartite graph on vertex set $([l], [k])$ s.t. the weight of edge $ij$ is given by $\dd(S_i, V_j)$. 	
\begin{theorem}[\cite{ConlonFox}, Theorem 7.1] \label{Thm_CFLowerBound}
There exists $\eps > 0, n \in \NN$ and a bipartite graph $G = (U, V, E)$ with $|U| = |V| = n$ s.t. every partition $P = (S_1...S_l, V_1...V_k)$ of $(U, V)$ \mhk{satisfying $\cutmgr(G, G^P) \leq \eps$} requires at least $l = \exp \left( \Theta(\eps^{-2}) \right)$ parts, independently of $k$.
\end{theorem} 

\begin{proof}[Proof of \Thm~\ref{Thm_rl_lower}]
Let $G$ be a graph given by the previous theorem and let $\kappa_G$ be the corresponding graphon. Denote by $\kappa$ a kernel consisting of $\kappa_G$ and its transposed graphon given by \eqref{eqkappaomega} in the special case $\Omega = \lbrace 0,1 \rbrace$. Denote by $\mu = \mu(G) \in \Law$ the corresponding law given by \Thm~\ref{Prop_kernel}. Assume there is $\nu \in \Law$ with support of size less then $l = \exp \left( \Theta(\eps^{-2}) \right)$ and $\cutm(\mu, \nu) < \frac{\eps}{2}.$ Then $\nu$ induces a partition $K$ of $[0,1]$ into at most $l$ parts s.t. $\cutm(\kappa, \kappa^\nu) = \cutm(\kappa, \kappa^K) \leq \frac{\eps}{2}$ which implies that there is a partition $S$ and a graphon $\kappa^S$ s.t. $\cutmgr(\kappa_G, \kappa^S) \leq \eps$. As $\kappa_G$ and $\kappa_S$ are by definition embeddings of (finite) graphs into the space of graphons, this is a contradiction to \Thm~\ref{Thm_CFLowerBound}.    
\end{proof}

\section{The pinning operation}\label{Sec_Extremality}

\noindent
In this section we prove \Thm~\ref{Thm_pinning}.
We begin by investigating a discrete version of the pinning operation, which played a key role in recent work on random factor graphs~\cite{AcoPinningPaper}.
The discrete version of the pinning theorem, \Thm~\ref{Thm_pin} below, was already established as~\cite[\Lem~3.5]{AcoPinningPaper}.
In \Sec~\ref{Sec_discPinning} we give a shorter proof, based on an argument form~\cite{Raghavendra}.
Moreover, in \Sec~\ref{Sec_pinningCont} we show by a somewhat delicate argument that the pinning operation is continuous with respect to the cut metric.
Finally, in \Sec~\ref{Sec_finish} we complete the proof of \Thm~\ref{Thm_pinning}.

\subsection{Discrete pinning}\label{Sec_discPinning}
For a probability measure $\mu\in\Law_n$ and a set $I\subset[n]$ we denote by $\mu_I$ the joint distribution of the coordinates $i\in I$.
Thus, $\mu_I$ is the probability distribution on $\Omega^I$ defined by
\begin{align*}
\mu_I(\sigma)&=\sum_{\tau\in\Omega^n}\vecone\cbc{\forall i\in I:\tau_i=\sigma_i}\mu(\tau).
\end{align*}
Where $I=\{i_1,\ldots,i_t\}$ is given explicitly, we use the shorthand $\mu_I=\mu_{i_1,\ldots,i_t}$.

\begin{theorem}\label{Thm_pin}
For every $\eps>0$ for all large enough $n$ and all $\mu\in\Law_n$ the following is true.
Draw and integer $0\leq\THETA\leq \lceil \log|\Omega|/\eps^2\rceil$ uniformly random and let $\vI\subset[n]$ be a random set of size $\THETA$.
Additionally, draw $\hat\SIGMA$ from $\mu$ independently of $\THETA,\vI$.
Let 
\begin{align}\label{eqhatmu}
	\hat\mu=\mu[\nix\mid\{\sigma\in\Omega^n:\forall i\in\vI:\sigma_i=\hat\SIGMA_i\}].
\end{align}
Then
\begin{align}\label{eqThm_pin}
	\sum_{1\leq i<j\leq n}\Erw\TV{\hat\mu_{i,j}-\hat\mu_{i}\tensor\hat\mu_{j}}&\leq\eps n^2.
\end{align}
\end{theorem}

Apart from~\cite[\Lem~3.5]{AcoPinningPaper},  statements related to \Thm~\ref{Thm_pin} were previously obtained by Montanari~\cite{Montanari} and Raghavendra and Tan~\cite{Raghavendra}.
To be precise, \cite[\Thm~2.2]{Montanari} deals with the special case of the discrete pinning operation for graphical channels and the number $\THETA$ of pinned coordinates scales linearly with the dimension $n$.
The original proof of \Thm~\ref{Thm_pin} in~\cite{AcoPinningPaper} was based on a generalisation of Montanari's argument.
Moreover, \cite[\Lem~4.5]{Raghavendra} asserted the existence of $T=T(\mu,\eps)>0$ such that 
$$\sum_{1\leq i<j\leq n}\Erw\brk{\TV{\hat\mu_{i,j}-\hat\mu_{i}\tensor\hat\mu_{j}}\mid\THETA=T}\leq\eps n^2,$$
rather than showing that a random $\THETA$ does the trick.
But at second glance the proof given in \cite{Raghavendra}, which is significantly simpler than the one from \cite{AcoPinningPaper}, actually implies \Thm~\ref{Thm_pin}.

For completeness we include the short proof of \Thm~\ref{Thm_pin} via the argument from \cite{Raghavendra}.
We need a few concepts from information theory.
Let $X,Y,Z$ be random variables that take values in finite domains.
We recall that the {\em conditional mutual information} of $X,Y$ given $Z$ is defined as
\begin{align*}
\cI(X,Y\mid Z)&=\sum_{x,y,z}\pr\brk{X=x,\,Y=y,\,Z=z}\log\frac{\pr\brk{X=x,\,Y=y\mid Z=z}}{\pr\brk{X=x\mid Z=z}\pr\brk{Y=y\mid Z=z}},
\end{align*}
with the conventions $0\log 0=0$, $0\log\frac{0}{0}=0$ and with the sum ranging over all possible values $x,y,z$ of $X,Y,Z$, respectively.
Moreover, the {\em conditional entropy} of $X$ given $Y$ reads
\begin{align*}
\cH(X\mid Y)&=\sum_{x,y}\pr\brk{X=x,\,Y=y}\log\pr\brk{X=x\mid Y=y}.
\end{align*}
We  also recall the basic identity
\begin{align}\label{eqMutInf}
\cI(X,Y\mid Z)&=\cH(X\mid Z)-\cH(X\mid Y,Z).
\end{align}
Finally, {\em Pinsker's inequality} provides that for any two probability distribution $\mu,\nu$ on a finite set $\cX$,
\begin{align}\label{eqPinsker}
\dTV(\mu,\nu)&\leq\sqrt{\KL{\mu}\nu/2},\qquad\mbox{where}&
\KL{\mu}\nu&=\sum_{x\in\cX}\mu(x)\log\frac{\mu(x)}{\nu(x)}
\end{align}
signifies the Kullback-Leibler divergence.
The proof of the following lemma is essentially identical to the proof of \cite[\Lem~4.5]{Raghavendra}.

\begin{lemma}\label{Lemma_mutInf}
Let $\mu\in\cP(\Omega^n)$ and let $\SIGMA\in\Omega^n$ be a sample drawn from $\mu$.
Let $\vi,\vi',\vi_1,\ldots\in[n]$ be uniformly distributed and mutually independent as well as independent of $\SIGMA$.
Then for any integer $T$ we have
$$\sum_{\theta=0}^T\cI(\SIGMA_{\vi},\SIGMA_{\vi'}\mid \vi,\vi',\vi_1,\ldots,\vi_\theta,\SIGMA_{\vi_1},\ldots,\SIGMA_{\vi_\theta})\leq\log \abs{\Omega}.$$
\end{lemma}
\begin{proof}
Due to \eqref{eqMutInf}, for every $\theta\geq0$,
\begin{align*}
	\cI(\SIGMA_{\vi},\SIGMA_{\vi'}\mid \vi,\vi',\vi_1,\ldots,\vi_\theta,\SIGMA_{\vi_1},\ldots,\SIGMA_{\vi_\theta})
	&=\cH(\SIGMA_{\vi}\mid \vi,\vi',\vi_1,\ldots,\vi_\theta,\SIGMA_{\vi_1},\ldots,\SIGMA_{\vi_\theta})-
	\cH(\SIGMA_{\vi}\mid\vi,\vi',\vi_1,\ldots,\vi_\theta,\SIGMA_{\vi_1},\ldots,\SIGMA_{\vi_\theta},\SIGMA_{\vi'})\\
	&
	=\cH(\SIGMA_{\vi}\mid\vi,\vi_1,\ldots,\vi_\theta,\SIGMA_{\vi_1},\ldots,\SIGMA_{\vi_\theta})
	-\cH(\SIGMA_{\vi}\mid\vi,\vi_1,\ldots,\vi_\theta,\vi_{\theta+1},\SIGMA_{\vi_1},\ldots,\SIGMA_{\vi_{\theta+1}}).
\end{align*}
Summing on $\theta=1,\ldots,T$, we obtain
\begin{align*}
	\sum_{\theta=0}^T\cI(\SIGMA_{\vi},\SIGMA_{\vi'}\mid\vi,\vi',\vi_1,\ldots,\vi_\theta,\SIGMA_{\vi_1},\ldots,\SIGMA_{\vi_\theta})
	&=\cH(\SIGMA_{\vi}\mid\vi)-\cH(\SIGMA_{\vi}\mid\vi,\vi_1,\ldots,\vi_{T+1},\SIGMA_{\vi_1},\ldots,\SIGMA_{\vi_{T+1}}).
\end{align*}
The desired bound follows because $\cH(\SIGMA_i)\leq\log |\Omega|$ and $\cH(\SIGMA_i\mid\SIGMA_{\vi_1},\ldots,\SIGMA_{\vi_{T+1}})\geq0$.
\end{proof}

Now let $T>0$ be an integer and draw $0\leq\THETA\leq T$ uniformly at random and construct $\hat\mu$ as in  \Thm~\ref{Thm_pin}.
Then as an immediate consequence of \Lem~\ref{Lemma_mutInf} we obtain the following bound, where, of course,  the expectation refers to the choice of $\hat\mu$ and the independently chosen and uniform $\vi,\vi'$.

\begin{corollary}\label{Cor_mutInf}
We have $\Erw\brk{\KL{\hat\mu_{\vi,\vi'}}{\hat\mu_{\vi}\tensor\hat\mu_{\vi'}}}\leq(\log \abs{\Omega})/T.$
\end{corollary}
\begin{proof}
Keeping the notation from \Lem~\ref{Lemma_mutInf}, we let $\vec I=(\vi,\vi',\vi_1,\ldots,\vi_\theta)$
and $\vec\Sigma=(\SIGMA_{\vi_1},\ldots,\SIGMA_{\vi_\theta})$.
Recalling the definition \eqref{eqhatmu} of $\hat\mu$,  we find
\begin{align*}
	\cI(\SIGMA_{\vi},\SIGMA_{\vi'}\mid \vec I,\vec S)
	&=\Erw\brk{\sum_{\omega,\omega'\in\Omega}
		\pr\brk{\SIGMA_{\vi}=\omega,\SIGMA_{\vi'}=\omega'\mid \vec I,\vec\Sigma}
		\log\frac{\pr\brk{\SIGMA_{\vi}=\omega,\SIGMA_{\vi'}=\omega'\mid \vec I,\vec\Sigma}}
		{\pr\brk{\SIGMA_{\vi}=\omega\mid \vec I,\vec S}\pr\brk{\SIGMA_{\vi'}=\omega'\mid \vec I,\vec\Sigma}}}\\
	&=\Erw\Bigg[\sum_{\sigma\in\Omega^n}\mu(\sigma)
	\sum_{\omega,\omega'\in\Omega}
	\mu\bc{\SIGMA_{\vi}=\omega,\SIGMA_{\vi'}=\omega'\mid
		\vec\Sigma=(\sigma_{\vi_1},\ldots,\sigma_{\vi_{\THETA}})}\\
	&\qquad\qquad\qquad\qquad\qquad
	\log\frac{
		\mu\bc{\SIGMA_{\vi}=\omega,\SIGMA_{\vi'}=\omega'\mid
			\vec\Sigma=(\sigma_{\vi_1},\ldots,\sigma_{\vi_{\THETA}})}}
	{\mu\bc{\SIGMA_{\vi}=\omega\mid
			\vec\Sigma=(\sigma_{\vi_1},\ldots,\sigma_{\vi_{\THETA}})}\mu\bc{\SIGMA_{\vi'}=\omega'\mid
			\vec\Sigma=(\sigma_{\vi_1},\ldots,\sigma_{\vi_{\THETA}})}}\Bigg]\\
	&=\Erw\brk{\KL{\hat\mu_{\vi,\vi'}}{\hat\mu_{\vi}\tensor\hat\mu_{\vi'}}}
\end{align*}
Hence, the assertion follows from \Lem~\ref{Lemma_mutInf}.
\end{proof}

\begin{proof}[Proof of \Thm~\ref{Thm_pin}]
Applying Pinsker's inequality \eqref{eqPinsker}, Jensen's inequality and \Cor~\ref{Cor_mutInf}, we find
\begin{align*}
	\Erw\TV{\hat\mu_{\vi,\vi'}-\hat\mu_{\vi}\tensor \hat\mu_{\vi'}}
	&\leq\Erw\sqrt{\KL{\hat\mu_{\vi,\vi'}}{\hat\mu_{\vi}\tensor\hat\mu_{\vi'}}/2}
	\leq\sqrt{\Erw\brk{\KL{\hat\mu_{\vi,\vi'}}{\hat\mu_{\vi}\tensor\hat\mu_{\vi'}}}/2}
	\leq\sqrt{\frac{\log \abs{\Omega}}{2T}},
\end{align*}
whence the desired bound follows if $T\geq (\log \abs{\Omega})/(2\eps^2)$.
\end{proof}

Finally, the following lemma clarifies the bearing that the bound \eqref{eqThm_pin} has on the cut metric. The lemma is an improved version of \cite[Lemma 2.9]{AcoPerkinsSymmetry}.
Following~\cite{Victor} we say that $\mu\in\Law_n$ is {\em $\eps$-symmetric} if
\begin{align*}
\sum_{1\leq i<i'\leq n}\TV{\mu_{i,i'}-\mu_{i}\tensor\mu_{i'}}&<\eps n^2.
\end{align*}

\begin{lemma}[]\label{Lemma_sym}
For any $\eps > 0$ and every finite set $\Omega$ there exists $n_0>0$ s.t.\ for every $n \geq n_0$ every $\eps^2/4$-symmetric $\mu \in \Law_n$ satisfies $\cutmw(\mu,\bigotimes_{i=1}^n\mu_{i})<\eps$.
\end{lemma}
\begin{proof}
Let $\delta = \eps^2/4$.
Since $\mu \tensor \bar{\mu}$ is a coupling of $\mu$ and $\bar{\mu}$
it suffices to show that for any set $I \subset [n]$ and every $\omega \in \Omega$,
\begin{align}
	\sup_{S \subset \Omega^{2n}}\abs{\sum_{(\sigma, \tau) \in S} \sum_{i \in I} \mu(\sigma) \bar{\mu}(\tau) \bc{\ind\cbc{\sigma_i = \omega} - \ind\cbc{\tau_i = \omega}} } \leq \eps n. \label{EqToShowCutdist}
\end{align}
Let $X(\sigma) = X(\sigma, I, \omega) = \sum_{i \in I} \ind\cbc{\sigma_i = \omega}$ and denote by $\bar{X}$ its expectation with respect to $\mu$, that is $\bar{X} = \langle X(\sigma), \mu \rangle$. 
Because $\mu$ is $\delta$-symmetric we can bound the second moment of $X$ as follows:
\begin{align*}
	\bck{X(\sigma)^2, \mu} &= \bck{ \sum_{i, j \in I} \ind\cbc{\sigma_i = \sigma_j = \omega} , \mu} = \sum_{i, j \in I} \mu_{ij}(\omega, \omega) 
	\leq \left( \sum_{i,j \in I: i \neq j} \mu_i(\omega)\mu_j(\omega) + \sum_{i \in I}\mu_i(\omega) \right) + \delta n^2 \leq \bar{X}(1 + \bar{X}) + \delta n^2,
\end{align*}
Hence,
\begin{align}
	\bck{X(\sigma)^2, \mu} - \bar{X}^2 \leq \bar{X} + \delta n^2 \leq \abs{I} + \delta n^2. \label{EqSecondMoment}
\end{align}
Let $\alpha \in (0,1)$ and $P(\alpha) = \mu \cbc{  \abs{X(\sigma) - \bar{X}} \geq \alpha n }$.
Then Chebyshev's inequality and \eqref{EqSecondMoment} yield $P(\alpha) \leq (|I| + \delta n^2) / (\alpha n)^2.$
Hence, for events $S_h = \left\lbrace \abs{ X(\sigma) - \bar{X}} \geq 2^h \eps n \right\rbrace$ we obtain $\mu(S_h) \leq P(2^h\eps)\leq 4^{-h} \delta / \eps^2+O(1/n). $	
Therefore,
\begin{align}
	\sup_{S \subset \Omega^{2n}}& \abs{ \sum_{i \in I} \sum_{(\sigma, \tau) \in S}  \mu(\sigma) \ind\cbc{\sigma_i = \omega} - \bar{\mu}(\tau) \ind\cbc{\tau_i = \omega} } 
								\leq \scal{\abs{X-\bar X}}{\mu}
								\leq  \sum_{h \geq 0} \mu(S_h) \cdot \eps 2^h \leq o(1)+\sum_{h \geq 0} 2^{-h} \delta / \eps = 2 \delta / \eps+o(1). \label{Eq_EndLemmaCalculation}  
\end{align}
Thus, \eqref{EqToShowCutdist} follows from \eqref{Eq_EndLemmaCalculation} and the choice of $\delta$.
\end{proof}
\color{black}

\subsection{Continuity}\label{Sec_pinningCont}

Recall that for a given $\mu\in\law$ the pinned $\mu_{\hat\SIGMA^\mu\pin n}\in\law$ is random.
Thus, for the pinned laws we consider the $\cutm$-Wasserstein metric.
The aim in this paragraph is to establish the following key statement.

\begin{proposition} \label{Prop_ContinuityOfPinning}
The operator $\mu\mapsto\mu_{\hat\SIGMA^\mu\downarrow n}$ is $(\cutm(\nix,\nix),\cutmW(\nix,\nix))$-continuous for any $n\geq1$.
\end{proposition}

Toward the proof of \Prop~\ref{Prop_ContinuityOfPinning} we need to consider a slightly generalised version of the pinning operation.
Specifically, for a measurable map $\kappa:[0,1]^2\to[0,1]^\Omega$ and $\tau\in\Omega^n$ let
\begin{align*}
\vz_{\tau}\bc\kappa&=\int_0^1\prod_{i=1}^n\kappa_{s,\hat\vx_i}(\tau_i)\dd s.
\end{align*}
Thus, $\vz_{\tau}$ is a random variable, dependent on the uniformly and independently chosen $\hat\vx_1,\ldots,\hat\vx_n\in[0,1]$.
Also let $\vz(\kappa)=\sum_{\tau\in\Omega^n}\vz_{\tau}\bc\kappa$.
Further, define $\kappa_{\tau\pin n}\in\Kernel_1$ as follows.
If $\vz_\tau(\kappa)=0$, then we let $\kappa_{\tau\pin n}=\kappa$.
But if $\vz_\tau(\kappa)>0$, then we let $\kappa_{\tau\pin n}$ be a kernel representation of the probability distribution
\begin{align*}
\int_0^1\frac{\prod_{i=1}^n\kappa_{s,\hat\vx_i}(\tau_i)}{\vz_\tau(\kappa)}\delta_{\kappa_s}\dd s
\in\cP(\conf_1)
\end{align*}
Additionally, let $\hat\SIGMA^\kappa\in\Omega^n$ denote a vector drawn from the distribution $(\vz_\tau(\kappa)/\vz(\kappa))_{\tau\in\Omega^n}$ if $\vz(\kappa)>0$, and let $\hat\SIGMA^\kappa\in\Omega^n$ be uniformly distributed otherwise.

\begin{lemma}\label{Lemma_L1cont}
For any $n\geq1$, $\eps>0$ there is $\delta>0$ such that for all $\kappa\in\Kernel{}$ and
all $\kappa'\in\Kernel_1$ with $D_1(\kappa,\kappa')<\delta$ we have
$$\cutmW\bc{\kappa_{\hat\SIGMA^\kappa\pin n},\kappa_{\hat\SIGMA^{\kappa'}\pin n}'}<\eps.$$
\end{lemma}
Toward the proof of \Lem~\ref{Lemma_L1cont} we require the following statement.
\begin{lemma}\label{Lemma_pinz}
For any $n\geq1$, $\eps>0$ and $\kappa\in\Kernel{}$ we have $\pr\brk{\vz_{\hat\SIGMA^\kappa}(\kappa)<\eps\mid\hat\vx_1,\ldots,\hat\vx_n}<\eps\abs\Omega^n.$
\end{lemma}
\begin{proof}
We have $\pr\brk{\vz_{\hat\SIGMA}(\kappa)<\eps\mid\hat\vx_1,\ldots,\hat\vx_n}=\sum_{\tau\in\Omega^n}\vecone\cbc{\vz_{\tau}(\kappa)<\eps}\vz_{\tau}(\kappa)<\eps |\Omega|^n.$
\end{proof}
\begin{proof}[Proof of \Lem~\ref{Lemma_L1cont}]
Given $\eps>0$ pick small enough $\eta=\eta(\eps,n)>0$, $\delta=\delta(\eta)>0$. 
Consider $\kappa\in\Kernel,\kappa'\in\Kernel_1$ such that $D_1(\kappa,\kappa')<\delta$ and let $\mu=\mu^\kappa$, $\mu'=\mu^{\kappa'}$.
Then we see that
\begin{align}\label{eqLemma_L1cont_100}
	\pr\brk{1-\eta<\vz(\kappa')<1+\eta}&>1-\eta.
\end{align}
Hence, in the following we may condition on the event that $1-\eta<\vz(\kappa')<1+\eta$.
Given that this is so, choose $\hat\SIGMA,\hat\SIGMA'\in\Omega^n$ from the distributions
\begin{align*}
	\pr\brk{\hat\SIGMA=\sigma\mid\hat\vx_1,\ldots,\hat\vx_n}&=\vz_\sigma(\kappa)/\vz(\kappa)=\vz_\sigma(\kappa),&\pr\brk{\hat\SIGMA'=\sigma\mid\hat\vx_1,\ldots,\hat\vx_n}&=\vz_\sigma(\kappa')/\vz(\kappa')&(\sigma\in\Omega^n).
\end{align*}
Further, define the probability density functions
\begin{align*}
	p_\kappa(s)&=\frac1{\vz_{\hat\SIGMA}(\kappa)}\prod_{i=1}^n\kappa_{s,\hat\vx_i}(\hat\SIGMA_i),&
	p_{\kappa'}(s)&=\frac1{\vz_{\hat\SIGMA'}(\kappa')}\prod_{i=1}^n\kappa'_{s,\hat\vx_i}(\hat\SIGMA_i')&&\mbox{ and set}\\
	\hat p(s)&=p(s)\wedge p'(s),&\hat p_\kappa(s)&=p_\kappa(s)-\hat p(s),&\hat p_{\kappa'}(s)&=p_{\kappa'}(s)-\hat p(s)
\end{align*}
so that
\begin{align*}
	\mu_{\hat\SIGMA\pin n}&=\int_0^1 p_\kappa(s)\delta_{\kappa_s}\dd s,&
	\mu'_{\hat\SIGMA'\pin n}&=\int_0^1p_{\kappa'}(s)\delta_{\kappa'_s}\dd s.
\end{align*}
To couple $\mu_{\hat\SIGMA\pin n},\mu_{\hat\SIGMA'\pin n}'$ draw a pair $(\vt,\vt')\in[0,1]^2$ from the following distribution:
with probability $\int_0^1\hat p(s)\dd s$, we draw $\vt=\vt'$ from the distribution
$(\int_0^1\hat p(s)\dd s)^{-1}\hat p(s)\dd s$, and with probability $1-\int_0^1\hat p(s)\dd s$ we draw
$\vt,\vt'$ independently from the distributions
\begin{align*}
	\bc{1-\int_0^1\hat p(s)\dd s}^{-1}\hat p_\kappa(s)\dd s,&&
	\bc{1-\int_0^1\hat p(s)\dd s}^{-1}\hat p_\kappa'(s)\dd s,
\end{align*}
respectively.
Then $(\kappa_{\vt},\kappa_{\vt'}')$ provides a coupling of $\mu_{\hat\SIGMA\pin n},\mu'_{\hat\SIGMA'\pin n}$.
Consequently,
\begin{align}\label{eqLemma_L1cont_6}
	\cutmW(\mu_{\hat\SIGMA\pin n},\mu'_{\hat\SIGMA'\pin n})&\leq D_1(\kappa,\kappa')+\pr\brk{\vt\neq\vt'}
	+\pr\brk{\vz(\kappa')\not\in(1-\eta,1+\eta)}<\delta+\pr\brk{\vt\neq\vt'}+\pr\brk{\vz(\kappa')\not\in(1-\eta,1+\eta)}.
\end{align}

To estimate $\pr\brk{\vt\neq\vt'}$ let 
\begin{align*}
	\cE&=\cbc{\sum_{\tau\in\Omega^n}\int_0^1\abs{\prod_{i=1}^n\kappa_{s,\hat\vx_i}(\tau_i)-\prod_{i=1}^n\kappa'_{s,\hat\vx_i}(\tau_i)}\dd s<\eta^2}.
\end{align*}
Picking $\delta$ sufficiently small ensures that
\begin{align}\label{eqLemma_L1cont_1}
	\pr\brk\cE&>1-\eta
\end{align}
and on the event $\cE$ we have
\begin{align*}
	\dTV\bc{\hat\SIGMA,\hat\SIGMA'}&=\frac12\sum_{\sigma\in\Omega^n}\abs{\pr\brk{\hat\SIGMA=\sigma}-\pr\brk{\hat\SIGMA'=\sigma}}=\sum_{\sigma\in\Omega^n}\abs{\vz_\sigma(\kappa)-\vz_\sigma(\kappa')/\vz(\kappa)}<\eta.
\end{align*}
Hence, on $\cE$ we can couple $\hat\SIGMA,\hat\SIGMA'$ such that
\begin{align}\label{eqLemma_L1cont_4}
	\pr[\hat\SIGMA\neq\hat\SIGMA']<\eta.
\end{align}
Additionally, let $\cE'=\cbc{\hat\SIGMA=\hat\SIGMA',\,\vz_{\hat\SIGMA}(\kappa)\geq \eta^{1/3}}$.
Then \Lem~\ref{Lemma_pinz}, \eqref{eqLemma_L1cont_1} and \eqref{eqLemma_L1cont_4} imply that
\begin{align}\label{eqLemma_L1cont_2} 
	\pr\brk{\cE'\mid\cE}&\geq 1-2\eta^{1/3} \abs\Omega^n.
\end{align}
Moreover, on $\cE\cap\cE'$ we have
\begin{align*}%
	\abs{\vz_{\hat\SIGMA}(\kappa')-\vz_{\hat\SIGMA}(\kappa)}&\leq \eta
\end{align*}
and consequently
\begin{align}\nonumber
	\pr\brk{\vt\neq\vt'\mid\cE\cap\cE'}&=
	1-\int_0^1\hat p(s)\dd s=1-\frac12\int_0^1\bc{p(s)+p'(s)-\abs{p(s)-p'(s)}}\dd s=\frac12\int_0^1\abs{p(s)-p'(s)}\dd s\\
	&\leq\frac1{2\vz_{\hat\SIGMA}(\kappa)}\int_0^1
	\abs{\prod_{i=1}^n\kappa_{s,\hat\vx_i}(\hat\SIGMA_i)-\prod_{i=1}^n\kappa'_{s,\hat\vx_i}(\hat\SIGMA_i)}\dd s
	+\frac{\vz_{\hat\SIGMA}(\kappa)-\vz_{\hat\SIGMA}(\kappa')}{2\vz_{\hat\SIGMA}(\kappa)\vz_{\hat\SIGMA}(\kappa')}
	\leq\sqrt\eta. \label{eqLemma_L1cont_3}
\end{align}
Finally, the assertion follows from \eqref{eqLemma_L1cont_6}, \eqref{eqLemma_L1cont_1}, \eqref{eqLemma_L1cont_2} and \eqref{eqLemma_L1cont_3}.
\end{proof}

\begin{lemma}\label{Lemma_cutmcont}
For any $\eps>0$, $\ell\geq1$ there is $\delta>0$ such that for all $\kappa\in\Kernel{}$
such that $\mu^\kappa\in\Law$ is supported on a set of size at most  $\ell$ and all $\iota\in\Kernel_1$ with $\cutmFK(\kappa,\iota)<\delta$ we have $\cutmW(\kappa_{\hat\SIGMA^\kappa\pin n},\iota_{\hat\SIGMA^\iota\pin n})<\eps$.
\end{lemma}
\begin{proof}
Pick $\alpha=\alpha(\eps,\ell,n)$, $\beta=\beta(\alpha)$, $\xi=\xi(\beta)$, $\zeta=\zeta(\xi)$, $\eta=\eta(\zeta)>0$ and $\delta=\delta(\eta)>0$ sufficiently small.
To summarise,
\begin{align}\label{eqGreeks}
	0<\delta\ll\eta\ll\zeta\ll\xi\ll\beta\ll\alpha\ll \eps/(n+\ell).
\end{align}
We may assume that there is a partition $S_1,\ldots,S_\ell$ of $[0,1]$ such that $\kappa$ is constant on $S_i\times\cbc x$ for all $x\in\brk{0,1}$.
Moreover, we may assume without loss that there is $k\in[\ell]$ such that $\lambda(S_i)>\eta$ for all $i\leq k$, while $\lambda(S_i)<\eta$ for all $i>k$.
Let $t_i:[0,1]\to S_i$ be a measurable bijection that maps the Lebesgue measure on $[0,1]$ to the probability measure $\lambda(S_i)^{-1}\dd s$ on $S_i$ for $i\leq k$ \aco{(see \Lem~\ref{Lemma_Lambda})}.
Assuming that $\delta$ is small enough, we see that the kernels
\begin{align*}
	\kappa^{(i)}_{s,x}&=\kappa_{t_i(s),x},&\iota^{\bc i}_{s,x}&=\iota_{t_i(s),x}
\end{align*}
have cut distance
\begin{align}\label{eqLemma_cutmcont1}
	\cutmFK(\kappa^{(i)},\iota^{(i)})&<\zeta&&\mbox{ for all $i\leq k$.}
\end{align}
Combining \Prop~\ref{Prop_oplus} and \eqref{eqLemma_cutmcont1}, we conclude that after an $n$-fold application of the $\oplus'$-operation we have $\cutmFK(\kappa^{(i)\oplus' n},\iota^{(i)\oplus' n})<\xi$.
Since for every $x\in[0,1]$ the map $s\mapsto \kappa^{(i)\oplus' n}_{s,x}$ is constant, we therefore find that
\begin{align}\label{eqLemma_cutmcont2}
	\sum_{\tau\in\Omega^n}
	\Erw\abs{\Erw\brk{\prod_{j=1}^n\kappa^{(i)}_{\vs,\hat\vx_1,\ldots,\hat\vx_n}(\tau_j)-
			\prod_{j=1}^n\iota^{(i)}_{\vs,\hat\vx_1,\ldots,\hat\vx_n}(\tau_j)\,\big|\,\hat\vx_1,\ldots,\hat\vx_n}}&<\beta&\mbox{for all }i\leq k.
\end{align}
Because $\lambda(S_i)<\eta$ for all $i>k$ and $\ell\eta<\beta$ for small enough $\eta$,
\eqref{eqLemma_cutmcont2} implies that
\begin{align}\label{eqLemma_cutmcont3}
	\sum_{\tau\in\Omega^n}
	\Erw\abs{\Erw\brk{\prod_{j=1}^n\kappa_{\vs,\hat\vx_1,\ldots,\hat\vx_n}(\tau_j)-
			\prod_{j=1}^n\iota_{\vs,\hat\vx_1,\ldots,\hat\vx_n}(\tau_j)\,\big|\,\hat\vx_1,\ldots,\hat\vx_n}}&<2\beta.
\end{align}
Combining \eqref{eqLemma_cutmcont3} with Markov's inequality, we conclude that
\begin{align}\label{eqLemma_cutmcont4}
	\pr\brk{\cE}&>1-\beta^{1/3},&\mbox{where}\qquad
	\cE=\cbc{\sum_{\tau\in\Omega^n}\abs{\int_0^1\prod_{j=1}^n\kappa_{s,\hat\vx_1,\ldots,\hat\vx_n}(\tau_j)-
			\prod_{j=1}^n\iota_{s,\hat\vx_1,\ldots,\hat\vx_n}(\tau_j)\dd s}<\beta^{1/3}}.
\end{align}
Consequently, on $\cE$ we have
\begin{align}\label{eqLemma_cutmcont5}
	\sum_{\tau\in\Omega^n}|\vz_\kappa(\tau)-\vz_\iota(\tau)|<\beta^{1/3}.
\end{align}
In particular, there exists a coupling of the reference configurations $\hat\SIGMA^\kappa,\hat\SIGMA^\iota\in\Omega^n$ such that
$\pr\brk{\hat\SIGMA^\kappa=\hat\SIGMA^\iota}\geq1-\beta^{1/4}$.
Hence, \Lem~\ref{Lemma_pinz} implies that the event 
$\cE'=\cbc{\hat\SIGMA^\kappa=\hat\SIGMA^\iota,\,\vz_\kappa(\hat\SIGMA^\kappa)\geq\alpha}$ satisfies
\begin{align}\label{eqLemma_cutmcont5}
	\pr\brk{\cE'\mid\cE}&\geq1-|\Omega|^n\alpha.
\end{align}

To complete the proof let
\begin{align*}
	p_\kappa(s)&=\prod_{i=1}^n\kappa_{s,\hat\vx_i}(\hat\SIGMA_i^\kappa),&
	p_{\iota}(s)&=\prod_{i=1}^n\iota_{s,\hat\vx_i}(\hat\SIGMA_i^\iota)&&\mbox{ and}\\
	\hat p_{\kappa,i}&=\int_{S_i}p_\kappa(s)\dd s,&\hat p_{\iota,i}&=\int_{S_i}p_\iota(s)\dd s.
\end{align*}
Further, let $\cE''=\cbc{\sum_{i=1}^\ell\abs{\hat p_{\kappa,i}-\hat p_{\iota,i}}<\alpha^3}$.
Then \eqref{eqLemma_cutmcont2}, \eqref{eqLemma_cutmcont4} and \eqref{eqLemma_cutmcont5} imply that
\begin{align}\label{eqLemma_cutmcont6}
	\pr\brk{\cE''\mid\cE\cap\cE'}>1-\alpha.
\end{align}
Moreover, since $\vz_\kappa(\hat\SIGMA^\kappa)\geq\alpha$ and $\hat\SIGMA^\kappa=\hat\SIGMA^\iota$, on $\cE\cap\cE'\cap\cE''$ we have $\vz_\iota(\hat\SIGMA^\iota)\geq\alpha/2$.
Therefore, on $\cE\cap\cE'\cap\cE''$ the probability distributions
$(p_{\kappa,i})_{i\in[\ell]},(p_{\iota,i})_{i\in[\ell]}$ with
\begin{align*}
	p_{\kappa,i}&=\hat p_{\kappa,i}/\vz_\kappa(\hat\SIGMA^\kappa),&
	p_{\iota,i}&=\hat p_{\iota,i}/\vz_\iota(\hat\SIGMA^\iota)
\end{align*}
have total variation distance $\dTV((p_{\kappa,i})_{i\in[\ell]},(p_{\iota,i})_{i\in[\ell]})<2\alpha$.
Consequently, there exists a coupling of random variables $\vi_\kappa,\vi_\iota$ with these distributions such that
\begin{align}\label{eqLemma_cutmcont7}
	\pr\brk{\vi_\kappa\neq\vi_\iota\mid\cE\cap\cE'\cap\cE''}<2\alpha.
\end{align}

We extend this coupling to a coupling $\gamma$ of $\mu_{\hat\SIGMA^\mu\pin n},\nu_{\hat\SIGMA^\nu\pin n}$:
given $\vi_\kappa,\vi_\iota$, pick any $\vs_\kappa\in S_{\vi_\kappa}$ and choose
$\vs_\iota\in S_{\vi_\iota}$ from the distribution $p_\iota(s)/\hat p_{\iota,\vi_\iota}\dd s$.
Then $\kappa_{\vs_\kappa}$, $\iota_{\vs_\iota}$ have distribution $\mu_{\hat\SIGMA^\mu\pin n},\nu_{\hat\SIGMA^\nu\pin n}$, respectively.
Further, we claim that on $\cE\cap\cE'\cap\cE''$,
\begin{align}\label{eqLemma_cutmcont8}
	\abs{\int_B\int_X \sigma_x(\omega)-\tau_x(\omega)\dd x\dd\gamma(\sigma,\tau)}&<\eps&\mbox{for all }B\subset\conf\times\conf,X\subset[0,1],\omega\in\Omega.
\end{align}
Indeed, thanks to \eqref{eqLemma_cutmcont7}, we may condition on the event $\vi_\iota=\vi_\kappa\leq k$.
Hence, to prove \eqref{eqLemma_cutmcont8} it suffices to show that for any $S\subset S_{\vi_\iota}$, $X\subset[0,1]$, $\omega\in\Omega$,
\begin{align}\label{eqLemma_cutmcont9}
	\int_X\int_S \frac{p_\iota(s)}{\hat p_{\iota,\vi_\iota}}\bc{\iota_{s,x}(\omega)-\kappa_{\vs_\kappa,x}(\omega)}\dd s\dd x<\eps/|\Omega|.
\end{align}
Because $\vz_\iota(\hat\SIGMA^\kappa)\geq\alpha/2$ we may also assume that $\hat p_{\iota,\vi_\iota}\geq\alpha^2/\ell$, and we observe that $p_\iota(s)\leq1$.
Now, assume for contradiction that there exist $S,X,\omega$ for which \eqref{eqLemma_cutmcont9} is violated.
Letting
\begin{align*}
	S^+&=\cbc{s\in S:\int_X\bc{\iota_{s,x}(\omega)-\kappa_{\vs_\kappa,x}(\omega)}\dd x>\alpha},
\end{align*}
we conclude that
\begin{align}\nonumber
	\frac\eps{2|\Omega|}&\leq
	\int_X\int_{S^+}\frac{p_\iota(s)}{\hat p_{\iota,\vi_\iota}}\bc{\iota_{s,x}(\omega)-\kappa_{\vs_\kappa,x}(\omega)}\dd s\dd x
	=\int_{S^+}\frac{p_\iota(s)}{\hat p_{\iota,\vi_\iota}}\int_X\bc{\iota_{s,x}(\omega)-\kappa_{\vs_\kappa,x}(\omega)}\dd x\dd s\\
	&\leq\frac\ell{\alpha^2}\int_{S^+}\int_X\bc{\iota_{s,x}(\omega)-\kappa_{\vs_\kappa,x}(\omega)}\dd x\dd s
	\leq\ell\alpha^{-2}\cutmFK(\iota^{\vi_\iota},\kappa^{\vi_\iota})
	\leq\ell\alpha^{-2}\zeta\qquad\mbox{due to \eqref{eqLemma_cutmcont1}}.
	\label{eqLemma_cutmcont10}
\end{align}
But \eqref{eqLemma_cutmcont10} contradicts the choice of the parameters from~\eqref{eqGreeks}.
Hence, we obtain \eqref{eqLemma_cutmcont9} and thus \eqref{eqLemma_cutmcont8}.
Finally, the assertion follows from  \eqref{eqLemma_cutmcont4}, \eqref{eqLemma_cutmcont5}, \eqref{eqLemma_cutmcont6} and \eqref{eqLemma_cutmcont8}.
\end{proof}

\begin{lemma}\label{Lem_DistanceConvergenceL1}
For every sequence $(k_i)_i$ in $\cK_1$ that converges to a kernel $k \in \Kernel_{1}$ with respect to $\cutmFK(\cdot, \cdot)$ and for every kernel $k' \in \Kernel_{1}$ there is a sequence of kernels $(k'_i)_i$, $k_i'\in\Kernel_1$, s.t.\ $\cutmFK(k'_i, k') \to 0$ and $D_1(k_i, k_i') \to D_1(k, k')$.
\end{lemma}
\begin{proof}
Let $(\kappa^\omega)_\omega, (\kappa'^\omega)_\omega ,(\kappa_i^\omega)_\omega, (\kappa'^\omega_i)_\omega $ be the families of bipartite graphons representing $k, k', (k_i)_i, (k'_i)_i$ given by \eqref{eqkappaomega}. From the definition of $D_1(\cdot, \cdot)$ and \Lem~\ref{Lemma_bipartite} we get 
\begin{align}
	\cutmFK(k_i, k) = \frac{1}{2} \max_{\omega} \cutmFK(\kappa^\omega_i, \kappa^\omega) \qquad \text{and} \qquad D_1(k, k') = \frac{1}{2} \sum_{\omega} D_1(\kappa^\omega, \kappa'^\omega). \label{eqD1Cut}
\end{align} 
The lemma follows from \eqref{eqD1Cut} and \cite[Proposition 8.25]{Lovasz}. 
\end{proof}

\noindent\aco{The following lemma and the proof of \Prop~\ref{Prop_ContinuityOfPinning} are adaptations of \cite[Proof of Lemma~9.16]{Lovasz}.}

\begin{lemma}\label{Lemma_Uk}
Let $\eps,\delta>0$ and let $k\in\Kernel{}$.
Let $U_k(\delta,\eps)$ be the set of all $\kappa\in\Kernel$ such that there exists $\kappa'\in\Kernel_1$ with $\cutmFK(k,\kappa')<\delta$ and $D_1(\kappa',\kappa)<\eps$.
Then $U_k(\delta,\eps)$ is $\cutmFK$-open.
\end{lemma}
\begin{proof}
Suppose that $\kappa\in U_k(\delta,\eps)$ and that the sequence $(\kappa_i)_{i\geq1}$ satisfies $\lim_{i\to\infty}\cutmFK(\kappa,\kappa_i)=0$.
It suffices to show that $\kappa_i\in U_k(\delta,\eps)$ for all large enough $i$.
To this end consider $\kappa'\in\Kernel_1$ such that $\cutmFK(k,\kappa')<\delta$ and $D_1(\kappa',\kappa)<\eps$.
By \Lem~\ref{Lem_DistanceConvergenceL1} there exists a sequence $\kappa_i'\in\Kernel_1$ such that $\lim_{i\to\infty}\cutmFK(\kappa',\kappa_i')=0$ and $\lim_{i\to\infty}D_1(\kappa_i',\kappa_i)=D_1(\kappa,\kappa')$.
\aco{For this sequence we have
	\begin{align*}
		\cutmFK(\kappa_i',\kappa')&\to0,&D_1(\kappa_i,\kappa_i')&\to D_1(\kappa,\kappa')<\eps.
	\end{align*}
Therefore, for large enough $i$ we have $D_1(\kappa_i,\kappa_i')<\eps$ and $\cutmFK(k,\kappa_i')\leq\cutmFK(k,\kappa')+\cutmFK(\kappa_i',\kappa')<\delta$, whence $\kappa_i'\in U_k(\delta,\eps)$.}
\end{proof}

\begin{proof}[Proof of \Prop~\ref{Prop_ContinuityOfPinning}]
Fix $\eps>0$.
\Lem~\ref{Lemma_L1cont} shows that there exists $\delta_0>0$ such that for all $\kappa\in\Kernel,\kappa'\in\Kernel_1$,
\begin{align}\label{eqProp_ContinuityOfPinning1}
	D_1(\kappa,\kappa')<\delta_0\quad\Rightarrow\quad\cutmW(\mu^\kappa_{\hat\SIGMA^\kappa\pin n},\mu^{\kappa'}_{\hat\SIGMA^{\kappa'}\pin n})<\eps/2.
\end{align}
Similarly, by \Lem~\ref{Lemma_cutmcont} there exists a sequence $(\delta_\ell)_\ell$ such that for all $\mu,\nu\in\Law$ with $\mu$ supported on at most $\ell\geq1$ configurations we have
\begin{align}\label{eqProp_ContinuityOfPinning2}
	\cutm(\mu,\nu)<\delta_\ell\quad\Rightarrow\quad\cutmW(\mu_{\hat\SIGMA^\mu\pin n},\nu_{\hat\SIGMA^\mu\pin n})<\eps/2.
\end{align}

Suppose that $k:[0,1]^2\to\cP(\Omega)$ is a step function that takes $\ell\geq1$ different values
and let $\cU_k=\cU_k(\delta_\ell,\delta_0)$ be as in \Lem~\ref{Lemma_Uk}.
Then $\cU_k$ is $\cutmFK$-open \aco{and $\bigcup_k\cU_k=\Kernel$ because $\cU_k$ contains the $\delta_0$-ball around $k$ with respect to the $D_1$-metric.}
Further, let $\fU_k\subset\kernel$ be the projection of $\cU_k$ onto $\kernel{}$.
Then $\fU_k$ is open because the canonical map $\Kernel\to\kernel$ is open.
Moreover, $\bigcup_k\fU_k=\kernel$.
Hence, a finite number of sets $\fU_k$ cover $\kernel$.
Thus, the assertion follows from \eqref{eqProp_ContinuityOfPinning1} and \eqref{eqProp_ContinuityOfPinning2}.
\end{proof}

\subsection{Proof of \Thm~\ref{Thm_pinning}}\label{Sec_finish}
Let $\eps>0$ and pick a small enough $\delta>0$ and then a large enough $N>0$.
Also let $T=T(\eps)=64\eps^{-8}\log|\Omega|$.
Given $\mu\in\law$ we apply \Thm~\ref{Thm_sampling} to obtain a probability distribution $\nu\in \Law_N$ such that $\cutm(\mu,\dot\nu)<\delta$.
Invoking \Thm~\ref{Prop_Polish} and \Prop~\ref{Prop_ContinuityOfPinning}, we find
\begin{align}\label{eqThm_pinning_1}
	\cutm(\mu_{\hat\SIGMA^\mu\pin n},\dot\nu_{\hat\SIGMA^{\dot\nu}\pin n})<\eps/4&&\mbox{for all }n\leq T(\eps).
\end{align}
By construction, for any $n$ the law $\dot\nu_{\hat\SIGMA^{\dot\nu}\pin n}$ obtained by first embedding $\nu\in\Law_N$ into $\Law$ and then applying the pinning operation coincides with the law obtained by first applying \eqref{eqhatmu} to $\nu$ and then embedding the resulting $\hat\nu$ into $\law$.
Hence, \Thm~\ref{Thm_pin} and \Lem~\ref{Lemma_sym} show that for a uniform $\THETA\leq T(\eps)$,
\begin{align}\label{eqThm_pinning_2}
	\Erw[\cutm(\overline{\dot\nu_{\hat\SIGMA^{\dot\nu}\pin \THETA}},\dot\nu_{\hat\SIGMA^{\dot\nu}\pin \THETA})]&<\eps^2/2.
\end{align}
Further, \Thm~\ref{Prop_embedding}, \Thm~\ref{Prop_metric_ext} and \eqref{eqThm_pinning_1}
show that 
\begin{align}\nonumber
\cutm\bc{\overline{\mu_{\hat\SIGMA^\mu\pin \THETA}},\mu_{\hat\SIGMA^\mu\pin \THETA}}
&\leq \cutm\bc{\mu_{\hat\SIGMA^\mu\pin \THETA},\dot\nu_{\hat\SIGMA^{\dot\nu}\pin \THETA}}+
\cutm\bc{\overline{\dot\nu_{\hat\SIGMA^{\dot\nu}\pin \THETA}},\dot\nu_{\hat\SIGMA^{\dot\nu}\pin \THETA}}
+\cutm\bc{\overline{\dot\nu_{\hat\SIGMA^{\dot\nu}\pin \THETA}},\overline{\mu_{\hat\SIGMA\pin \THETA}}}\\
&\leq2\cutm\bc{\mu_{\hat\SIGMA\pin \THETA},\dot\nu_{\hat\SIGMA^{\dot\nu}\pin \THETA}}+\cutm\bc{\overline{\dot\nu_{\hat\SIGMA^{\dot\nu}\pin \THETA}},\dot\nu_{\hat\SIGMA^{\dot\nu}\pin \THETA}}<\eps+\cutm\bc{\overline{\dot\nu_{\hat\SIGMA^{\dot\nu}\pin \THETA}},\dot\nu_{\hat\SIGMA^{\dot\nu}\pin \THETA}}.
\label{eqThm_pinning_3}
\end{align}
Combining \eqref{eqThm_pinning_2} and \eqref{eqThm_pinning_3} and applying Markov's inequality, we obtain the first part of \Thm~\ref{Thm_pinning}.
The second assertion follows from a similar argument.

\subsection{Proof of \Thm~\ref{Prop_embedding}}\label{Sec_Prop_embedding}
We postponed the proof \Thm~\ref{Prop_embedding}, because it relies on some of the prior results from this section.
To finally carry the proof out we adapt the proof strategy from \cite{
Lovasz}, where a statement similar to \Thm~\ref{Prop_embedding} was established for graphons, to the present setting of probability distributions.
We begin with the following simple bound.

\begin{lemma}\label{Lem_Embeddings1}
For any $\mu, \nu \in \Law_n$ we have $\Cutm(\mu,\nu)\leq n^3\cutm(\dot\mu,\dot\nu).$
\end{lemma}
\begin{proof}
Let $\psi\in\SS$ and let $\gamma\in\Gamma(\dot\mu,\dot\nu)$.
We are going to show that there exist a coupling $g\in\Gamma(\mu,\nu)$ and a permutation $\phi\in\SS_n$ such that
\begin{align}\label{eqLem_Embeddings1_1}
	\max_{\substack{S\subset\Omega^n\times\Omega^n\\X\subset[n]\\\omega\in\Omega}}
	\abs{\sum_{(\sigma,\sigma')\in S,x\in X}g(\sigma,\sigma')\bc{\vecone\cbc{\sigma_x=\omega}-\vecone\cbc{\sigma'_{\phi(x)}=\omega}}}&\leq
	n^4\sup_{\substack{S\subset\conf\times\conf\\X\subset[0,1]\\\omega\in\Omega}}
	\abs{\int_S\int_X\bc{\sigma_{x}(\omega)-\sigma'_{\psi(x)}(\omega)}\dd x\dd s}.
\end{align}
The assertion is immediate from \eqref{eqLem_Embeddings1_1} and the definitions \eqref{eqdisc}, \eqref{eqcutm}.

With respect to the coupling $g$, matters are easy:
the construction of $\dot\mu,\dot\mu'\in\Law$ ensures that the coupling $\gamma$ readily induces a coupling $g$ of the original probability distributions $\mu,\nu$ such that $g(\sigma,\tau)=\gamma(\dot\sigma,\dot\tau)$ for all $\sigma,\tau\in\Omega^n$.

We are left to exhibit the permutation $\phi$.
To this end let $I_j=[(j-1)/n,j/n)$.
We construct a bipartite auxiliary graph $\cG$ with vertex set $\{v_1,\ldots,v_n\}\cup\{w_1,\ldots,w_n\}$ in which $v_i,w_j$ are adjacent iff $\lambda(I_j\cap\psi(I_i))\geq n^{-3}$.
Then the Hall's theorem implies that $\cG$ possesses a perfect matching.
Indeed, assume that $\emptyset\neq V\subset\{v_1,\ldots,v_n\}$ satisfies $|\partial V|<|V|$.
Then because $\psi$ preserves the Lebesgue measure we obtain
\begin{align*}
\aco{\frac{1}{n}\leq\frac{|V|-|\partial V|}{n}=\frac{|V|}{n}-\sum_{v_i\in V,\,w_j\in\partial V}\lambda(I_j\cap\psi(I_i))=\sum_{v_i\in V,\,w_j\not\in\partial V}\lambda(I_j\cap\psi(I_i))\leq\frac{|V|(n-|V|)}{n^3}<\frac{1}{n},}
\end{align*}
a contradiction.
Thus, let $\phi$ be the permutation of $[n]$ induced by any perfect matching of $\cG$.

\aco{
To complete the proof we claim that $g,\phi$ satisfy \eqref{eqLem_Embeddings1_1}.
Indeed, given a set $S\subset\Omega^n\times\Omega^n$ let $\dot S=\{(\dot\sigma,\dot\tau):(\sigma,\tau)\in S\}$.
Further, for $X\subset[n]$ let $\dot X\subset[0,1]$ be any measurable set such that $\lambda(\dot X\cap I_i)=n^{-3}$ for all $i\in[n]$ and $\dot X\cap I_j\subset\psi(I_i)$ if $\phi(i)=j$.
Then we obtain \eqref{eqLem_Embeddings1_1}.}
\end{proof}

As a second step we will complement the coarse multiplicative bound from \Lem~\ref{Lem_Embeddings1} with a somewhat more subtle additive bound.
To this end, we need an enhanced version {of} a 'Frieze-Kannan type' regularity lemma for probability distributions.
Specifically, let $\mu\in\Law_n$ and let $S=\{S_1,\ldots,S_k\}$ and $X=\{X_1,\ldots,X_\ell\}$ be partitions of $\Omega^n$ and $[n]$, respectively.
We call the partition $S$ {\em canonical} if there exists a set $\cI\subset[n]$ such that
\begin{align*}
S=\cbc{\cbc{\sigma\in\Omega^n:\forall i\in\cI:\sigma_i=\tau_i}:\tau\in\Omega^\cI}.
\end{align*}
In words, $S$ partitions the discrete cube $\Omega^n$ into the $\Omega^{|\cI|}$ sub-cubes defined by the entries on the set $\cI$ of coordinates.
In this case we define
\begin{align*}
\mu^{S,X}(\sigma)&=\sum_{h=1}^k
\mu(S_h)\prod_{i=1}^\ell\prod_{j\in X_i}\sum_{x\in X_i}\frac{\mu_x(\sigma_j|S_h)}{|X_i|}
\in\Law_n.
\end{align*}
Thus, $\mu^{S,X}$ is a mixture of product measures, one for each class of the partition $S$.

\begin{lemma}\label{Cor_pin}
For any $\Omega$ there exists $c=c(\Omega)>0$ such that for every $0<\eps<1/2$, $n>0$ and all $\mu,\nu\in\Law_n$ there exist a canonical partition $S_1,\ldots,S_k$ of $\Omega^n$ and a partition $X_1,\ldots,X_\ell$ of $[n]$ such that the following statements are satisfied.
\begin{itemize}
	\item $k+\ell\leq \exp(\eps^{-c})$.
	\item with $\gamma\in\Gamma(\mu,\mu^{S,X})$ and $\gamma'\in\Gamma(\nu,\nu^{S,X})$ defined by
	\begin{align*}
		\gamma(\sigma,\tau)&=\sum_{h=1}^k\vecone\cbc{\sigma,\tau\in S_h}\mu(\sigma)\mu^{S,X}(\tau)/\mu(S_h),\\
		\gamma'(\sigma,\tau)&=\sum_{h=1}^k\vecone\cbc{\sigma,\tau\in S_h}\nu(\sigma)\nu^{S,X}(\tau)/\nu(S_h)
	\end{align*}
	we have
	\begin{align}
		\max_{S\subset\Omega^n\times\Omega^n,\,X\subset[n],\omega\in\Omega}
		\abs{\sum_{(\sigma,\tau)\in S}\sum_{x\in X}\gamma(\sigma,\tau)
			\bc{\vecone\cbc{\sigma_x=\omega}-\vecone\cbc{\tau_x=\omega}}}&<\eps n,\label{eqCor_pin1}\\
		\max_{S\subset\Omega^n\times\Omega^n,\,X\subset[n],\omega\in\Omega}
		\abs{\sum_{(\sigma,\tau)\in S}\sum_{x\in X}\gamma'(\sigma,\tau)
			\bc{\vecone\cbc{\sigma_x=\omega}-\vecone\cbc{\tau_x=\omega}}}&<\eps n.
		\label{eqCor_pin2}
	\end{align}
	Hence, $\Cutm(\mu,\mu^{S,X})<\eps$, $\Cutm(\nu,\nu^{S,X})<\eps$.
\end{itemize}
\end{lemma}
\begin{proof}
Combining \Thm~\ref{Thm_pin} and \Lem~\ref{Lemma_sym}, we find a set $\cI\subset[n]$ such that the induced canonical partition $S_1,\ldots,S_k$ satisfies
\begin{align}\label{eqCor_pin_1}
	\sum_{i=1}^k\mu(S_i)\cutmw\bc{\mu[\nix|S_i],\bigotimes_{x=1}^n\mu_x[\nix|S_i]}
	&<\eps/8,&
	\sum_{i=1}^k\nu(S_i)\cutmw\bc{\nu[\nix|S_i],\bigotimes_{x=1}^n\nu_x[\nix|S_i]}
	&<\eps/8.
\end{align}
Moreover, the size $k$ of the partition is bounded by $\exp(\eps^{-c'})$ for some $c'=c'(\Omega)$.
Now, for each $i\in[k]$ we can partition the set $[n]$ into at most $32/\eps$ classes $X_{i,1},\ldots,X_{i,\ell_i}$ 
such that for all $x,y\in X_{i,j}$ we have $\dTV(\mu_x[\nix|S_i],\mu_y[\nix|S_i])<\eps/16$.
A similar partition $X_{i,1}',\ldots,X'_{i,\ell_i'}$ exists for $\nu[\nix|S_i]$.
Hence, the smallest common refinement $X_1,\ldots,X_\ell$ of all these partitions $(X_{i,j}),(X_{i,j}')$ has at most 
$\exp(\eps^{-c})/2$ classes, for some suitable $c=c(\Omega)>0$.
Further, by construction, letting
\begin{align*}
	\mu^{(i)}(\sigma)&=\prod_{j=1}^\ell\prod_{x\in X_j}\frac{1}{|X_j|}\sum_{x\in X_j}\mu_x(\sigma_x|S_i),&
	\nu^{(i)}(\sigma)&=\prod_{j=1}^\ell\prod_{x\in X_j}\frac{1}{|X_j|}\sum_{x\in X_j}\nu_x(\sigma_x|S_i),
\end{align*}
we obtain from  \eqref{eqCor_pin_1} that
\begin{align}\label{eqCor_pin_2}
	\sum_{i=1}^k\mu(S_i)\cutmw\bc{\mu[\nix|S_i],\mu^{(i)}}
	&<\eps/4,&
	\sum_{i=1}^k\nu(S_i)\cutmw\bc{\nu[\nix|S_i],\nu^{(i)}}
	&<\eps/4.
\end{align}
In addition, since $\mu^{(i)},\nu^{(i)}$ are product measures, the couplings $\gamma^{(i)},\gamma^{(i)\prime}$ for which the cut distance in \eqref{eqCor_pin_2} attained are trivial, i.e., 
$\gamma^{(i)}=\mu[\nix|S_i]\tensor \mu^{(i)}$ and $\gamma^{(i)\prime}=\nu[\nix|S_i]\tensor \nu^{(i)}$.
Therefore, \eqref{eqCor_pin_2} implies \eqref{eqCor_pin1}--\eqref{eqCor_pin2}.
\end{proof}

\begin{lemma}\label{Lem_Embeddings2}
For any $\mu, \nu \in \Law_n$ we have
$\Cutm(\mu,\nu)\leq\cutm(\dot\mu,\dot\nu)+o(1)$ as $n\to\infty$.
\end{lemma}
\begin{proof}
Let $0<\eps=\eps(n)=o(1)$ be a sequence that tends to zero sufficiently slowly.
By \Cor~\ref{Cor_pin} there exist partitions $S_1,\ldots,S_k$ of $\Omega^n$ and $X_1,\ldots,X_\ell$ of $[n]$ such that $\Cutm(\mu,\mu^{S,X})+\Cutm(\nu,\nu^{S,X})<\eps$ and $k+\ell\leq\exp(\eps^{-c})$.
By the triangle inequality,
\begin{align*}
	\cutm(\dot\mu^{S,X},\dot\nu^{S,X})&\leq\cutm(\dot\mu,\dot\nu)+\cutm(\dot\mu,\dot\mu^{S,X})+\cutm(\dot\nu,\dot\nu^{S,X})\\
	&\leq \cutm(\dot\mu,\dot\nu)+\Cutm(\mu,\mu^{S,X})+\Cutm(\nu,\nu^{S,X})
	\leq \cutm(\dot\mu,\dot\nu)+2\eps.
\end{align*}
\aco{Hence, there exist $\phi\in\SS$ and a coupling $g$ of $\mu^{S,X},\nu^{S,X}$ such that the induced coupling $\dot g$ of $\dot\mu^{S,X},\dot\nu^{S,X}$ satisfies}
\begin{align}\label{eqLem_Embeddings2_1}
	\sup_{T\subset\conf\times\conf,\,Y\subset[0,1],\,\omega\in\Omega}
	\abs{\int_T\int_Y\bc{\sigma_{y}(\omega)-\tau_{\phi(y)}(\omega)}\dd y\dd g(\sigma,\tau)}
	&<\cutm(\mu,\nu)+3\eps.
\end{align}
Because $\phi$ {preserves} the Lebesgue measure, there exists a bijection $\varphi:[n]\to[n]$ such that the following is true.
For a class $X_i\subset[n]$ let $\dot X_i=\bigcup_{x\in X_i}[(x-1)/n,x/n)$.
Then uniformly for all $h,i\in[\ell]$ we have
\begin{align}\label{eqLem_Embeddings2_2}
	\abs{X_h\cap \varphi(X_i)}=n\lambda(\dot X_h\cap\phi(\dot X_i))+O(1).
\end{align}
Further, we construct a coupling $G\in\Gamma(\mu,\nu)$ by letting
\begin{align*}
	G(\sigma,\tau)&=\sum_{\substack{\sigma'\in\Omega^n:\mu^{S,X}(\sigma')>0\\\tau'\in\Omega^n:\nu^{S,X}(\tau')>0}}\frac{\gamma(\sigma,\sigma')g(\sigma',\tau')\gamma'(\tau,\tau')}{\mu^{S,X}(\sigma')\nu^{S,X}(\tau')}
\end{align*}
and we claim that
\begin{align}\label{eqLem_Embeddings2_4}
	\frac1n&\max_{\substack{T\subset\Omega^n\times\Omega^n\\Y\subset[n]\\\omega\in\Omega}}
	\abs{\sum_{(\sigma,\tau)\in T}G(\sigma,\tau)(\sigma_Y(\omega)-\tau_{\varphi(Y)}(\omega))}<\cutm(\mu,\nu)+6\eps,
	\ \mbox{where }\ 
	\sigma_Y(\omega)=\sum_{y\in Y}\vecone\cbc{\sigma_y=\omega}.
\end{align}
Clearly, \eqref{eqLem_Embeddings2_4} readily implies the assertion.

To verify \eqref{eqLem_Embeddings2_4} we observe that, due to symmetry and the triangle inequality, it suffices to show that
\begin{align}\label{eqLem_Embeddings2_5}
	\abs{\sum_{(\sigma,\tau)\in T}\sum_{\sigma',\tau'}
		\frac{\gamma(\sigma,\sigma')g(\sigma',\tau')\gamma'(\tau,\tau')}{\mu^{S,X}(\sigma')\nu^{S,X}(\tau')}
		\bc{\sigma_Y(\omega)-\sigma'_Y(\omega)}}&<\eps n,\\
	\abs{\sum_{(\sigma,\tau)\in T}\sum_{\sigma',\tau'}
		\frac{\gamma(\sigma,\sigma')g(\sigma',\tau')\gamma'(\tau,\tau')}{\mu^{S,X}(\sigma')\nu^{S,X}(\tau')}
		\bc{\sigma_Y'(\omega)-\tau'_{\varphi(Y)}(\omega)}}&<\cutm(\mu,\nu)n+4\eps n.
	\label{eqLem_Embeddings2_6}
\end{align}
for all $T,Y,\omega$.
Now, invoking \Lem~\ref{Cor_pin}, we obtain
\begin{align*}
	\sum_{(\sigma,\tau)\in T}\sum_{\sigma',\tau'}
	\frac{\gamma(\sigma,\sigma')g(\sigma',\tau')\gamma'(\tau,\tau')}{\mu^{S,X}(\sigma')\nu^{S,X}(\tau')}
	\bc{\sigma_Y(\omega)-\sigma'_Y(\omega)}_+
	\leq\sum_{\sigma,\sigma'}\gamma(\sigma,\sigma')\bc{\sigma_Y(\omega)-\sigma'_Y(\omega)}_+
	<\eps n.
\end{align*}
As the same bound holds for the negative part $\bc{\sigma_Y(\omega)-\sigma'_Y(\omega)}_-$,  we obtain \eqref{eqLem_Embeddings2_5}.
Similarly, due to \Cor~\ref{Cor_pin}, \eqref{eqLem_Embeddings2_1} and \eqref{eqLem_Embeddings2_2},
\begin{align*}
	\sum_{(\sigma,\tau)\in T}\sum_{\sigma',\tau'}
	\frac{\gamma(\sigma,\sigma')g(\sigma',\tau')\gamma'(\tau,\tau')}{\mu^{S,X}(\sigma')\nu^{S,X}(\tau')}
	\bc{\sigma_Y'(\omega)-\tau'_{\varphi(Y)}(\omega)}_+
	&\leq\sum_{\sigma',\tau'}g(\sigma',\tau')\bc{\sigma_Y'(\omega)-\tau'_{\varphi(Y)}(\omega)}_+\\
	&< n\cutm(\mu,\nu)+3\eps n+O(k\ell)\leq n\cutm(\mu,\nu)+4\eps n,
\end{align*}
whence \eqref{eqLem_Embeddings2_6} follows.
\end{proof}

\begin{proof}[Proof of \Thm~\ref{Prop_embedding}]
The theorem follows by combining \Lem s~\ref{Lem_Embeddings1} and \ref{Lem_Embeddings2}.
\end{proof}

\subsection*{Acknowledgment.} We thank Viresh Patel for bringing \cite{TaoRegular} to our attention, an anonymous reviewer for their careful reading, which has led to numerous corrections, and a second anonymous reviewer for pointing out several further references.


\begin{thebibliography}{29}

\bibitem{Aldous}
D.\ Aldous: Representations for partially exchangeable arrays of random variables. J.\ Multivariate Anal.\ {\bf11} (1981) 581--598.

\bibitem{AlonFrenandezKannanKarpinski}
N.\ Alon, W.\ Fernandez de la Vega, R.\ Kannan, M.\ Karpinski: Random sampling and approximation of MAX-CSPs. J.\ Comput.\ System Sci.\ {\bf 67} (2003) 212--243

\bibitem{Austin_exchange}
T.~Austin: On exchangeable random variables and the statistics of large graphs and hypergraphs. Probab. Surveys {\bf 5} (2008) 80--145. 

\bibitem{Austin_ExchangeableRandomMeasures}
T.~Austin: Exchangeable random measures.
Annales de l'institut Henri Poincar\'{e}, Probabilit\'{e}s et Statistiques {\bf 51} (2015) 842--861.

\bibitem{Victor}
V.\ Bapst, A.\ Coja-Oghlan: Harnessing the Bethe free energy.
Random Structures and Algorithms {\bf 49} (2016) 694--741.

\bibitem{BCLSV1}
C.\ Borgs, J.\ Chayes, L.\ \Lovasz, V.\ S\'os, K.\ Vesztergombi:
Convergent sequences of dense graphs I: subgraph frequencies, metric properties and testing.
Adv.\ Math.\ {\bf219} (2008), 1801--1851.

\bibitem{BCLSV2}
C.\ Borgs, J.\ Chayes, L.\ \Lovasz, V.\ S\'os, K.\ Vesztergombi:
Convergent sequences of dense graphs II: multiway cuts and statistical physics.\
Ann.\ Math.\ {\bf 176} (2012) 151--219.

\bibitem{borgs_graphexes}
C.\ Borgs, J.\ Chayes, H.\ Cohn, N. Holden: Sparse exchangeable graphs and their limits via graphon processes. Journal of Machine Learning Research. {\bf 18} (2018) 1 -- 71. 

\bibitem{borgs_lp}
C.\ Borgs, J.\ Chayes, H.\ Cohn, Y. Zhao: An Lp theory of sparse graph convergence II: LD convergence, quotients and right convergence. Ann. Probab. {\bf 46 } (2018) 337--396.

\bibitem{Buehler}
T.\ B\"uhler: Functional Analysis.
American Mathematical Society (2018).

\bibitem{AcoPinningPaper}
A.\ Coja-Oghlan, F. \ Krzakala, W.\ Perkins, L.\ Zdeborov\'a: Information-theoretic thresholds from the cavity method.  Advances in Mathematics {\bf333} (2018) 694--795.

\bibitem{AcoPerkins}
A.\ Coja-Oghlan, W.\ Perkins: Spin systems on Bethe lattices. Communications in Mathematical Physics {\bf 372} (2018) 441--523.

\bibitem{AcoPerkinsSymmetry}
A.\ Coja-Oghlan, W.\ Perkins: Bethe states of random factor graphs.
Communications in Mathematical Physics {\bf366} (2019) 173--201.

\bibitem{AcoPerkinsSkubch}
A.\ Coja-Oghlan, W.\ Perkins, K.\ Skubch: Limits of discrete distributions and Gibbs measures on random graphs. European Journal of Combinatorics {\bf 66} (2017) 37--59.

\bibitem{cai_edge_exchange} 
D. Cai, T. Campbell, T. Broderick: Edge-exchangeable graphs and sparsity. Advances in Neural Information Processing Systems {\bf 29} (2016) 4249--4257.

\bibitem{Razborov_DenseObjects}
L. Coregliano, A. Razborov: Semantic Limits of Dense Combinatorial Objects. Uspekhi Matematicheskikh Nauk {\bf 75} (2020)  45--152.

\bibitem{crane_edge_exchange} 
H.\ Crane, W.\ Dempsey: Edge Exchangeable Models for Interaction Networks, Journal of the American Statistical Association, {\bf 113:523} 1311-1326 (2018).

\bibitem{ConlonFox}
D.\ Conlon, J.\ Fox: Bounds for graph regularity and removal lemmas.
Geometric and Functional Analysis {\bf 22} (2012) 1191--1256.

\bibitem{JansonDiaconis_GraphLimits}
P.\ Diaconis, S.\ Janson: Graph limits and exchangeable random graphs. Rend.\ Mat.\ Appl.\ {\bf 28} (2008) 33--61.

\bibitem{Eldad}
R.\ Eldan: Taming correlations through entropy-efficient measure decompositions with applications to mean-field approximation. arXiv:1811.11530 (2018).

\bibitem{FriezeKannan}
A.\ Frieze, R.\ Kannan: Quick approximation to matrices and applications. Combinatoria {\bf 19} (1999) 175--220.

\bibitem{Andreas}
A.\ Galanis, D.\ Stefankovic, E.\ Vigoda:
Inapproximability for antiferromagnetic spin systems in the tree nonuniqueness region.
J.\ ACM {\bf 62} (2015) 50

\bibitem{Georgii}
H.-O.\ Georgii: Gibbs measures and phase transitions. 2nd edition. De Gruyter (2011).

\bibitem{Hartig} 
D. G. Hartig: The Riesz representation theorem revisited.
American Mathematical Monthly {\bf 90} (1983) 277--280.

\bibitem{Hoover}
D.\ Hoover: Relations on probability spaces and arrays of random variables. Preprint, Institute of Advanced Studies, Princeton, 1979.

\bibitem{hoppen_permutons}
C.\ Hoppen, Y.\ Kohayakawa, C.\ Moreira, B.\ Rath, R.\ Sampaio: Limits of permutation sequences. Journal of Combinatorial Theory Series B. {\bf 103} (2011)  10.1016/j.jctb.2012.09.003. 

\bibitem{Janson_posetons}
S.~Janson: Poset limits and and exchangeable random posets. Combinatorica {\bf 31} 529--563 (2011).

\bibitem{Janson_CutMetric}
S.~Janson: Graphons, cut norm and distance, couplings and rearrangements. 
NYJM Monographs {\bf 4} (2013).


\bibitem{Kallenberg}
O.\ Kallenberg: Probabilistic symmetries and invariance principles. Springer, New York, 2005.

\bibitem{Kechris}
A. S.\ Kechris: Classical descriptive set theory. Springer (1995).


\bibitem{pnas}
F.~Krzakala, A.~Montanari, F.~Ricci-Tersenghi, G.~Semerjian, L.~Zdeborov\'a:
Gibbs states and the set of solutions of random constraint satisfaction problems.
Proc.~National Academy of Sciences {\bf104} (2007) 10318--10323.

\bibitem{Lovasz}
L.\ Lov\'asz:  Large Networks and Graph Limits. American Mathematical Society~2012. 

\bibitem{LovaszSzegedy_LimitsDecorated}
L. \Lovasz, B. Szegedy: Limits of compact decorated graphs.
ArXiV 1010.5155 (2010).

\bibitem{BS1}
L.\ \Lovasz, B.\ Szegedy: Limits of dense graph sequences.
J.\ Combin.\ Theory Ser.\ B {\bf 96} (2006) 933--957.

\bibitem{BS2}
L.\ \Lovasz, B.\ Szegedy: \Szemeredi's lemma for the analyst.
Geom.\ Funct.\ Anal.\ {\bf17} (2007) 252--270.

\bibitem{LovaszSzegedy}
L.\ Lov\'asz, B.\ Szegedy: Regularity partitions and the topology of graphons. 
In: I.\ B\'ar\'any, J.\ Solymosi, G.\ S\'agi:  An Irregular Mind. 
Bolyai Society Mathematical Studies \textbf{21} (2010).

\bibitem{Mackey}
G. W. Mackey: Borel structure in groups and their duals.
Trans. Amer. Math. Soc. {\bf 85} (1957) 134--165.

\bibitem{MPRTRLZ}
E.\ Marinari, G.\ Parisi, F.\ Ricci-Tersenghi, J.\ Ruiz-Lorenzo, F.\ Zuliani:
Replica symmetry breaking in short-range spin glasses: theoretical foundations and numerical evidences.
J.\ Stat.\ Phys.\ {\bf98} (2000) 973

\bibitem{MM}
M.~M\'ezard, A.~Montanari:
Information, physics and computation.
Oxford University Press~2009.

\bibitem{Montanari}
A.\ Montanari: Estimating random variables from random sparse observations. European Transactions on Telecommunications {\bf19} (2008) 385--403.

\bibitem{nesetril_modelings}
J.\ Ne\v{s}et\v{r}il, P.\ Ossona de Mendez: Existence of modeling limits for sequences of sparse structures. The Journal of Symbolic Logic {\bf 84} (2019) 452--472. 

\bibitem{Nicolay}
S.\ Nicolay, L.\ Simons: Building Cantor's Bijection. arXiv 1409.1755 (2014).

\bibitem{Panchenko_SherringtonKirkpatrick}
D.\ Panchenko: The Sherrington-Kirkpatrick Model.
Springer Monographs in Mathematics (2013).

\bibitem{Panchenko}
D.\ Panchenko:
Spin glass models from the point of view of spin distributions.
Annals of Probability {\bf 41}  (2013) 1315--1361.

\bibitem{Raghavendra}
P.\ Raghavendra, N.\ Tan: Approximating CSPs with global cardinality constraints using SDP hierarchies. Proc.\ 23rd SODA (2012) 373--387.

\bibitem{SlyUniqueness}
A.\ Sly: Computational transition at the uniqueness threshold.
Proc.\ 51st FOCS (2010) 287--296.

\bibitem{SS}
A.~Sly, N.~Sun: The computational hardness of counting in two-spin models on d-regular graphs.  Proc.\ 53rd FOCS (2012) 361--369.

\bibitem{SzemerediRegLemma}
E. ~Szemer\'edi: On sets of integers containing no $k$ elements in arithmetic progression. Acta Arithmetica {\bf{27}} (1975) 199--245.

\bibitem{TaoRegular}
T.~Tao: Szemeredi's regularity lemma via the correspondence principle. Blog entry.
https://terrytao.wordpress.com/2009/05/08/szemeredis-regularity-lemma-via-the-correspondence-principle/

\bibitem{veitch_graphexes}
V.~Veitch, D.~Roy: The Class of Random Graphs Arising from Exchangeable Random Measures. arXiv 1512.03099 (2015).

\bibitem{Villani}
C.~Villani: Optimal Transport. Springer (2009).


\end{thebibliography}
\end{document}